\documentclass[a4paper,12pt,reqno]{amsart}
\usepackage{vmargin}
\usepackage[pdftex]{graphicx}
\usepackage{amsfonts}
\usepackage{mathrsfs}
\usepackage{stmaryrd}
\usepackage{amsmath}
\usepackage{amssymb}
\usepackage{amsthm}
\usepackage[shortlabels]{enumitem}
\usepackage{nomencl}
\usepackage[bookmarks=true]{hyperref}   
\usepackage{esint}
\usepackage{dsfont}
\usepackage{upgreek}
\usepackage{xcolor}
\usepackage{slashed}
\usepackage{scalerel}
\usepackage[utf8]{inputenc}

\newcommand{\Bk}{\color{black}}

\newcommand{\Rd}{\color{red}}
\newcommand{\Bl}{\color{blue}}

\renewcommand{\Bk}{}
\renewcommand{\Rd}{}
\renewcommand{\Bl}{}

\renewcommand{\imath}{{\rm i}}	

\makeatletter
\renewcommand\section{\@startsection{section}{1}{\z@}%
                       {-3\p@ \@plus -4\p@ \@minus -4\p@}%
                       {3\p@ \@plus 4\p@ \@minus 4\p@}%
                      {\normalfont\normalsize\centering\scshape}}
\makeatother

\numberwithin{equation}{section} 	

\hypersetup{
    colorlinks,%
}

\newcommand{\Rosen}{Ros\'en}

\author{Lashi Bandara}
\author{Alan McIntosh}
\author{Andreas \Rosen}
\title[Riesz continuity of the Dirac operator under metric perturbations]{Riesz continuity of the Atiyah-Singer Dirac
operator under perturbations of the metric}
\date{\today}
\address{Lashi Bandara, 
Institut für Mathematik,
Universität Potsdam, 
D-14476, Potsdam OT Golm, Germany
}
\urladdr{\href{http://www.math.uni-potsdam.de/~bandara}{http://www.math.uni-potsdam.de/~bandara}}
\email{\href{mailto:lashi.bandara@uni-potsdam.de}{lashi.bandara@uni-potsdam.de}}

\address{Alan McIntosh, Mathematical Sciences Institute,
Australian National University, Canberra, ACT, 2601, Australia }
\urladdr{\href{http://maths.anu.edu.au/~alan}{http://maths.anu.edu.au/~alan}}
\email{\href{mailto:alan.mcintosh@anu.edu.au}{alan.mcintosh@anu.edu.au}}

\address{Andreas \Rosen, Mathematical Sciences,
Chalmers University of Technology and University of Gothenburg, SE-412 96, Gothenburg, Sweden}
\urladdr{\href{http://www.math.chalmers.se/~rosenan}{http://www.math.chalmers.se/~rosenan}}
\email{\href{mailto:andreas.rosen@chalmers.se}{andreas.rosen@chalmers.se}}


\def\colour{\colour}

\def\colour{\color}

\newtheorem{theorem}{Theorem}[section]

\newtheorem{lemma}[theorem]{Lemma}
\newtheorem{proposition}[theorem]{Proposition}

\newtheorem{definition}[theorem]{Definition}
\newtheorem{remark}[theorem]{Remark}

\newcommand{\floor}[1]{\lfloor #1 \rfloor} 

\newcommand{\mdot}{\cdotp}

\newcommand{\cbrac}[1]{\left(#1\right)}
\newcommand{\bbrac}[1]{\left[#1\right]}
\newcommand{\dbrac}[1]{\left\{#1\right\}}
\newcommand{\modulus}[1]{\left|#1\right|}
\newcommand{\set}[1]{\dbrac{#1}}

\newcommand{\dom}{ {\mathcal{D}}}
\newcommand{\ran}{ {\mathcal{R}}}

\newcommand{\comp}{\, \circ\, }

\newcommand{\e}{\mathrm{e}}

\newcommand{\R}{\mathbb{R}}
\newcommand{\C}{\mathbb{C}}
 
\newcommand{\In}{\mathbb{Z}}

\newcommand{\Na}{\ensuremath{\mathbb{N}}}

\newcommand{\script}[1]{\mathscr{#1}}

\renewcommand{\emptyset}{\varnothing}
\newcommand{\union}{\cup}

\newcommand{\intersect}{\cap}

\newcommand{\rest}[1]{{{\lvert_{}}_{}}_{#1}}
\newcommand{\close}[1]{\overline{#1}}		

\newcommand{\id}{{\rm id}}
\renewcommand{\epsilon}{\varepsilon}

\renewcommand{\phi}{\varphi}

\newcommand{\ch}[1]{\upchi_{{#1}}}	

\newcommand{\embed}{\hookrightarrow}		

\newcommand{\tensor}{\otimes}

\newcommand{\comm}[1]{\bbrac{#1}}		

\newcommand{\norm}[1]{\| #1 \|}			
\newcommand{\snorm}[1]{\left\| #1 \right\|}			

\newcommand{\spt}[1]{{\rm spt} {\text{ }}#1}	

\newcommand{\conform}{{\upomega}}		

\DeclareMathOperator{\tr}{tr}			
\DeclareMathOperator{\diam}{diam}		

\DeclareMathOperator{\len}{\ell}			
\DeclareMathOperator{\rad}{rad}			

\DeclareMathOperator{\divv}{div}		

\newcommand{\ddt}[1]{\frac{d}{d#1}}		

\newcommand{\Ld}[1]{[#1]}			

\newcommand{\Ric}{{\rm Ric}}			
\newcommand{\Rs}{\mathcal{R}_S}			

\DeclareMathOperator{\inj}{inj} 		

\newcommand{\Forms}[1][{}]{\mathbf{\Omega}^{#1}}		
\newcommand{\Tensors}[1][{}]{{\mathcal{T}}^{(#1)}}	
\newcommand{\Sect}{\mathbf{\Gamma}}		
\newcommand{\tanb}{{\rm T}}		
\newcommand{\cotanb}{{\rm T}^\ast}	

\newcommand{\pushf}[1]{#1_\ast}			
\newcommand{\pullb}[1]{#1^\ast}			

\DeclareFontFamily{OT1}{restrictfont}{}
\DeclareFontShape{OT1}{restrictfont}{m}{n}{<-> fmvr8x}{}

\newcommand{\adj}[1]{{#1}^\ast}			
\newcommand{\biadj}[1]{{#1}^{\ast\ast}}		
\newcommand{\extd}{{\rm d}}			

\newcommand{\inprod}[1]{\left\langle #1 \right\rangle}	
\DeclareMathOperator{\grad}{grad}			
\newcommand{\conn}[1][{}]{{\nabla_{{#1}}}}		
\DeclareMathOperator{\Lipp}{Lip}		
\newcommand{\Leb}[1][{}]{\script{L}^{#1}}			

\newcommand{\Cliff}[1][{}]{\Delta^{#1}}		
\DeclareMathOperator{\cliff}{{\scaleobj{0.5}{\triangle}\, }}	
\DeclareMathOperator{\Spin}{Spin}			
\DeclareMathOperator{\Spinors}{\slashed{\Delta}}	
\newcommand{\spin}[1]{\slashed{#1}}		
\DeclareMathOperator{\SO}{SO}			
\newcommand{\Prin}[1]{\mathrm{P}_{#1}}		

\newcommand{\bddlf}{\mathcal{L}} 	
\newcommand{\rs}[1]{\textsf{R}_{#1}}			
\DeclareMathOperator{\nrad}{nrad}			

\newcommand{\Lp}[2][{}]{{\rm L}^{#2}_{\rm #1}}		
\newcommand{\Ck}[2][{}]{{\rm C}^{#2}_{\rm #1}}		
\newcommand{\Sob}[2][{}]{{\rm W}^{#2}_{\rm #1}}		
\newcommand{\SobH}[2][{}]{\Sob[#1]{#2,2}}
\newcommand{\Lips}[1][{}]{{\rm Lip}_{\rm #1}}		

\newcommand{\Sec}[1]{\mathrm{S}_{#1}}
\newcommand{\OSec}[1]{\mathrm{S}^\mathrm{o}_{#1}}

\newcommand{\maxx}[1]{\left<#1\right>}

\newcommand{\iden}{{\mathrm{I}}}
\DeclareMathOperator{\sgn}{sgn}

\newcommand{\Hil}{\script{H}}			
\newcommand{\Lap}{\Delta}			

\newcommand{\Q}[1][{}]{Q_{#1}}			
\newcommand{\DyQ}{\script{Q}}			
\newcommand{\ancester}[1]{\widehat{#1}}		

\newcommand{\brad}{\rho} 			
\newcommand{\scale}{\mathrm{t_S}}			
\newcommand{\jscale}{\mathrm{J}}
\newcommand{\hscale}{\mathrm{t_H}}			

\newcommand{\sC}{\script{C}}

\newcommand{\cV}{\mathcal{V}}
\newcommand{\cM}{\mathcal{M}} 
\newcommand{\cN}{\mathcal{N}}
\newcommand{\cW}{\mathcal{W}}

\newcommand{\Hol}{{\rm Hol}}

\newcommand{\mg}{\mathrm{g}}
\newcommand{\mgt}{\tilde{\mg}}
\newcommand{\mh}{\mathrm{h}}

\newcommand{\Poincare}{Poincar\'e~}		

\newcommand{\Av}{\mathbb{E}}
\newcommand{\Pri}{\upgamma}

\newcommand{\Carl}{\mathcal{C}}

\newcommand{\CBox}{\mathrm{R}}

\newcommand{\B}{\mathrm{B}}

\newcommand{\Dir}{{\rm D} }
\newcommand{\Dirb}{\Dir}		
\newcommand{\Dirp}{\tilde{\Dir}}

\newcommand{\Qqb}{{\rm Q}}
\newcommand{\Ppb}{{\rm P}}
\newcommand{\Qq}{\tilde{\Qqb}}
\newcommand{\Pp}{\tilde{\Ppb}}
\newcommand{\QQ}{{\rm\bf Q}}
\newcommand{\PP}{{\rm\bf P}}
\newcommand{\Rrb}{{\rm R}}
\newcommand{\Rr}{\tilde{\Rrb}}

\newcommand{\dtt}[1][{t}]{\frac{d#1}{#1}}

\newcommand{\const}{\mathrm{C}} 
\newcommand{\met}{\uprho}		

\newcommand{\Ball}{\mathrm{B}}

\newcommand{\sinprod}[1]{\inprod{#1}_{\ast}}

\newcommand{\SDir}{\slashed{\Dir}}
\newcommand{\Sconn}[1][{}]{{{\nabla}_{{#1}}}}		
\newcommand{\SDirp}{\tilde{\SDir}}
\newcommand{\U}{\mathrm{U}}
\newcommand{\rep}{\cdot}

\newcommand{\RNum}[1]{\uppercase\expandafter{\romannumeral #1\relax}}

\begin{document}

\maketitle
\vspace*{-2em}
\begin{abstract}
We prove that the Atiyah-Singer Dirac operator
$\SDir_\mg$ in $\Lp{2}$ depends Riesz continuously
on $\Lp{\infty}$ perturbations of complete metrics $\mg$ on a
smooth manifold.
The Lipschitz bound for the map $\mg \to \SDir_\mg(1 + \SDir_\mg^2)^{-\frac{1}{2}}$
depends on bounds on Ricci curvature and its first derivatives
as well as a lower bound on injectivity radius.
Our proof uses harmonic analysis techniques related to
Calder\'on's first commutator and the
Kato square root problem. We also show perturbation results
for more general functions of general Dirac-type operators
on vector bundles.
\end{abstract}
\tableofcontents
\vspace*{-2em}

\parindent0cm
\setlength{\parskip}{\baselineskip}

\section{Introduction}

In this paper we prove perturbation estimates for self-adjoint first-order partial differential operators 
$\Dir$ and $\Dirp$ of 
Dirac type, elliptic with domains $\Sob{1,2}(\cM, \cV)$ in 
$\Lp{2}(\cM,\cV)$, on vector bundles $\cV$ over complete Riemannian 
manifolds $(\cM,\mg)$.
A typical quantity to bound is
\begin{equation}  \label{eq:rieszpert}
  \snorm{ \frac{\Dirp}{\sqrt{I+\Dirp^2}} - \frac \Dir{\sqrt{\iden+\Dir^2}}}_{\Lp{2}(\cM, \cV)\to \Lp{2}(\cM, \cV)}.
\end{equation}
Our motivating and main example is when $\Dir = \SDir$ is the Atiyah--Singer Dirac operator $\SDir$ on $\cM$, acting on sections 
of a given spin bundle $\cV = \Spinors\cM$ over $(\cM,\mg)$. 
The perturbations $\Dirp$  we consider arise from the 
pullback of the Atiyah--Singer operator on a nearby manifold $(\cN,\mh)$. More precisely, we have a diffeomorphism $\zeta: \cM\to \cN$ 
which induce a map $\spin{\U}: \Spinors\cM \to \Spinors\cN$ between the two spinor bundles, 
and we set  $\Dirp:= \spin{\U}^{-1} \SDir_{\cN} \spin{\U}$   on  $\cM$. 
For the construction of the induced spinor pullback, we follow  \cite{Bourguignon}  by Bourguignon and Gauduchon  and build this 
from the isometric factor  of the polar factorisation of the differential of $\zeta$.

The perturbation \eqref{eq:rieszpert} is for the symbol $f(\lambda)= \lambda/\sqrt{1+\lambda^2}$ in 
the functional calculi of the operators $\Dir$ and $\Dirp$. This will yield continuity 
results in the Riesz metric given by \eqref{eq:rieszpert} for unbounded self-adjoint operators.
\Rd However, our method of proof  applies \Bk equally well to any other symbol $f(\lambda)$ which is 
holomorphic and bounded on the neighbourhood
$\OSec{\omega,\sigma}:=\set{x+iy: y^2 <\tan^2\omega x^2 +\sigma^2 }$
of $\R$ for some $0<\omega<\pi/2$ and $\sigma>0$.
\Rd
Our Riesz continuity result is non-trivial 
as it entails cutting through the spectrum at infinity
with the added complication that 
the symbol has different limits at infinity ($\lim_{\lambda\to \pm \infty} f(\lambda) = \pm 1$). \Bk
This should be compared to the weaker continuity result for the graph metric 
$$
    \snorm{\frac{\Dirp-\imath}{\Dirp+\imath} - \frac {\Dir-\imath}{\Dir+\imath}}_{\Lp{2}(\cM, \cV )\to \Lp{2}(\cM, \cV)}
$$
for unbounded self-adjoint operators, which is simpler since the symbol 
$g(\lambda)=(\lambda-\imath)/(\lambda+\imath)$ is holomorphic at $\infty$.

\Bl
The Riesz and graph topologies are of great importance in the study of self-adjoint
unbounded operators because of their connection to the 
\emph{spectral flow}. Loosely speaking, this  is the net
number of eigenvalues crossing zero along a curve 
from the unit interval to the set of self-adjoint operators.
The study of the spectral flow was initiated by Atiyah and Singer in \cite{AS69}
since it has important connections to particle physics. Their focus, however, was
on bounded Fredholm self-adjoint operators and their point of view 
was largely topological. An analytic formulation 
of the spectral flow also exists due to Phillips in \cite{P96}.

In the bounded case, the choice of topology for the study 
of the spectral flow is canonically given by the norm topology. 
However, in order to study differential operators, 
the unbounded case needs to be considered. Here,  a choice of topology 
needs to be made and the graph metric is most
commonly used in the study of the spectral flow, primarily since it is easier 
to establish continuity in this topology. However, 
the Riesz topology is a preferred alternative since it better connects to topological 
and $K$-theoretic aspects of the spectral flow that were
observed in \cite{AS69} for bounded operators.
Further details of the relation between different metrics on 
the set of unbounded  self-adjoint operators can be found in \cite{Lesch} by Lesch.
Moreover, the survey paper \cite{BB} by Booß-Bavnbek provides a recent account of problems remaining in field of spectral flow.

Since in this paper we establish results 
in the Riesz topology, of particular relevance is Proposition 2.2 in \cite{Lesch}
where it is proved that  \Bk 
$$
  \snorm{ \frac{\Dirp}{\sqrt{\iden+\Dirp^2}} - \frac \Dirb{\sqrt{\iden+\Dirb^2}}}_{\Lp{2}(\cM, \cV)\to \Lp{2}(\cM, \cV)}
	\lesssim \|\Dirp-\Dirb\|_{\Sob{1,2}(\cM, \cV)\to \Lp{2}(\cM, \cV)}
$$
\Rd holds for small perturbations  $\Dirp$ of $\Dir$ 
with both operators self-adjoint and with domain $\Sob{1,2}(\cM, \cV)$. \Bk
We achieve a non-trivial strengthening of this estimate for Dirac-type differential operators, using techniques from harmonic analysis.
The structure of the perturbation that we consider is
\begin{equation}
\label{Eq:Struc}
  \Dirp-\Dirb= A_1\nabla + \divv A_2 +A_3,
\end{equation} 
where $A_1$, $A_2$  and $A_3$ are bounded multiplication operators $\cotanb\cM\otimes \cV\to \cV$, $\cV\to \cotanb\cM\otimes \cV$ and 
$\cV\to \cV$ respectively.
Typically in applications, and in particular for the Atiyah-Singer Dirac operator, one can achieve
$$
  \norm{A_i}_\infty \lesssim \norm{\mgt-\mg}_\infty,
$$
where $\mg$ is the metric on $\cM$ and $\mgt = \pullb{\zeta}\mh$ is the metric on $\cN$ pulled back to $\cM$. 
\Rd In order to conclude small Riesz distance between $\Dirb$ and $\Dirp$ using
the aforementioned  Proposition 2.2 in \cite{Lesch}, one would need not only
smallness of $\norm{A_i}_\infty$ but also smallness of $\norm{\conn^\mg A_2}$.
Via our methods, we are able to dispense this
requirement and only require the finiteness of $\norm{\conn^\mg A_2}$. \Bk

\Rd 
In Theorem \ref{Thm:Main}, which is our main result, \Bk  we prove the perturbation estimate
\begin{equation}  
\label{eq:mainpertest}
  \norm{f(\Dirp)- f(\Dirb)} \lesssim  \max_i \norm{A_i}_{\infty}  \norm{f}_{\Hol^\infty(\OSec{\omega,\sigma})}, 
\end{equation}
where the implicit constant depends on the geometry of $\cV\to \cM$ and the operators $\Dirb$ and $\Dirp$ as described 
in the hypothesis \ref{Hyp:First}-\ref{Hyp:Last} preceding Theorem \ref{Thm:Main}.
\Rd In Theorem \ref{Thm:MainApp}, we specialise 
Theorem \ref{Thm:Main} to the case  where the operators  $\Dir$ and $\Dirp$
are the Atiyah-Singer Dirac operators as previously discussed. Here, 
 the implicit constant depends \Bk 
roughly on the $\Ck{0,1}$ norm of $\mgt$ and $\Ck{2}$ norm of $\mg$. 
Injectivity radius bounds coupled
with bounds on Ricci curvature and its first derivatives 
allow us to obtain uniformly sized balls corresponding to harmonic
coordinates at every point.
Moreover, we obtain uniform $\Ck{2}$ control of the metric $\mg$ in each such chart. 
Therefore our result, unlike Proposition 2.2 in \cite{Lesch}, will apply to metric perturbations with 
$\mgt - \mg$ small only in $\Lp{\infty}$ norm,
\Rd under uniform $\Ck{2}$ control of $\mg$ and
uniform $\Ck{0,1}$ control of $\mgt$ 
in each such chart. A concrete example of such metrics
are $\mg = \iden$ and $\mgt(x) = (1 + \epsilon\sin(\modulus{x}/\epsilon))\iden$
on $\R^n$. \Bk

The main work in establishing \eqref{eq:mainpertest} is to prove quadratic estimates of the form
\begin{equation}  
\label{Eq:MainSFE}
\int_0^1 \snorm{\frac{t\Dirp}{\iden +t^2\Dirp^2} B \frac{\iden}{\iden+t^2 \Dirb^2}u}_{\Lp{2}(\cM,\cV)}^2 \frac{dt}{t}
	\lesssim \norm{B}_{\Lp{\infty}(\cM,\cV)}^2 \norm{u}_{\Lp{2}(\cM,\cV)}^2,
\end{equation}
where $B$ is a bounded operator, a multiplication operator, or special kind of a singular integral. 
The use of such quadratic estimates to bound functional calculi goes
back to the work of Coifman, McIntosh and Meyer  \cite{CMcMA, CMcM} on 
Calder\'on's problem on the boundedness of the Cauchy integral of Lipschitz curves. 
The quadratic estimates that we require in this paper are at the level of those needed to bound Calder\'on's first commutator. 
An additional technical difficulty for us in the present work is that $B$ also may involve a certain singular integral operator.
To overcome this problem, we need a Riesz-Weitzenb\"ock condition stated as hypothesis \ref{Hyp:Weitz}.

The starting point for our work in this paper, was a twin result for the Hodge--Dirac operator $\extd+\adj{\extd}$ proved by the last two named 
authors jointly with Keith in \cite{AKMC}. There it was proved, in the case of compact manifolds, that 
\begin{equation}
\label{Eq:HD}  
\snorm{\sgn\cbrac{\frac{\extd_{\mgt} + \adj{\extd}_{\mgt}}{\sqrt{ 1 + (\extd_{\mgt} + \adj{\extd}_{\mgt})^2}}}
	-\sgn\cbrac{\frac{\extd_{\mg} + \adj{\extd}_{\mg}}{\sqrt{ 1 + (\extd_{\mg} + \adj{\extd}_{\mg})^2}}}}
	\lesssim \norm{f}_{\Hol^\infty(\OSec{\omega,\sigma})} \norm{\mgt -\mg}_{\infty},
\end{equation}

This made use not only of the methods from \cite{CMcM} described above, but also of stopping time arguments for Carleson measures from the solution
of the Kato square root problem by Auscher, Hofmann, Lacey, McIntosh and Tchamitchian~\cite{AHLMcT}.
These techniques give results for perturbations when the domains of the Hodge-Dirac operators change, that is when no Lipschitz
control of the metric is assumed, and even give holomorphic dependence of $\sgn(\extd_{\mg} + \adj{\extd}_{\mg})$ 
on the metric $\mg$ and not only Lipschitz dependence.
However, there are also reasons to prefer the softer methods used in this paper and to avoid the stopping time arguments. Namely,
even though they make the implicit constant in \eqref{Eq:HD} independent of any Lipschitz control of the metrics, 
this constant in applications  may  become too large for the estimate to be useful. 
Our plan is to return to the perturbation problem for the Hodge-Dirac operator in a forthcoming paper.

\Rd
As aforementioned, since the Riesz topology is one of the most important operator topologies for 
unbounded self-adjoint operators, it is a natural question how much of the above estimates 
hold for more general Dirac type operators, and in particular the most fundamental Atiyah--Singer Dirac operator.
For these operators we no longer have access to Hodge splittings, 
and it is not even clear that the Dirac operators exist as 
closed and densely-defined operators for rough metrics 
(measurable coefficients but locally bounded below).
Therefore the perturbation estimates that we achieve in this paper, with the constant 
depending on the Lipschitz norm of the metrics, may be quite sharp. 
We do not even know however if it is possible to go beyond Lipschitz metrics for 
Dirac operators like the Atiyah-Singer one.
In any case, as Lesch rightly points out in \cite{Lesch}, it is more 
difficult to 
prove Riesz continuity as compared to other operator topologies
and therefore, our results should have interesting applications to the study of 
spectral flow and to index theory of Dirac operators. 
Moreover, given the generality of Theorem \ref{Thm:Main}, 
we anticipate that these applications will go beyond the fundamental case 
of the Atiyah-Singer Dirac operator that we consider as an application here.\Bk

The outline of this paper is as follows. Our main perturbation theorem, Theorem \ref{Thm:Main} 
for general Dirac-type operators is formulated in \S\ref{Sec:MainPert}.
Before stating it, we discuss the geometric and operator
theoretic assumptions and we list quantities that the implicit
constant in the estimate \eqref{eq:mainpertest}
depends on as hypotheses \ref{Hyp:First}-\ref{Hyp:Last}.

For the proof of Theorem \ref{Thm:Main}, the reader may jump 
directly to \S\ref{Sec:Red} and \S\ref{Sec:SFE}. Independent
of this, we first devote \S\ref{Sec:Dirac} to prove
Theorem \ref{Thm:MainApp}, which is an application of 
Theorem \ref{Thm:Main} to the Atiyah-Singer Dirac operator. 
For the sake of concreteness, we only consider
this Dirac operator obtained from the
\Rd standard spin representation of dimension $2^{\lfloor{\frac{n}{2}}\rfloor}$ \Bk
and a given Spin structure. But we
expect Theorem \ref{Thm:MainApp} to hold for
more general Dirac-type operators on Dirac-bundles
under similar geometric assumptions.
The proof of Theorem \ref{Thm:MainApp} amounts
to verifying \ref{Hyp:First}-\ref{Hyp:Last}
and the perturbation structure \eqref{Eq:Struc}. A key observation
regarding the latter is the following exploited in \S\ref{Sec:Struc}.
A perturbation term $A_3$ of the form $A_3 = \partial B$ (with $\partial$
denoting a partial derivative) with $\norm{B}_\infty$
small, but with $\norm{\partial B}_\infty$ only bounded,
can be handed as terms $A_1 \partial + \partial A_2$,
with $B = A_2 = -A_1$, since by the product rule,
$(\partial B)f = \partial(B f) - B(\partial f)$.

The proof of Theorem \ref{Thm:Main} in \S\ref{Sec:Red} and \S\ref{Sec:SFE}, 
brought together in \S\ref{Sec:SFEUnited},
contains the following steps. Using the functional calculus
of $\Dir$ and $\Dirp$, the estimate of $\norm{f(\Dir) - f(\Dirp)}$
is reduced to the quadratic estimate 
\eqref{Eq:MainSFE} in Proposition \ref{Prop:FirstRed}
and \ref{Prop:FinalRed}.
This quadratic estimate is obtained in three
steps described by the formula
\eqref{Eq:SFEBreak}, following a well known 
harmonic analysis technique used in the solution of the
Kato square root problem with its origins
from R. Coifman and Y. Meyer. For us,
the last term $\Pri_t \Av_t Sf$ is not 
the main one, since
the needed Carleson measure estimate follows
directly from the self-adjointness of $\Dirp$, 
as shown in \S\ref{Sec:SFECarl}. The main term in 
\eqref{Eq:SFEBreak} is rather the first, which localises the operator 
$\QQ_t$, which is local on scale $t$, to the multiplication
operator $\Pri_t$. Our problem here is
the presence of $S = \conn (\imath \iden + \Dir)^{-1}$,
which is   essentially  a singular integral operator.
To handle the non-local operator $S$ in 
Proposition \ref{Prop:PrinPart}, 
we require some smoothness of $\Dir$, guaranteed
by the Riesz-Weitzenb\"ock condition \eqref{Def:Weitz}.
In \cite{Bunke}, Bunke  obtains 
such an estimate, but with assumptions on the Riemannian
curvature tensor in place of the Ricci curvature.
Our proof here is inspired by the improvements
that Hebey presents using harmonic coordinate
charts under the presence of positive injectivity 
radius and bounds on Ricci curvature to 
prove density theorems for Sobolev spaces of functions 
on noncompact manifolds in \cite{Hebey}. 
\section*{Acknowledgements}
The first author was supported by the Knut and Alice Wallenberg foundation, KAW 2013.0322 
postdoctoral program in Mathematics for researchers from outside Sweden.
The second author appreciates the support of the Mathematical Sciences Institute at The Australian National University, 
and also the support of Chalmers University of Technology and University of Gothenburg during his visits to Gothenburg. 
Further he acknowledges support from the Australian Research Council.
The third author was supported by Grant 621-2011-3744 from the Swedish Research Council, VR.

\Rd 
We would like to thank Alan Carey and Krzysztof P. Wojciechowski for discussions about the 
relevance of our approach to open questions involving spectral flow for paths of unbounded self-adjoint operators. 
Unfortunately Wojciechowski  became ill and died before these promising early discussions could be developed.
\Bk 
\section{Setup and the statement of the main theorem}
\label{Sec:Setup}

\subsection{Notation}
Throughout this paper, we assume Einstein 
summation convention and use the analysts inequality
$a \lesssim b$ to mean that $a \leq C b$, where
$C>0$, and equivalence $a \simeq b$.
 The characteristic function on a set
$E$ will be denoted by $\ch{E}$. 
Throughout, we will identify vectorfields
and derivations. That is, for a function 
$f$ differentiable at $x$ and a vectorfield $X$
at $x$, we write $Xf$ to denote $\extd f(X) = \partial_{X} f$.
 Often, $X = e_i$, where $\set{e_i}$ is a
basis vector field inside a local frame. 
The support of a function (or section) $f$
is denoted by $\spt f$.
Whenever we write $\Ck{k,\alpha}$,
we do not assume $\Ck{k,\alpha}$
with global control of the norm
but rather, only $\Ck{k,\alpha}$
regularity locally.

\subsection{Manifolds and vector bundles}
\label{Sec:GBG}

Let $\cM$ be a smooth, connected manifold 
and $\mg$ be a metric on $\cM$ that is
at least $\Ck{0,1}$ (locally Lipschitz).
 By $\met$ denote the distance metric induced by $\mg$
and  by $\mu$ the induced volume measure. 

Throughout this paper, 
we assume that $(\cM,\mg)$ is
complete, by which we mean that
$(\cM,\met)$ is a complete  metric space.
By $\Ball(x,r)$ or $\Ball_r(x)$, 
we denote a $\met$-metric open ball 
of radius $r >0$ centred at $x \in \cM$.
For an arbitrary ball $\Ball$, we denote
its radius  by $\rad(\Ball)$.
We recall that by the Hopf-Rinow theorem, 
the condition of completeness is equivalent
to the fact that $\close{\Ball(x,r)}$ is compact
for any $x \in \cM$ and $r < \infty$.

By $\cV$, we denote a smooth complex vector
bundle of dimension $\dim \cV = N$ 
over $\cM$ with a metric 
$\mh$ that is at least $\Ck{0,1}$.
We let  $\pi_{\cV}:\cV \to \cM$ 
be the bundle projection map.
We define the 
space of $\mu$-measurable sections of
$\cV$ by $\Sect(\cV)$. 
Using the Riemannian measure $\mu$
and the bundle metric $\mh$, we define
the standard $\Lp{p}$ spaces 
which we denote by $\Lp{p}(\cV)$.

Let us now assume that $\conn$ is a
connection on $\cV$, compatible with $\mh$ 
\Rd almost-everywhere. \Bk By 
$\conn_2: \dom(\conn_2) \to \Lp{2}(\cotanb\cM \tensor \cV)$, 
denote the operator $\conn$  with  domain  
$$\dom(\conn_2) = \set{u \in \Ck{\infty} \intersect \Lp{2}(\cV): \conn u \in \Lp{2}(\cotanb\cM \tensor \cV)}.$$
Then, $\conn_2$ is densely-defined
and closable, and we define the
Sobolev space $\SobH{1}(\cV) = \dom(\close{\conn_2})$, 
with norm 
$\norm{u}_{\SobH{1}}^2  = \norm{\close{\conn_2}u}^2 + \norm{u}^2.$  
Moreover, recall that the
divergence operator is
then $\divv = -\adj{\close{\conn_2}}$.
It is well known that
$\Ck[c]{\infty}(\cV)$ is dense
in $\SobH{1}(\cV)$ and
when $\mg$ is smooth, 
that $\Ck[c]{\infty}(\cotanb\cM \tensor \cV)$
is dense in $\dom(\divv)$
(see \cite{B1}).
In what is to follow, 
we will sometimes write 
$\conn$ in place of $\close{\conn_2}$.

We shall require the following concept
of growth of the measure $\mu$ in later analysis.

\begin{definition}[Exponential volume growth]
\label{Def:Exp}
We say that $(\cM,\mg,\mu)$ has exponential
volume growth if there exists $c_E \geq 1$, $\kappa, c > 0$
such that 
\begin{equation*}
\tag{$\mathrm{\text{E}_{\text{loc}}}$}
\label{Def:Eloc}
0 < \mu(\Ball(x,tr)) \leq ct^\kappa \e^{c_E tr} \mu(\Ball(x,r)) < \infty, 
\end{equation*}
for every $t \geq 1,\ r > 0$ and $x\in \cM$.
\end{definition} 

We shall also require the following property.
\begin{definition}[Local \Poincare inequality]
\label{Def:Poin}
\index{Local \Poincare inequality}
We say that $\cM$ satisfies a \emph{local \Poincare inequality}
if there exists $c_P \geq 1$ such that
for all $f \in \SobH{1}(\cM)$, 
\begin{equation*}
\Rd
\snorm{f - \cbrac{\fint_{B} f\ d\mu_\mg}}_{\Lp{2}(B)} \leq c_P\ \rad(B) \norm{f}_{\SobH{1}(B)} \Bk
\tag{$\text{P}_{\text{loc}}$}
\label{Def:Ploc}
\end{equation*}
for all balls $B$ in $\cM$ such that $\rad(B) \leq 1$.
\end{definition}

This growth assumption as well as the 
local \Poincare inequality are very natural,
i.e., if the Ricci curvature $ \Ric_{\mg}$
of a smooth $\mg$ satisfies $\Ric_{\mg} \geq \eta \mg$
for some $\eta \in \R$, then by the
Bishop-Gromov comparison theorem (c.f. Chapter 9 in \cite{Petersen}),
\eqref{Def:Eloc}  and \eqref{Def:Ploc}  are both satisfied.

As for the vector bundle $\cV$,  we require
the following uniformly local Euclidean structure,
referred to as \emph{generalised bounded geometry} 
or \emph{GBG} following terminology from \cite{BMc}. 

\begin{definition}[Generalised Bounded Geometry]
\label{Def:GBG}
We say that $(\cM,\mh)$ satisfies
\emph{generalised bounded geometry}, or \emph{GBG} for short,
if there exist $\brad > 0$ and
$C\geq 1 $ such that, for each $x \in \cM$,
there exists a continuous local trivialisation 
$\psi_x: \Ball(x,\brad) \times \C^N \to \pi^{-1}_\cV(\Ball(x,\brad))$
satisfying
$$
C^{-1} \modulus{\psi_x^{-1}(y)u}_\delta \leq \modulus{u}_{\mh(y)} \leq C \modulus{\psi_x^{-1}(y)u}_\delta,$$
for all $y \in \Ball(x,\brad)$,
where $\delta$ denotes the usual inner product
in $\C^N$ and $\psi_x^{-1}(y) u = \psi_x^{-1}(y,u)$ is the pullback of the
vector $u \in \cV_y$ to $\C^N$ via the local trivialisation $\psi_x$
at $y \in \Ball(x,\brad)$. We call $\brad$ the \emph{GBG
radius}.  
\end{definition}

We remark that, unlike in \cite{BMc},
we do not ask for the trivialisations
to be smooth.
A trivialisation satisfying the above condition 
is said to be a \emph{GBG chart} and a set of trivialisations 
$\set{\psi_x: x \in \cM}$ a \emph{GBG atlas}. For each GBG chart $\psi_x$, the associated
\emph{GBG frame} is then 
$$\set{e^i(y) = \psi_x(y, \hat{e}^i): \set{\hat{e}^i}\ \text{standard basis for}\ \C^N}.$$
If these trivialisations have higher regularity, i.e.
the trivialisations are $\Ck{k,\alpha}$ for some $k \geq 0$ and 
$\alpha \in (0,1)$, then we refer to this aforementioned
terminology as a $\Ck{k,\alpha}$ GBG chart/atlas/frame respectively.

Like exponential growth, generalised bounded geometry 
is a geometrically natural condition. In the case 
that the metric $\mg$ is smooth and complete, 
under the assumption 
$\inj(\cM,\mg) \geq \kappa > 0$ and $\Ric_\mg \geq \eta \mg$
for some $\kappa > 0$ and $\eta \in \R$, the bundle
of $(p,q)$-tensors satisfies GBG.
See Theorem 1.2 in \cite{Hebey} 
and Corollary 6.5 in \cite{BMc}.

\subsection{Functional calculus}
\label{Sec:FunC}

In  this  section, we introduce some notions
from operator theory and functional calculi that 
will be of relevance in subsequent sections.

Let $\Hil$ be a Hilbert space, 
and $T: \dom(T) \subset \Hil \to \Hil$
a self-adjoint operator.  Indeed, 
by the spectral theorem (see \cite{Kato}, Chapter 6, \S5),
for every Borel function $b: \R \to \R$, 
we can define and estimate the operator $b(T)$.
However, we shall only consider symbols $b$
which are holomorphic on a neighbourhood
of $\R$, in which case $b(T)$ is obtained by the Riesz-Dunford
functional calculus as we now explain.

For $\omega \in (0, \pi/2)$ and $\sigma \in (0, \infty)$,
define
$$\OSec{\omega,\sigma}:=\set{x+iy: y^2<\tan^2\omega x^2 + \sigma^2},$$
We say that a function $\psi \in \Psi(\OSec{\omega,\sigma})$
if it is holomorphic on $\OSec{\omega,\sigma}$ and
there exists an $\alpha > 0$, $C > 0$ such that
$$ \modulus{\psi(\zeta)} \leq \frac{C}{\modulus{\zeta}^\alpha}.$$
Letting the curve $\gamma$ denote
$\set{y^2 = \tan^2 (\omega/2) x^2/2 + \sigma^2/2}$, 
oriented counter-clockwise inside $\OSec{\omega, \sigma}$,
the Riesz-Dunford functional calculus is
\begin{equation} 
\label{Eq:RD}
\psi(T)u = \frac{1}{2\pi\imath} \oint_{\gamma} \psi(\zeta)\rs{T}(\zeta)u\ d\zeta,
\end{equation}
for each $u \in \Hil$,
with $\rs{T}(\zeta) = (\zeta\iden - T)^{-1}$
and where the integral converges absolutely as Riemann-sums.

We say that a holomorphic function
$f \in \Hol^\infty(\OSec{\omega,\sigma})$ if 
there exists $C > 0$ such that
$\norm{f(\zeta)}_\infty \leq C$.
For such a function, there exists a uniformly
bounded $\psi_n \in \Psi(\OSec{\omega,\sigma})$
such that $\psi_n \to f$ pointwise, and the
functional calculus is defined as
$$ f(T)u = \lim_{n \to \infty} \psi_n(T)u,$$
for $u \in \Hil$, which converges due to the fact that $T$ is self-adjoint,
and is independent of the sequence $\psi_n$.

These details are obtained as a 
special case of the functional
calculus for the so-called $\omega$-bisectorial 
operators. A detailed exposition can be found in 
\cite{Morris2} by Morris 
and for $\omega$-sectorial operators 
in \cite{ADMc} by Albrecht, Duong, and McIntosh and \cite{Haase} by Haase. 

\subsection{The main theorem}
\label{S:DiffOp}
\label{Sec:MainPert}

We assume that
the manifold $(\cM,\mg)$ is complete
and that both $\mg$ and $\mh$ 
are at least $\Ck{0,1}$.

Let $\Dirb$ be a first-order differential
operator on $\Ck{\infty}(\cV)$. By this, we
mean that, inside each frame $\set{e^i}$
for $\cV$ and $\set{v_j}$ for $\tanb \cM$
near $x$,  there exist coefficients $\alpha^{jk}_l$ 
and terms $\omega^p_q$ (not
necessarily smooth) such that
\begin{equation}
\label{Def:DirForm}
\Dirb u = (\alpha^{jk}_l \conn[v_j] u_k + u_i \omega_l^i)\ e^l,
\end{equation}
where $u = u_i e^i \in \Ck{\infty}(\cV)$. 

We consider two essentially self-adjoint
first-order differential operators $\Dirb$ and
$\Dirp$, and with slight abuse of notation 
we use this notation for their self-adjoint
extensions.

In establishing our main perturbation estimate
from \Rd  $\Dirb$ 
to $\tilde{\Dirb}$  \Bk on $\cV \to \cM$, we will make the following hypotheses:
\begin{enumerate}[({A}1)]
\item
\label{Hyp:First}
\label{Hyp:BasicF}
\label{Hyp:FinDim}
$\cM$ and $\cV$ are finite dimensional, 
quantified by $\dim \cM < \infty$
and $\dim \cV < \infty$,

\item 
\label{Hyp:ExpVol}
$(\cM,\mg)$ has exponential volume growth as defined
in Definition \ref{Def:Exp}, quantified by 
$c < \infty$, $c_E < \infty$ and $\kappa < \infty$ in 
\eqref{Def:Eloc},

\item 
\label{Hyp:Poin}
A local \Poincare inequality \eqref{Def:Ploc}
holds on $\cM$ as in Definition \ref{Def:Poin}
quantified by $c_P < \infty$,
\label{Hyp:BasicL}

\item 
\label{Hyp:TanGBG}
$\cotanb \cM$ has $\Ck{0,1}$ GBG frames 
$\nu_j$ quantified by $\brad_{\cotanb\cM} > 0$ and
$C_{\cotanb\cM} < \infty$ in Definition 
\ref{Def:GBG}, with regularity
$ \modulus{\conn \nu_j} < C_{G,\cotanb\cM}$
with $C_{G, \cotanb \cM} < \infty$ almost-everywhere,

\item 
\label{Hyp:VecGBG}
\label{Hyp:BasGBGL}
$\cV$ has $\Ck{0,1}$ GBG frames
$e_j$ quantified by $\brad_{\cV} > 0$
and $C_{\cV} < \infty$ in Definition 
\ref{Def:GBG}, with regularity
$\modulus{\conn e_j} < C_{G,\cV}$
with $C_{G,\cV} < \infty$ almost-everywhere,

\item
\label{Hyp:PDO}
$\Dirb$ is a first-order PDO with $\Lp{\infty}$
coefficients. In particular, $ \comm{ \Dirb, \eta}$ is a 
pointwise multiplication operator
on almost-every fibre $\cV_x$, and
there exists $c_{\Dirb} > 0$ such that 
\begin{equation}
\label{Def:CommConst}
\modulus{ \comm{ \Dirb, \eta} u(x)} \leq c_{\Dirb} \Lipp \eta(x) \modulus{u(x)}
\end{equation}
for almost-every $x \in \cM$, every 
bounded Lipschitz function $\eta$, and
where $\Lipp\eta(x)$ is the \emph{pointwise Lipschitz
constant}. 

\item
\label{Hyp:DirGBG}
$\Dirb$ satisfies
$\modulus{\Dirb e_j} \leq C_{D,\cV}$
with $C_{D,\cV} <\infty$
almost-everywhere inside each GBG
frame $\set{e_j}$,

\item
\label{Hyp:Dom}
$\Dirb$ and $\Dirp$ both have 
domains $\Sob{1,2}(\cV)$
with $\const \geq 1$  the smallest 
constants satisfying
\begin{equation}
\label{Def:DomConst}
\const^{-1} \norm{u}_{\Dirb} \leq \norm{u}_{\SobH{1}} \leq \const \norm{u}_{\Dirb}
\quad \text{and}\quad
\const^{-1} \norm{u}_{\Dirp} \leq \norm{u}_{\SobH{1}} \leq \const \norm{u}_{\Dirp}.
\end{equation}

\item 
\label{Hyp:Weitz}
\label{Hyp:Last}
$\Dirb$ satisfies the Riesz-Weitzenb\"ock  condition
\begin{equation}
\label{Def:Weitz}
\norm{\conn^2 u} \leq c_{W} (\norm{\Dir^2 u} + \norm{u})
\end{equation}
with $c_W < \infty$.
\end{enumerate}

The implicit constants in our perturbation
estimates will be allowed to depend on 
\begin{multline}
\label{Def:MainConst}
\const(\cM,\cV, \Dirb, \Dirp)
	= \max \{\dim \cM, \dim\cV,  c ,  c_E, \kappa,
	c_P,  \brad_{\cotanb\cM} , C_{\cotanb \cM}, C_{G,\cotanb\cM},
\\
 {\brad_{\cV}} , C_{\cV},
C_{G, \cV}, c_{\Dirb}, C_{\Dirb, \cV}, \const,  c_W \} < \infty.
\end{multline}

In section \S\ref{Sec:Red} and \S\ref{Sec:SFE}, we
prove the following theorem.

\begin{theorem}
\label{Thm:Main}
Let $(\cM, \mg)$ be a smooth Riemannian manifold
with $\mg$ that is $\Ck{0,1}$, complete, 
and satisfying \eqref{Def:Eloc} and \eqref{Def:Ploc}.
Let $(\cV,\mh,\conn)$ be a smooth vector bundle
with $\Ck{0,1}$ metric $\mh$ and connection $\conn$
that are compatible almost-everywhere.

Let $\Dirb,\ \Dirp$ be self-adjoint operators
on $\Lp{2}(\cV)$  and assume the hypotheses
\ref{Hyp:First}-\ref{Hyp:Last} on $\cM$, $\cV$, $\Dirb$ 
and $\Dirp$. 
Moreover, assume that
\begin{equation}
\label{Def:Struc}
\Dirp \psi = \Dirb \psi + A_1 \conn \psi + \divv A_2 \psi  + A_3 \psi,
\end{equation}
holds in a distributional sense for $\psi \in \Sob{1,2}(\cV)$, where
\begin{equation}
\label{Def:Coeff}
\begin{aligned}
A_1 &\in \Lp{\infty}( \bddlf(\cotanb\cM \tensor \cV, \cV)),  \\
A_2 &\in \Lp{\infty}(\SobH{1}(\cV), \dom(\divv)), 
\ \text{and} \\
A_3 &\in \Lp{\infty}(\bddlf(\cV)),
\end{aligned}
\end{equation}
and let $\norm{A}_\infty = \norm{A_1}_\infty + \norm{A_2}_\infty + \norm{A_3}_\infty.$

Then, for each $\omega \in (0, \pi/2)$
and $\sigma \in (0, \infty]$, 
whenever $f \in \Hol^\infty(\OSec{\omega, \sigma})$,
we have the perturbation estimate
$$\norm{f(\Dirp) - f(\Dirb)}_{\Lp{2}(\cV) \to \Lp{2}(\cV)} \lesssim \norm{f}_{\Lp{\infty}(\Sec{\omega,\sigma})} \norm{A}_\infty,$$
where the implicit constant depends on $\const(\cM,\cV,\Dirb,\Dirp)$.
\end{theorem}

\begin{remark}
\label{Rem:Main}
 The  assumption of self-adjointness of the operators $\Dirb$
and $\Dirp$ in Theorem \ref{Thm:Main} can be relaxed, as we only 
use this to deduce quadratic estimates for $\Dirb$
and $\Dirp$. For example, it suffices
to assume that $\Dirb$ and $\Dirp$ are similar in $\Lp{2}$
to self-adjoint operators. 
\end{remark} 

\Rd 
\begin{remark}
Although our motivation and key application in is 
in the case that $\Dirb$ and $\Dirp$ correspond to 
the Atiyah-Singer Dirac operators on a Spin manifold
corresponding to two different metrics, we
allow for greater generality in our main theorem
since we anticipate it to have a much broader set of applications. 
For instance, in the study of particle physics, twisted
bundles and their associated twisted Dirac operators 
are of significance and we expect that such situations 
might also be analysed by our main theorem.
For readers interested in such operators, we
hope that \S\ref{Sec:Dirac} will serve as a guideline to
how hypotheses \ref{Hyp:First}-\ref{Hyp:Last}
can be shown to be satisfied.
\end{remark}
\Bk  
\section{Applications to the Atiyah-Singer Dirac operator}
\label{Sec:Dirac}

Let $\cM$ be a smooth 
manifold with a $\Ck{0,1}$ (locally Lipschitz) 
metric $\mg$. We let $\Forms\cM$
denote the bundle of differential 
forms and on fixing 
a Clifford product $\cliff$,
we let $\Cliff\cM = \Cliff \tanb\cM$ denote the 
Clifford bundle. Recall that 
$\Cliff\cM \cong \Forms\cM$ as a vector space.
Moreover, we remind the reader that we identify 
vectorfields and derivations throughout, so $Xf$
means the directional derivative $\partial_X f$
where $X$ is a vectorfield and $f$ is a scalar function.
 
Fix a frame $\set{v_j}$ near $x$, let $\mg_{ij} = \mg(v_i, v_j)$
and define $w^i_{kl}$ at points where $\mg$ is differentiable inside the
frame by
\begin{equation}
\label{Eq:Wcoeff}
w_{kl}^i = \frac{1}{2} \mg^{im}
	(\partial_{v_l} \mg_{mk} + \partial_{v_k} \mg_{ml} - \partial_{v_m} \mg_{kl}
	+ c_{mkl} + c_{mlk} - c_{klm}),
\end{equation}
where $c_{klm} = \mg(\Ld{v_k, v_l}, v_m)$ 
are the \emph{commutation coefficients}
and  $\Ld{\mdot,\mdot}$ is the Lie derivative.
Let $\conform^a_i = w^{a}_{ji} e^j$
be the connection $1$-form, 
and define $\conn^\mg v_j = \conform^a_j \tensor v_a$.
Thus, we obtain the Levi-Civita connection
almost-everywhere in $\cM$
as a map
$\conn^\mg: \Ck{\infty}(\tanb\cM) \to \Sect(\cotanb\cM \tensor \tanb\cM)$.
Note that since $\mg$ is only locally Lipschitz, 
we have that smooth sections are mapped to 
locally bounded  (1,1)-tensors . When the context is clear, we often simply 
denote $\conn^\mg$ by $\conn$.

A manifold $(\cM,\mg)$
is said to be \emph{Spin} if it admits
a spin structure 
$\xi: \Prin{\Spin}(\tanb\cM) \to \Prin{\SO}(\tanb\cM)$, 
i.e., a $2-1$ covering of the frame bundle.
It is well known that
this occurs if and only if the 
first and second Stiefel-Whitney classes
of the tangent bundle vanish. 
The triviality of the first Stiefel-Whitney
class is equivalent to the orientability
of $\cM$.

For the case of $\cM = \R^n$ with $\mg = \delta$,
the usual Euclidean inner product, we
let $\Spinors\R^n$ denote linear space of
standard complex spinors of dimension $2^{\floor{\frac{n}{2}}}$. 
\Rd 
In odd dimensions, this space corresponds to the 
non-trivial minimal complex irreducible representation 
$\eta: \Spin_n \to \bddlf(\Spinors\R^n)$,
where $\Spin_n$ is the spin group, the double
cover of $\SO_n$, and in even dimension
to $\eta = \eta_+ \oplus \eta_-$
where $\eta_{\pm}: \Spin_n \to \bddlf(\Spinors_{\pm} \R^n)$ 
are the representations of the positive/negative half spinors.
For example, see \cite{LM}. \Bk  
We define the standard \emph{(complex) Spin bundle} to be
$$ \Spinors\cM = \Prin{\Spin}(\tanb\cM) \times_{\eta} \Spinors\R^n$$
as the bundle with fibre $\Spinors\R^n$ associated
to $\Prin{\Spin}(\tanb\cM)$ via $\eta$. 
We note that this is the bundle with transition 
functions $(\eta \comp T_{\alpha\beta})$ on $\Omega_\alpha \intersect \Omega_\beta \neq \emptyset$
for $\Omega_\alpha$ and $\Omega_\beta$ open sets,
where $T_{\alpha\beta}: \Omega_\alpha \intersect \Omega_\beta \to \Spin_n$
are transition functions for $\Prin{\Spin}(\tanb\cM)$.

The representation $\eta$ induces an action 
$\rep: \Cliff\cM \to \Spinors\cM$. When $n$ is odd, there are two such 
multiplications up to equivalence opposite from each other, 
and for $n$ even, there is exactly one up to equivalence.
Fixing such a Clifford action,
$\Spinors\cM$ has an induced hermitian metric
$\sinprod{\cdot,\cdot}$, pointwise unique up to scale
satisfying $\sinprod{X \rep \phi, \psi} = - \sinprod{\phi, X \rep \psi}$
for all $X \in \tanb_x\cM$ and $\phi, \psi \in \Spinors_x \cM$
for every $x \in \cM$. See Proposition 1.2.1 and 1.2.3 in \cite{Ginoux}.

Let $E(e_1, \dots, e_n)$ be an orthonormal frame for $\tanb\cM$
and $\set{\spin{e}_\alpha}$ be the induced 
orthonormal spin frame on $\Spinors\cM$.
Let $\conform^a_i$ be the connection $1$-form in $E$
and define the connection 
$\Sconn: \Ck{\infty}(\Spinors\cM)  \to \Lp[loc]{\infty}(\cotanb\cM \tensor\Spinors\cM)$
by writing
\begin{equation}
\label{Eq:SConn} 
\Sconn \spin{e}_\alpha = \frac{1}{2} \sum_{b < a} \conform^a_b \tensor (e_b\rep  \e_a \rep \spin{e}_\alpha).
\end{equation}
This connection satisfies the two following properties:
\begin{enumerate}[(i)]
\item it is almost-everywhere compatible with the induced
spinor metric $\sinprod{\mdot,\mdot}$, and
\item it is a \emph{module derivation}: whenever $X \in \Ck{\infty}(\tanb\cM)$,
$$ \Sconn[X](\omega \rep \psi) = (\conn[X] \omega) \rep \psi + \omega \rep (\Sconn[X] \psi)$$
holds almost-everywhere
for every $\omega \in \Ck{\infty}(\Cliff \cM)$ and $\psi \in \Ck{\infty}(\Spinors\cM)$.
\end{enumerate}

We refer the reader to \S1.2 in \cite{Ginoux} for a exposition
of these ideas, as well as Chapter 2, \S3 to \S5 in \cite{LM}
for a detailed overview, noting that their
proofs in the smooth setting hold in our setting almost-everywhere.

Write 
\begin{equation}
\label{Eq:Conn2Form} 
\conform^2_{E} =  \frac{1}{2} \sum_{b < a} \conform^a_b \tensor e_b \rep e_a
\end{equation} 
to denote the lifting of the connection $2$-form
$\frac{1}{2} \sum_{b < a} \conform^a_b \tensor e_b \wedge e_a$ to $\Spinors\cM$,
and  where $E$ is 
used to denote the dependence on the frame $E(e_1, \dots, e_n)$.
By $\SDir_\mg$ denote the associated Atiyah-Singer Dirac
operator given by the expression
\begin{equation}
\label{Eq:ASDirac}
\SDir_\mg \spin{e}_\alpha = e^j \rep \Sconn[e_j] \spin{e}_\alpha = e^j \rep \conform^2_E(e_j)\rep \spin{e}_\alpha,
\end{equation}
so that $\SDir_\mg(\psi^\alpha \spin{e}_\alpha) = (\conn[e_j] \psi^\alpha)\ e^j \rep \spin{e}_\alpha 
+ \psi^\alpha e^j \rep \Sconn[e_j] \spin{e}_\alpha.$
Note that,
\begin{equation}
\label{Eq:ProdRuleDirac}
\SDir_\mg(\eta \psi) = (\extd \eta)\rep \psi + \eta \SDir_\mg(\psi)
\end{equation} 
for every $\eta \in \Ck{\infty}(\cM)$ and $\psi \in \Ck{\infty}(\Spinors\cM)$
and, as a consequence of the aforementioned
module-derivation property of the connection $\Sconn$ on $\Spinors \cM$, 
\begin{equation}
\label{Eq:ProdRuleDirac2}
\SDir_\mg(\omega \rep \psi) = (\Dir_H \omega) \rep \psi  -\omega \rep \SDir_\mg \psi - 2 \conn[\omega^\sharp] \psi
\end{equation}
for all $\omega \in \Ck{\infty}(\Cliff\cM)$ and $\psi \in \Ck{\infty}(\Spinors\cM)$, 
where $\Dir_H  = \extd + \adj{\extd}$ is the Hodge-Dirac operator, and 
 $\sharp: \cotanb\cM \to \tanb\cM$ given by 
$\omega^\sharp = \mg(\omega, \cdot)$.

Next, let $(\cN,\mh)$ be another Spin manifold with a 
smooth differentiable structure and $\mh$
at least $\Ck{0,1}$. Suppose that $\zeta: \cM \to \cN$
is a $\Ck{1,1}$-diffeomorphism
and let $\Spinors \cN$ denote 
the complex standard spin bundle of $\cN$
obtained via $\eta$.
Following \cite{Bourguignon}, we define an
induced unitary map of spinors $\spin{\U}: \Spinors\cM \to \Spinors\cN$.
Let $P = \pushf{\zeta}: \tanb\cM \to \tanb\cN$.
Then, the pullback metric is
$\mgt(u,v) = \mh(Pu, Pv)$ and we have that
$\mgt(u,v) = \mg((\adj{P}_\mg P)u, v)$, \Rd
where $\adj{P}_\mg$ is the adjoint of $P$, 
and this expression \Bk is readily checked to be a metric of class $\Ck{0,1}$.
On letting $\U = P(\adj{P}_\mg P)^{-\frac{1}{2}}$,
we have that 
$\mh(\U u, \U v) = \mg(u, v).$
So, $\U: (\tanb\cM,\mg) \to (\tanb\cN,\mh)$
is an isometry of class $\Ck{0,1}$.
By $\U(x)$, we denote the induced
linear isometry
$\U(x): (\tanb_x \cM, \mg) \to (\tanb_{\zeta(x)}\cN, \mh)$.

Since $\zeta$ is a homeomorphism, 
an open set $\Omega \subset \cM$ is contractible if and only 
if $\zeta(\Omega) \subset \cN$ is contractible. 
For an orthonormal frame $E(e_1, \dots, e_n) \in \Prin{\SO}(\tanb\cM)$
in $\Omega$,  we obtain $\U E (e_1, \dots, e_n) \in \Prin{\SO}(\tanb\cN)$.
Lifting $E$ and $\U E$ through the spin structures
locally, we obtain two possible maps
 $\spin{\U}_{\Spin, \Omega}: \Prin{\Spin}(\cM) \to \Prin{\Spin}(\cN)$ 
differing by a sign. 
We say that the bundles $\Spinors\cM$ and $\Spinors\cN$
are \emph{compatible} if  $\spin{\U}_{\Spin, \Omega}$  induces a well-defined
global unitary map $\spin{\U}:  \Spinors\cM \to \Spinors\cN$.
By examining the local expression, we see that 
$\spin{\U}: \Spinors\cM \to \Spinors \cN$
and $\spin{\U}^{-1}: \Spinors \cN \to \Spinors \cM$
are $\Ck{0,1}$ maps.

Finally, we say that
$\mg$ and $\mh$ are $C$-close for some $C \geq 1$,
if for all $x \in \cM$,
$$
C^{-1} \modulus{u}_{\mg(x)} 
	\leq   \modulus{\pushf{\zeta}u}_{\mh(\zeta(x))} 
	\leq  	C \modulus{u}_{\mg(x)}.$$
Define
\begin{equation}
C_{L} = \inf\set{C \geq 1: \mg\text{ and }\mh\text{ are $C$-close}}
\quad\text{and}\quad  \met_{M}(\mg,\pullb{\zeta} \mh) = \log(C_L). 
\end{equation}
The map $\met_M$ is readily verified to be a distance-metric on the space
of metrics. 

\Rd What follows is the main the result of this section. 
In fact, this theorem was the original motivation of this paper, 
whereas Theorem \ref{Thm:Main} is a natural generalisation.
As aforementioned, we anticipate the more general result to have
wider implications, particularly to Dirac operators that arise
through twisting the spin bundle by other natural vector
bundles. The analysis of such objects is beyond the scope of this paper
and hence, we focus on the particular case of the Atiyah-Singer
Dirac operator.\Bk

\begin{theorem}
\label{Thm:MainApp}
Let $\cM$ be a smooth Spin manifold with smooth, complete metric $\mg$
with Levi-Civita  connection $\conn^\mg$,
let $\cN$ be a smooth Spin manifold with a $\Ck{0,1}$ 
metric $\mh$, and $\zeta: \cM \to \cN$  a $\Ck{1,1}$-diffeomorphism
with $\met_M(\mg, \pullb{\zeta} \mh) \leq 1$. 
 We assume that the spin bundles $\Spinors\cM$ and $\Spinors\cN$ 
are compatible.  Moreover, suppose that the following hold:
\begin{enumerate}[(i)]
\item 
\label{Hyp:MainAppFirst}
\label{Hyp:Inj}
there exists $\kappa > 0$ such that $\inj(\cM,\mg) \geq \kappa$,
\item 
\label{Hyp:Curv}
there exists $C_{R} > 0$ such that 
	$\modulus{\Ric_\mg} \leq C_R$ and $\modulus{\conn^\mg \Ric_\mg} \leq C_R$, 
\item 
\label{Hyp:MainAppLast}
there exists  $C_{\mh} > 0$ such that $\modulus{\conn^\mg (\pullb{\zeta}\mh)} \leq C_{\mh}$
	almost-everywhere.
\end{enumerate}

Then, for $\omega \in (0, \pi/2)$, $\sigma > 0$, whenever $f \in \Hol^\infty(\OSec{\omega,\sigma})$, 
we have the perturbation estimate
$$
\norm{f(\SDir_\mg) - f(\spin{\U}^{-1}\SDir_\mh \spin{\U})}_{\Lp{2} \to \Lp{2}}	
	\lesssim \norm{f}_{\infty} \met_M(\mg, \pullb{\zeta}\mh) $$
where the implicit
constant depends on $\dim\cM$ and the constants appearing
in \ref{Hyp:MainAppFirst}-\ref{Hyp:MainAppLast}. 
\end{theorem}

\begin{remark}
The map $\spin{\U}$ is the fibrewise unitary map $\Spinors_p\cM \to \Spinors_{\zeta(p)}\cN$.
For the $\Lp{2}(\Spinors\cM) \to \Lp{2}(\Spinors\cN)$ unitary operator
 $\spin{\U}_2 = \sqrt{\det \B}\ \spin{\U}$,  we also have an estimate of
$\norm{f(\SDir_\mg) - f(\spin{\U}_2^{-1}\SDir_\mh \spin{\U}_2)}_{\Lp{2} \to \Lp{2}}$
as in Theorem \ref{Thm:MainApp}. 
Either this can be seen by inspection of the proof, noting Remark \ref{Rem:Main},
or by using  the functional calculus to write 
$$f(\SDir_\mg) - f(\spin{\U}_2^{-1} \SDir_\mh \spin{\U}_2)
	= (f(\SDir_\mg) - f(\spin{\U}^{-1} \SDir_\mh \spin{\U})) +
		(\spin{\U}^{-1}f(\SDir_\mh) \spin{\U} - \spin{\U}_2^{-1}f(\SDir_\mh) \spin{\U}_2),$$ 
noting that the second term is straightforward to bound.
\end{remark}

\begin{remark}
 On fixing a Spin structure $\xi: \Prin{\Spin}(\tanb\cM) \to \Prin{\SO}(\tanb\cM)$,
we obtain an induced
$\xi' = \U \xi: (\xi^{-1} \U^{-1}\Prin{\SO}(\tanb\cN)) \to \Prin{\SO}(\tanb\cN)$
which is a Spin structure for $\cN$. Since $\U:\Prin{\SO}(\tanb\cM) \to \Prin{\SO}(\tanb\cN)$
is a homeomorphism, it is an easy matter to verify 
that the bundles $\Spinors \cN = (\xi^{-1} \U^{-1} \Prin{\SO}(\tanb\cN)) \times_{\eta} \Spinors\R^n$
and $\Spinors\cM$ are compatible.   

  For the case of $\cM = \cN$, where $\Spinors\cM$ and $\Spinors\cN$
denote the respective bundles constructed via $\mg$ and $\mh$, we obtain this 
theorem for $\zeta = \id$. If further $\cM = \cN$ is compact, then \ref{Hyp:MainAppFirst}-\ref{Hyp:MainAppLast} in 
the hypothesis of the theorem are automatically satisfied, and thus we obtain the
result under the sole geometric assumption that $\met_M(\mg, \mh) \leq 1$.  
\end{remark}

\begin{proof}[Proof of Theorem \ref{Thm:MainApp}]
We apply Theorem \ref{Thm:Main}, to the operators
$\Dir = \SDir_\mg$ 
and $\Dirp = \spin{\U}^{-1}\SDir_\mh \spin{\U}$,
setting $\cV = \Spinors\cM$.

The assumptions of completeness of $\mg$ along
with \ref{Hyp:Inj} and \ref{Hyp:Curv}
imply \eqref{Def:Eloc} and \eqref{Def:Ploc}
immediately (see Theorem 1.1 in \cite{Morris3}).
Moreover, there exists
$r_H, C_H > 0$, such that
for all $x \in \cM$ such that $\psi_x: B(x,r_H) \to \R^n$
are coordinate charts such that inside
each chart, $\norm{\mg_{ij}}_{\Ck{2}(B(x, r_H))} \leq C_H$
and $\mg \simeq \pullb{\psi}_x \delta_{\R^n}$
with constant $C_H$. 
See Theorem 1.2 in \cite{Hebey}.

This $\Ck{2}$-control of the metric
inside each $B(x,r_H)$ means that coordinate
frames $\set{\partial_{x_i}}$ satisfy
$\modulus{\conn \partial_{x_i}} \lesssim 1$
and $\modulus{\conn^2 \partial_{x_i}} \lesssim 1$.
On orthonormalisation of these frames
in each $B(x,r_H)$ via the Gram-Schmidt
algorithm yields frames $\set{e_i}$ for $\tanb\cM$, 
$\set{e^i}$ for $\cotanb\cM$ (the dual frame),
and $\set{\spin{e}_\alpha}$ for $\Spinors \cM$.
These are 
smooth GBG frames with constant $C_{\cotanb\cM} = C_{\Spinors\cM} = 1$,
and with $\modulus{\conn e_j}, \modulus{\conn^2 e_j} \lesssim 1$
and $\modulus{\conn \spin{e}_\alpha}, \modulus{\conn^2 \spin{e}_\alpha} \lesssim 1$.
The constants only depend on \ref{Hyp:Inj} and \ref{Hyp:Curv}.
Thus, we have verified the hypotheses \ref{Hyp:First}-\ref{Hyp:VecGBG}.

The hypothesis \ref{Hyp:PDO} follows with $\Ck{\infty}$
coefficients due to the derivation property \eqref{Eq:ProdRuleDirac} of 
$\SDir_\mg$  with constant $C_D = 1$, 
and $\ref{Hyp:DirGBG}$ follows from the
fact that $\modulus{\SDir_\mg \spin{e}_\alpha} \lesssim 
\modulus{\conn \spin{e}_\alpha} \lesssim 1$.

The hypothesis \ref{Hyp:Dom} is proved in \S\ref{Sec:DiracSob}
as  Proposition \ref{Prop:NormCmp}, which makes use of the
completeness of $\mg$, $C$-closeness of $\mh$ to $\mg$ and the
geometric assumptions \ref{Hyp:Inj} and \ref{Hyp:Curv}
The hypothesis \ref{Hyp:Weitz} is proved in \S\ref{Sec:Weitz}
as Proposition \ref{Prop:RW}. It depends
on the crucial covering Lemma \ref{Lem:Cover}
which is a consequence of completeness of $\mg$
coupled with \ref{Hyp:Inj} and \ref{Hyp:Curv}.

The  remaining hypothesis to verify in Theorem \ref{Thm:Main}
is the perturbation structure \ref{Eq:Struc}, which is
done in \S\ref{Sec:Struc}.
\end{proof}

Through the remaining sections, we assume 
the hypothesis of Theorem \ref{Thm:MainApp} to hold.

\subsection{The domain of the Dirac operator as the Spinor Sobolev space}
\label{Sec:DiracSob}

In this section, 
we establish the essential-self adjointness of $\SDir_\mg$
and $\SDir_\mh$. By the smoothness
of $\mg$, it is well known that this operator, 
and all of its positive powers, are essentially-self
adjoint. For instance, see \cite{Chernoff}.
Thus, we focus only on $\SDir_\mh$ which arises
from the lower regularity metric.

First, we assert $\SDir_\mh$ is a symmetric operator on 
$\Ck[c]{\infty}(\Spinors\cN)$.
This is immediate since we assume that $\mh$ is at least
$\Ck{0,1}$, and therefore, the remaining divergence term
in when computing the symmetry pointwise almost-everywhere
is the divergence of a compactly supported Lipschitz vectorfield.
A particular consequence of symmetry is that $\SDir_\mh$ is
a closable operator by the density of
$\Ck[c]{\infty}(\Spinors\cN)$ in
$\Lp{2}(\Spinors\cN)$. 
Operator theory yields that $\close{\SDir_\mh} = \biadj{\SDir}_\mh$.
With these observations in mind, we prove the following.

\begin{proposition}
\label{Prop:EssSA}
The operator $\SDir_\mh$ on $\Ck[c]{\infty}(\Spinors\cN)$ is essentially
self-adjoint.
\end{proposition}
\begin{proof}
The conclusion is established 
if we prove that $\Ck[c]{\infty}(\Spinors\cN)$
is dense in 
$$\dom(\adj{\SDir}_{\mh}) = \set{\psi \in \Lp{2}(\Spinors\cN): 
	\modulus{\inprod{\psi, \SDir_{\mh} \phi}} \lesssim \norm{\phi},\ \phi\in \Ck[c]{\infty}(\Spinors\cN)}.$$

The first reduction we make is to note that 
$\dom_{\rm c}(\adj{\SDir}_{\mh}) = \set{ u \in \dom(\adj{\SDir}_{\mh}): \spt u\ \text{compact}}$
is dense in $\dom(\adj{\SDir}_{\mh})$. This is a direct 
consequence of the fact that  we  are able to find a $C$-close smooth metric $\mh'$,
which is complete since $\mh$ is complete,
and for this metric $\mh'$,  there exists a sequence 
of smooth functions $\rho_k: \cN \to [0,1]$ with $\spt \rho_k$ compact,
with $\rho_k \to 1$ pointwise, and
$\modulus{\extd \rho_k}_{\mh'} \leq C^{-1} 1/k$
for almost-every $x \in \cN$  
(and hence $\modulus{\extd \rho_k}_{\mh} \leq 1/k$ for almost-every  
$x \in \cN$). 
See Proposition 2.3 in \cite{BMc}
or Proposition 1.3.5 in \cite{Ginoux} for the existence
of such a sequence. 
The aforementioned density is then 
simply a consequence of
noting the formula $\adj{\SDir}_{\mh}(f\phi) = f \adj{\SDir}_{\mh}(\phi) + (\extd f) \rep \phi$,
for $f \in \Ck[c]{\infty}(\cN)$ and $\phi \in  \dom(\adj{\SDir}_{\mh})$.

Next, for $\psi \in \dom_{\rm c}(\adj{\SDir}_{\mh}) \intersect \Sob{1,2}(\Spinors\cN)$, 
we can write 
$\psi = \sum_{j=1}^N \psi_j$ where $\psi_j = \eta_j \psi$,
where $\eta_j$ is a finite partition of unity 
and $\spt \eta_j$ is contained in a coordinate
patch.
On obtaining a sequence $\psi^\delta \in \Ck[c]{\infty}(\Spinors\cN)$ 
by obtaining a mollification $\eta_j^\delta$  of $\eta_j$ inside each coordinate patch,
using the fact that
$\psi \in \Sob{1,2}(\Spinors\cN)$ so that 
$\norm{\adj{\SDir}_{\mh} \psi} = \norm{\SDir_{\mh} \psi} \lesssim \norm{\conn \psi}$,
we have that $\psi^\delta \to \psi$ in $\norm{\mdot}_{\adj{\SDir}_{\mh}}$.

The proof is then complete if we show
that whenever $\psi \in \dom_{\rm c}(\adj{\SDir}_{\mh})$, 
we have that $\psi \in \Sob{1,2}(\Spinors\cN)$. 
By the compactness of $\spt \psi$, we
assume without the loss of generality that
$\spt \psi$ is contained in a coordinate patch 
corresponding to a ball $B$.
Thus assume that for every $\phi \in \Ck[c]{\infty}(\Spinors\cN)$,
$ \modulus{\inprod{\psi, \SDir_{\mh} \phi}} \lesssim \norm{\phi}.$
In particular, this holds when $\spt \phi \subset B$, 
so let us further assume that. Then, note that
$$ \inprod{\psi, \SDir_{\mh} \phi}
	 = \int_{B} \sinprod{\psi, e^i \rep (\partial_{e_i} \phi^\alpha) \spin{e}_{\alpha}}\ d\mu_\mh
		+ \int_{B} \sinprod{\psi, e^i \rep \frac{1}{2}\conform^2(e_i)\rep \phi}\ d\mu_{\mh},$$
and since $\conform^2 \in \Lp{\infty}(B)$ since $\mh$ is locally Lipschitz, 
we obtain that
$$ \modulus{\int_{B} \sinprod{\psi, e^i \rep (\partial_{e_i} \phi^\alpha) \spin{e}_{\alpha}}\ d\mu_\mh} 
	\lesssim \norm{\phi}.$$
Moreover, letting $\Leb$ denote the Lebesgue measure, 
we have that $d\mu_\mh = \uptheta d\Leb$, 
where $\uptheta = \sqrt{\det \mh}$ is Lipschitz 
inside $B$ since $\mh$ is locally Lipschitz.
Thus
$$ (\partial_{e_i}\phi^\alpha) \uptheta = \partial_{e_i} (\uptheta \phi^\alpha) - (\partial_{e_i}\uptheta) \phi^\alpha,$$
and since $(\partial_{e_i} \uptheta) \in \Lp{\infty}(B)$, we further obtain that
$$\modulus{\int_{B} \sinprod{\psi, e^i \rep \partial_{e_i}(\uptheta \phi^\alpha) \spin{e}_{\alpha}}\ d\Leb} 
	\lesssim \norm{\uptheta \phi}_{\Lp{2}(B,\Leb)}.$$
Now, note that $e^i \rep \spin{e}_\alpha = \eta(e^i) \spin{e}_\alpha$, 
which is a constant expression inside $B$. Identifying $B$ with $\chi(B)$
where $\chi: B \to \R^n$ is the coordinate map,
$$ \widehat{(e^i \rep f)} = e^i \rep \widehat{f}$$
for $f \in \Lp{2}(\Spinors\R^n)$, where $\widehat{f}$ is
the Fourier Transform of $f$.
On extending $\psi$ by
zero to all of $\R^n$, we obtain that for any 
$\phi \in \Ck[c]{\infty}(\R^n, \Spinors\R^n)$,
$$\modulus{\inprod{\psi,  \partial_{e_i}(\theta \phi^\alpha) e^i \rep \spin{e}_{\alpha}}_{\Lp{2}(\Spinors\R^n)}}
	\lesssim \norm{\uptheta \phi}_{\Lp{2}(\Spinors\R^n)}.$$
Then, by Parseval's identity, we have that 
$$ \inprod{\psi, \partial_{e_i}(\theta \phi^\alpha) e^i \rep \spin{e}_\alpha}_{\Lp{2}(\Spinors\R^n)}
	=  \inprod{\widehat{\psi}, e^i \rep \xi_i \widehat{\theta \phi}}_{\Lp{2}(\Spinors\R^n)}.$$
\Rd That is,
$$
\modulus{\inprod{\widehat{\psi}, \xi \rep \widehat{\theta \phi}}_{\Lp{2}(\Spinors\R^n)}}
	\lesssim \norm{\theta \phi}_{\Lp{2}(\Spinors\R^n)}$$
where $\xi  = \xi_i e^i$ and for all $\phi \in \Ck[c]{\infty}(\Spinors\R^n)$.
Since $\uptheta \phi$ is dense in $\Lp{2}(\Spinors\R^n)$, 
we have that $\xi \rep \widehat{\psi} \in \Lp{2}(\Spinors\R^n)$ \Bk
(since vectors act skew-adjointly on spinors)
which implies that $\psi \in \Sob{1,2}(\Spinors\R^n)$.
On recalling that $\spt \psi \subset B$ and that
\Rd $\conform^2 \in \Lp{\infty}(\Forms[1]\cM \tensor \Cliff \cM)$, \Bk we have that 
$\psi \in \Sob{1,2}(B, \Spinors\cN)$.
\end{proof} 

To characterise the domains of the operators $\SDir_\mg$
and $\SDir_\mh$ as $\Sob{1,2}$, we first note the following lemma. 

\begin{lemma}
\label{Lem:Cover}
On the manifold $(\cM,\mg)$, 
there exists a sequence of points $x_i$
and a smooth partition of unity
$\set{\eta_i}$ uniformly locally finite and 
subordinate to $\set{B(x_i, r_H)}$
satisfying $\sum_{i} \modulus{\conn^j \eta_i}) \leq C_H$
for $j = 0, ..., 3$. Moreover, 
there exists $M > 0$ such that $1 \leq M \sum_i \eta_i^2$.
\end{lemma}
\begin{proof}
The proof of this lemma, except for the 
estimate on the sum of squares of the partition of unity, 
is included in the proof 
of Proposition 3.2 in \cite{Hebey}. This is
due  to the completeness of $\mg$ and \ref{Hyp:Inj} 
and \ref{Hyp:Curv}. We prove the remaining estimate, 
by noting that by the uniformly locally finite
property, there exists a constant $M$ such that
for each $x \in \cM$, $1 = \sum_{k = 1}^M \eta_{i_k}(x)$.
Moreover, by Cauchy-Schwarz inequality,
$$ 1 = \cbrac{\sum_{k = 1}^M \eta_{i_k}(x)}^2 
	\leq \cbrac{\sum_{k = 1}^M \eta_{i_k}(x)^2}\cbrac{\sum_{k=1}^M 1^2}
	= M \sum_i \eta_i^2(x).$$
\end{proof}

With this, the following proposition becomes
immediate.
 
\begin{proposition}
\label{Prop:NormCmp}
We have $\dom(\SDir_\mh) = \Sob{1,2}(\Spinors\cN)$ with
$\norm{\SDir_\mh \phi}^2 + \norm{\phi}^2 \simeq \norm{\Sconn \phi}^2 + \norm{\phi}^2$
whenever $\phi \in \Sob{1,2}(\Spinors\cN)$.
A similar conclusion holds for $\SDir_\mg$.
\end{proposition}
\begin{proof}
By Proposition \ref{Prop:EssSA}, it suffices to 
demonstrate the estimate $\norm{\SDir_\mh \psi} + \norm{\psi} \simeq \norm{\conn \psi} + \norm{\psi}$
for $\psi \in \Ck[c]{\infty}(\Spinors\cN)$.
From the definition of the operator $\SDir_\mh$, we obtain
$\norm{\SDir_\mh \psi} \lesssim \norm{\conn \psi}$ for
all $\psi \in \Ck[c]{\infty}(\Spinors\cN)$. Thus, 
$\Sob{1,2}(\Spinors\cN) \embed \dom(\SDir_\mh)$ is a continuous
embedding.

Let the partition of unity given by Lemma \ref{Lem:Cover} for the metric $\mg$
be denoted by $\set{\eta^\mg_i}$.
Define $\eta_i = \pullb{\zeta} \eta_i^\mg = (\eta^\mg_i \comp \zeta^{-1})$.
Now, $\conn \eta_i = \extd^{\cM} \eta_i$ and by the fact
that pullback commutes with the exterior derivative,
we have that 
$ \extd^\cN \eta_i = \extd^\cN \pullb{\zeta}\eta^\mg_i = \pullb{\zeta} \extd^\cM \eta^\mg_i.$
Thus,
$ \sum_{i} \modulus{ \extd^\cN \eta_i} \leq C C_H,$
since $\mg$ and $\mh$ are $C$-close.

Fix $\psi \in \Ck[c]{\infty}(\Spinors\cN)$ so we can write
$\psi = \sum_{i=1}^N \eta_i \psi$. By Fourier theory, we obtain a
constant $C' > 0$ such that
$\norm{\conn (\eta_i \psi)}^2 \leq C' (\norm{\SDir_\mh (\eta_i  \psi)}^2 + \norm{\eta_i \psi}^2 )$
since $\spt \eta_i \subset B(x_i, r_H)$, which corresponds to a
chart for which the metric $\mg$ is uniformly comparable to 
the pullback Euclidean metric.
 
Moreover, note that since $\conn$ is a derivation,
$ \modulus{\eta_i \conn \psi}^2 
	\lesssim \modulus{\extd^\cN \eta_i}^2 \modulus{\psi}^2
		+ \modulus{\conn (\eta_i \phi)}^2,$
and since $\modulus{\conn \psi}^2 \leq M \sum_i \eta_i^2 \modulus{\conn \psi}^2$
pointwise by Lemma \ref{Lem:Cover}, 
\begin{multline*}
\norm{\conn \psi}^2 
	\leq M \int \sum_i \modulus{\eta_i \conn \psi}^2\ d\mu 
	\lesssim \int \sum_i \modulus{\extd^\cN \psi}^2 \modulus{\psi}^2\ d\mu 
		+ \sum_i \int \modulus{\conn(\eta_i\psi)}^2\ d\mu \\
	\lesssim \norm{\psi}^2 + \sum_i \int \modulus{\SDir_\mh(\eta_i \psi)}^2\ d\mu .
\end{multline*}
But by the definition of $\SDir_\mh$, we have that
$$\modulus{\SDir_\mh (\eta_i \psi)}^2 \lesssim  \modulus{\extd^\cN \eta_i}^2 \modulus{\psi}^2 + \eta_i^2 \modulus{\SDir_\mh \psi}^2.$$
Integrating this estimate and on combining it with 
the previous estimates proves the claim.
The argument for $\mg$ is similar.
\end{proof}

\begin{remark}
Typically, the estimate $\norm{\SDir_\mh \psi}^2 + \norm{\psi}^2 \simeq
\norm{\conn \psi}^2 + \norm{\psi}^2$ is obtained
via the Bochner-Lichnerowicz-Schr\"odinger-Weitzenb\"ock identity:
$$ \SDir_\mh^2 \psi = -\tr \conn^2 \psi + \frac{1}{4} \Rs^\mh \psi,$$
where $\Rs^\mh$ is the scalar curvature of $\mh$. This would
force $\mh$ to be at least $\Ck{1,1}$ and we would need
to assume that $\Rs^\mh \geq \gamma$ almost-everywhere for some
$\gamma \in \R$. However, the fact that $\mh$ is $C$-close
to the smooth metric $\mg$ with stronger curvature bounds
allow us to work in the setting where $\mh$ is only $\Ck{0,1}$.
\end{remark} 

\subsection{Pullback of Lebesgue and Sobolev spaces of spinors}
\label{Sec:Pullb}

In this section, we demonstrate that
the unitary map $\spin{\U}$ as defined 
before Theorem \ref{Thm:MainApp} induces maps between $\Lp{p}$ spaces and Sobolev spaces.

For the remainder of this section, 
let us write 
\begin{equation}
\label{Eq:B}
\B = (P_\mg P)^{-\frac{1}{2}},
\quad\text{and}\quad \uptheta = \det \B
\end{equation}
so that $\mg(\B^{-1}u, \B^{-1}v) = \mgt(u,v)$ and 
$d\mu_\mg = \uptheta d\mu_{\mgt}$.

\begin{proposition}
The isometry $\U:(\tanb\cM,\mg) \to (\tanb\cN,\mh)$
is of class $\Ck{0,1}$ and
the induced $\spin{\U}: \Spinors\cM \to \Spinors\cM$
itself induces a bounded invertible map 
$\spin{\U}: \Lp{p}(\Spinors\cM) \to \Lp{p}(\Spinors\cN)$
for all $p \in [1,\infty]$ satisfying 
$$\norm{\spin{\U} u}_{\Lp{p}(\Spinors\cN)} \simeq \norm{u}_{\Lp{p}(\Spinors\cM)}.$$
\end{proposition}

In what is to follow, let us fix some notation.
As noted in the proof of Theorem \ref{Thm:MainApp}, 
the assumptions we make yields:
uniform constants  
$r_H, C > 0$ such that at each $x \in \cM$,
the ball $B(x,r_H)$ is contractible
and inside $B(x,r_H)$, we have  orthonormal frames $\set{e_i}$ for $\tanb\cM$ 
and $\set{\spin{e}_\alpha}$ for $\Spinors\cM$
so that 
\begin{equation}
\label{Eq:FB}
\norm{e_i}_{\Ck{2}(B(x,r_H))} \leq C
\quad\text{and}\quad\norm{\spin{e}_\alpha}_{\Ck{2}(B(x,r_H))} \leq C.
\end{equation} 

Let the induced orthonormal frame for
$\tanb\cN$ and $\Spinors\cN$ inside 
$\zeta(B(x,r_H))$ respectively
\begin{equation}
\label{Eq:IndFrame} 
\set{\tilde{e}_i = \U e_i}
\quad\text{and}\quad
\set{\spin{\tilde{e}}_\alpha = \spin{\U} \spin{e}_\alpha}.
\end{equation}
Throughout, by $\Omega$ we mean such a ball $B(x,r_H)$.

\begin{lemma}
\label{Lem:ConExp}
\label{Lem:LdExp}
We have
$\conform^a_b(e_j) = \mg(\conn[e_j]e_b, e_a)$
and 
$$
2\mg(\conn[e_j]e_b, e_a) = \mg(\Ld{e_a, e_b}, e_j)
	+ \mg(\Ld{e_j, e_a}, e_b) - \mg(\Ld{e_b, e_j}, e_a)$$
almost-everywhere inside $\Omega$.
Similarly conclusion holds for $\tilde{\conform}^a_b(\tilde{e}_i)$
with respect to the metric $\mh$.
Moreover: 
$\mh(\Ld{\U u, \U v}, \U w) = \mg(\Ld{\B u, \B v}, \B^{-1} w).$ 
\end{lemma}
\begin{proof}
We note that $\conform^a_b(e_j) = w^a_{jb} = e^a(\conn[e_j]e_b) = \mg(\conn[e_j]e_b, e_a)$
\Rd by \eqref{Eq:Wcoeff}. \Bk
The expression for $2\mg(\conn[e_j]e_b, e_a)$ is well known. \Rd Since
$P = \pushf{\zeta}$, we have $\Ld{Pu, Pv} = P\Ld{u,v}$ and on recalling \eqref{Eq:B},
we obtain  
\begin{multline*}
\mh(\Ld{\U u, \U v}, \U w) 
	= \mh( \Ld{P \B u, P \B v}, P \B w)
	= \mh( P \Ld{\B u, \B v}, P \B w) \\
	= \mgt( \Ld{\B u, \B v}, \B w) 
	= \mg( \B^{-1} \Ld{\B u, \B v}, w)
	= \mg( \Ld{\B u, \B v}, \B^{-1} w).
\qedhere
\end{multline*}
\end{proof}

The following lemma allow us to relate derivatives 
of the metric $\mgt = \pullb{\zeta} \mh$ to the
coefficients of the tensorfield $\B$. 
We note that this lemma can also be obtained via a functional calculus
argument.   Inside $\Omega$, we write $\B = (\beta^i_j)$
and $\B^{-1} = (\bar{\beta}^i_j)$. 
 
\begin{lemma}
\label{Lem:DerBnd}
Then, there is a constant $C_2 > 0$
independent of $\Omega$ such that
such that  $\modulus{\partial_{e_t} \beta^{i}_j} \leq C_2 $
and $\modulus{\partial_{e_t} \bar{\beta}^i_j} \leq C_2$
\end{lemma} 
\begin{proof}
First note that we have 
$\modulus{\partial_{e_t} \mgt(e_i, e_j)} \lesssim 1$
inside $\Omega$, since in this frame, 

$$\conn^\mg (\pullb{\zeta} \mh) = (\partial_{e_t} \mgt_{ij}) e^t \tensor e^i \tensor e^j + 
		\mgt_{ij} e^t \tensor \conn[e_t] (e^i \tensor e^j),$$ 
and by assumption \Rd \ref{Hyp:MainAppLast} of Theorem \ref{Thm:MainApp},\Bk we have that $\modulus{\conn^\mg (\pullb{\zeta} \mh)} \lesssim 1$
as well as $\modulus{\conn[e_t] (e^i \tensor e^j)} \lesssim 1$. 
Now,
$ e_t \bar{\beta}^r_s = - (e_t \beta^q_p) \bar{\beta}^r_q \bar{\beta}^p_s$
and so it suffices to simply bound $\modulus{e_t \beta^{i}_j} \lesssim 1$.
We first note that 
$$ e_t \mgt_{rs} = e_t \mg(\B e_r, \B e_s) = e_t \mg (\B^2 e_r, e_s) = e_t (\B^2)_{rs},$$
where $\B^2 = ((\B^2)_{rs})$ as a matrix.
Thus, we obtain $\modulus{\e_t \beta^r_s} \lesssim 1$
if we are able to prove $\modulus{\e_t \B}\lesssim \modulus{e_t \B^2}$.
Now, by the product rule, note that we obtain 
$e_t \B^2 = \B (e_t \B) + (e_t \B) \B$, and that, 
for a vector $u \in \tanb_x \cM$ with $\modulus{u} = 1$,
\begin{multline*}
\mg( e_t \B^2 u, u) 
	= \mg ( (\B e_t \B)u, u ) + \mg( (e_t \B) \B u, u) \\
	= \mg ( e_t \B u, \B u ) + \mg( \B u, e_t \B) 
	= 2 \mg( \B (e_t \B) u, u)
\end{multline*}
since $\B$ is real-symmetric, as is $e_t \B$.
This proves that the numerical radius
$\nrad(e_t \B^2) = 2\nrad(\B (e_t \B)$.
Moreover, note that $\nrad(\mdot)$ is a norm,
and since any two norms on a finite dimensional 
vector space are equivalent, and by the $C$-closeness
of $\mg$ and $\mgt$ we have that $\modulus{\B u} \geq C^{-1} \modulus{u}$,
$
\modulus{e_t \B^2}
	\simeq \nrad(e_t \B^2)
	= 2 \nrad(\B (e_t \B))
	\simeq \modulus{\B (e_t \B)}
	\geq C^{-1} \modulus{ e_t \B}.$
\end{proof}

With the aid of these lemmas, we obtain the following
boundedness of $\spin{\U}$ between Sobolev spaces. 

\begin{proposition}
\label{Prop:SobMap}
The space $\spin{\U} \Sob{1,2}(\Spinors\cM) = \Sob{1,2}(\Spinors\cN)$
with
$ \norm{\spin{\U} \psi} + \norm{\Sconn^\mh (\spin{\U} \psi)}
	\simeq \norm{ \psi} + \norm{\Sconn^\mg \psi}.$
In fact, the pointwise estimate
$ \modulus{\spin{\U} \psi} + \modulus{\Sconn^\mh (\spin{\U} \psi)}
	\simeq \modulus{ \psi} + \modulus{\Sconn^\mg \psi}$
holds almost-everywhere.
\end{proposition}
\begin{proof}
Note that the \Rd assumptions \ref{Hyp:MainAppFirst}-\ref{Hyp:MainAppLast} 
in Theorem \ref{Thm:MainApp}  imply an open covering 
$\set{\Omega_p = B(p,r_H)}$ \Bk of $\cM$ satisfying
$\modulus{\conn^\mg e_{p,i}} \leq C$ and
$\modulus{\partial_{e_{p,k}} \mgt (e_{p,i}, e_{p,j})} \lesssim C$,
where $\set{e_{p,i}}$ is the frame inside $\Omega_p$.
So, fix $p$ and 
let $\psi \in \Sect(\Spinors\cM)$ be differentiable at $x \in \Omega_p$
and note that at $x$,
$$\modulus{\Sconn^\mh (\spin{\U} \psi)}^2
	= \sum_{j} \modulus{\Sconn_{\tilde{e}_j}^\mh (\spin{\U} \psi)}^2
	= \sum_{j} \modulus{\spin{\U}^{-1} \Sconn_{\tilde{e}_j}^\mh (\spin{\U} \psi)}^2.$$
Now, note that
$$\Sconn_{\tilde{e}_j}^\mh(\spin{\U}\psi) 
	= \partial_{\tilde{e}_j} (\psi^\alpha \comp \zeta^{-1}) \spin{\tilde{e}}_\alpha 
		+ (\psi^\alpha \comp \zeta^{-1}) \Sconn_{\tilde{e}_j}^\mh \spin{\tilde{e}}_\alpha,$$
and that by the chain rule, 
on noting that 
$\Sconn_{\tilde{e}_j}^\mh \spin{\tilde{e}}_\alpha = 
	\frac{1}{2} \sum_{b < a} \conform^2_F(\tilde{e}_j) \tilde{e}_b\rep \tilde{e}_a \rep \spin{\tilde{e}}_\alpha$
\Rd from \eqref{Eq:SConn} and \eqref{Eq:Conn2Form},\Bk we obtain that 
$$\spin{\U}^{-1} \Sconn_{\tilde{e}_j}^\mh(\spin{\U} \psi)
	= \partial_{B e_j} (\psi^\alpha) \spin{e}_\alpha
	+ \psi^\alpha (\conform^b_a(\tilde{e}_j) \comp \zeta) e_b\rep e_a \rep \spin{e}_\alpha.$$
We estimate each term on the right side of the equation.

First, note that by Lemma \ref{Lem:ConExp},  
$$\conform^b_a(\tilde{e}_j) = \frac{1}{2} \cbrac{ \mg( \Ld{\B e_a, \B e_b}, \B^{-1} e_j)
		+ \mg( \Ld{\B e_j, \B e_a}, \B^{-1} e_b)
		- \mg( \Ld{\B e_b, \B e_j}, \B^{-1} e_a)},$$
and by metric compatibility between $\mg$ and $\conn^\mg$,
we have that 
$$\mg(\Ld{ \B e_r, \B e_s}, \B^{-1} e_t) = 
	\mg(\conn^\mg_{\B e_r }(\B e_s ), \B^{-1} e_t) - \mg(\conn^\mg_{\B e_s}(\B e_r), \B^{-1} e_t).$$
We compute 
$$\conn^\mg_{\B e_r} (\B e_s)  
	= \B_r^j \conn^\mg_{e_j}( \B^k_s e_k)
	= \B_r^j \cbrac{(e_j \B^k_s) e_k + \B^k_s \conn^\mg_{e_j} e_k}.$$
On combining these calculations using Lemma \ref{Lem:ConExp}, we obtain that
$$
\sum_j \modulus{\psi^\alpha (\conform^b_a(\tilde{e}_j) \comp \zeta) e_b\rep e_a \rep \spin{e}_\alpha}^2
	\lesssim \modulus{\psi}^2.$$

To estimate the remaining term, we note that
$$
(\partial_{\B e_j} \psi^\alpha) \spin{e}_\alpha 
	= \B^k_j (\partial_{e_k}\psi^\alpha) \spin{e}_\alpha
	= \B^k_j \Sconn_{e_k}^\mg \psi - \B^k_j \psi^\alpha \Sconn_{e_k}^\mg \spin{e}_\alpha.$$
But by Lemma \ref{Lem:ConExp}
$$\modulus{\Sconn^\mg_{e_k} \spin{e}_\alpha} 
	\leq \frac{1}{2} \sum_{b < a} \modulus{\conform^b_a(e_k) e_b \rep e_a \rep \spin{e}_\alpha}
	\lesssim \sum_{b < a} \modulus{\conn^\mg_{e_k}e_a} \modulus{e_b}
	\lesssim 1.$$
Therefore,
$$ \sum_{j}  \modulus{(\partial_{\B e_j} \psi^\alpha) \spin{e}_\alpha} 
	\lesssim \modulus{\Sconn^\mg \psi} + \modulus{\psi}.$$
This proves the pointwise estimate, and interchanging the
roles of $\cM$ and $\cN$ proves the reverse estimate.
\end{proof}

\subsection{The pullback Dirac operator and the structural condition}
\label{Sec:Struc}

In this section, we pullback the Dirac operator $\SDir_\mh$ to 
on $\Spinors\cN$ to an operator $\SDirp$
on $\Spinors\cM$, and prove
\eqref{Def:Struc}.

Fix  an $\Omega = B(x,r_H)$ and 
\Rd let $\psi \in \Sect(\Spinors\cM)$. 
For $y \in \Omega$ for which $\Sconn \psi(y)$ 
exists, define 
\begin{equation}
\label{Eq:SDirs}
\SDir \psi(y) = \SDir_\mg \psi(y)\quad \text{and}\quad 
\SDirp \psi(y) 
	= \spin{\U}^{-1}(y) \SDir_{\mh} (\spin{\U} \psi)(y).
\end{equation}
Recall the map $\B$ from \eqref{Eq:B} and since $\B \in \Sect(\Tensors[1,1]\cM)$,
in an orthonormal frame $\set{e_i}$, we have that
$\B e_i = \beta^j_i e_j$ and $\B e^j = \beta^j_i e^i$. \Bk
Moreover, we note that since $\met_{M}(\mg, \pullb{\zeta}\mh) \leq 1$, 
$\modulus{\delta^j_i - \beta^j_i}\leq \norm{\iden - \B}_\infty \leq \met_M(\mg,\mh)$

First, we examine the structure of the difference 
$\SDirp - \SDir$ locally in a frame, the main
point being the use of the derivation property
in Proposition \ref{Prop:Coeff1}, 
before establishing the global result in Proposition 
\ref{Prop:StrucL2}.
 
Recall from  \eqref{Eq:IndFrame} that $\tilde{e}_i = \U e_i$
and $\spin{\tilde{e}}_\alpha = \spin{\U} \spin{e}_\alpha$.
Note that this is the fibre-wise $\spin{\U}$ 
and not the $\spin{\U}$ in $\Lp{2}$.
We also denote the induced fibrewise Clifford bundle
pullback between $\Cliff\cM$ and $\Cliff\cN$ by $\U$.

\begin{proposition}
\label{Prop:GlobDiff}
We have 
$$(\SDir - \SDirp)\psi  = Z \Sconn \psi 
	- ((\iden - \B)e^i) \rep \conform^2_E(e_i)\rep \psi
	+ e^i \rep (\conform^2_E(e_i) - \U^{-1}\conform^2_F(\tilde{e}_i))\rep \psi,$$ 
distributionally for $\psi \in \Sob{1,2}(\Spinors\cM)$,
where $Z \in \Lp{\infty}(\cotanb\cM \tensor \Spinors\cM, \Spinors\cM)$
with norm $\norm{Z}_\infty \lesssim \norm{\iden - \B}_\infty$. 
\end{proposition}
\begin{proof}
If $\psi = \psi^\alpha \spin{e}_\alpha$, 
$ \spin{\U}\psi = (\psi^{\alpha} \comp \zeta^{-1}) \spin{\tilde{e}}_\alpha$,
and so 
$$
\SDir_\mh \spin{\U} \psi = \tilde{e}^i \rep \partial_{\tilde{e}_i}(\psi^\alpha \comp \zeta^{-1}) \spin{\tilde{e}}_\alpha
	+ (\psi^\alpha \comp \zeta^{-1}) \tilde{e}^i \rep \Sconn[\tilde{e}_i]\spin{\tilde{e}}_\alpha.$$
Thus, on pulling back this expression to $\Spinors\cM$ via $\spin{\U}^{-1}$,
and invoking the chain rule to the first sum in this expression,
we obtain that
$$ \SDirp \psi = e^i \rep (\partial_{\B e_i} \psi^\alpha) \spin{\tilde{e}}_\alpha
	+ \psi^\alpha e^i \rep \spin{\U}^{-1}\Sconn[\tilde{e}_i]\spin{\tilde{e}}_\alpha.$$
Thus, the difference of these operators are given by the expression
$$ (\SDir - \SDirp) \psi 
	= e^i \rep (\partial_{e_i} \psi^\alpha - \partial_{\B e_i} \psi^\alpha) \spin{e}_\alpha 
		+ \psi^\alpha e^i \rep( \Sconn[e_i] \spin{e}_\alpha - \spin{\U}^{-1} (\Sconn[\tilde{e}_i]\spin{\tilde{e}}_\alpha)).$$
\Rd Recalling \Bk that $\Sconn[e_i] \spin{e}_\alpha = \conform^2_E(e_i) \rep \spin{e}_\alpha$
and that 
$$\spin{\U}^{-1} \Sconn[\tilde{e}_i]\spin{\tilde{e}}_\alpha =
	\spin{\U}^{-1} (\conform^2_F(\tilde{e}_i) \rep \spin{\tilde{e}}_\alpha)
	= \U^{-1} \conform^2_F(\tilde{e}_i) \rep \spin{\U}^{-1}\spin{\tilde{e}}_\alpha
	= (\U^{-1} \conform^2_F(\tilde{e}_i)) \rep \spin{e}_\alpha.$$
The first expression is then given by 
\begin{multline*} 
e^i \rep (\partial_{e_i} \psi^\alpha - \partial_{\B e_i} \psi^\alpha) \spin{e}_\alpha
	= (\delta^j_i - \beta^j_i) e^i \rep (\partial_{e_j} \psi^\alpha) \spin{e}_\alpha \\
	= ((\iden - \B)e^j) \rep (\partial_{e_j} \psi^\alpha) \spin{e}_\alpha
	= ((\iden - \B)e^j) \rep \Sconn[e_j] \psi - \psi^\alpha (\iden - \B)e^j \rep \Sconn[e_i]\spin{e}_\alpha.
\end{multline*}
Let $\omega = w^a \tensor \spin{w}_a \in \Sect(\cotanb\cM \tensor \Spinors\cM)$
and define $Z \omega = (\iden - \B)w^a \rep \spin{w}_a$. 
This defines a frame invariant expression with
$$Z \Sconn \psi = ((\iden - \B)e^j) \rep \Sconn[e_j] \psi,$$
and
$ \modulus{Z \omega} 
	= \modulus{(\iden - \B)w^a \rep \spin{w}_a}
	\leq  \modulus{(\iden - \B)w^a} \modulus{\spin{w}_a}
	\lesssim \modulus{w^a} \modulus{\spin{w}_a} \simeq \modulus{\omega}.$
\end{proof}

As a consequence of this proposition, 
we will continue to examine remaining terms of
the expression 
$(\SDir - \SDirp  - Z \Sconn) \psi$
with the main term being 
$e^i \rep (\conform^2_E(e_j) - \U^{-1}\conform^2_F(\tilde{e}_j))\rep \psi$.
\Rd Letting $\B^{-1} = (\bar{\beta}^{j}_i)$ in the frame $\set{e_i}$, note that \Bk  
\begin{equation}
\label{Eq:MainExp}
\begin{aligned}
(\conform^2_E(e_j) &- \U^{-1}\conform^2_F(\tilde{e}_j)) 
	= \frac{1}{2} \sum_{b < a} (\conform^b_a(e_i) - \tilde{\conform}^b_a(\tilde{e}_i) \comp \zeta^{-1})\ e_b \rep e_a \\
	&\quad= \frac{1}{4} \sum_{b < a} \bigg \{ (\mg(\Ld{e_a, e_b}, e_j)
	+ \mg(\Ld{e_j, e_a}, e_b) - \mg(\Ld{e_b, e_j}, e_a)) \\
	&\quad\qquad- (\mh(\Ld{\tilde{e}_a, \tilde{e}_b}, \tilde{e}_j)
	+ \mh(\Ld{\tilde{e}_j, \tilde{e}_a}, \tilde{e}_b) - \mh(\Ld{\tilde{e}_b, \tilde{e}_j}, \tilde{e}_a)) \bigg \}\ e_b \rep e_a \\ 
	&\quad= \frac{1}{4} \sum_{b < a} \bigg \{ (\mg(\Ld{e_a, e_b}, e_j) 
	+ \mg(\Ld{e_j, e_a}, e_b) - \mg(\Ld{e_b, e_j}, e_a)) \\
	&\quad\qquad- (\mg(\Ld{\B e_a, \B e_b},\B^{-1}  e_j)
	+ \mg(\Ld{\B e_j , \B e_a}, \B^{-1} e_b )  \\
	&\qquad\qquad\qquad\qquad\qquad\qquad\qquad- \mg(\Ld{\B e_b, \B e_j}, \B^{-1} e_a)) \bigg \}\ e_b \rep e_a,
\end{aligned} 
\end{equation}
where the last line follows from 
Lemma \ref{Lem:LdExp}.
Hence, it suffices to consider the differences of the form 
$ \mg(\Ld{u, v}, w) - \mg(\Ld{\B u, \B v}, \B^{-1} w)$. 

\begin{lemma}
We have \begin{align*} 
\mg(\Ld{e_i, e_j}, e_k) - \mg(\Ld{\B e_i, \B e_j}, \B^{-1} e_k) 
	&= (\delta_i^a \delta_j^b \delta_k^c - \beta^a_i \beta_j^b \bar{\beta}_k^c )
		\mg(\Ld{e_a, e_b},e_c) \\
	&\qquad\qquad-\mg(\ (\partial_{\B e_i}(\beta_j^a) - \partial_{\B e_j}(\beta^a_i))e_a , \B^{-1} e_k)
\end{align*}
almost-everywhere in $\Omega$.
\end{lemma}
\begin{proof}
Using the derivation property, we obtain that
\begin{align*}
\Ld{\B e_i, \B e_j}f
	&= \partial_{\B e_i}(\beta_j^b) e_b(f) + \beta_j^b \beta_i^a e_a e_b (f)
	 	- 
	 \partial_{\B e_j}(\beta_i^a) e_a(f) - \beta_i^a \beta_j^b e_b e_a (f) \\ 
	&= (\partial_{\B e_i}(\beta_j^a)  - \partial_{\B e_j}(\beta_i^a)) e_a(f)
		+ \beta_i^a \beta_j^b [e_a, e_b]f,
\end{align*}
where the last equality follows from the fact that \Rd $a$ and $b$ 
are dummy indices, i.e., $\beta^b_j e_b = \beta^a_j e_a$. \Bk 
Therefore, 
\begin{multline*}
\mg(\Ld{e_i, e_j}, e_k) - \mg(\Ld{\B e_i, \B e_j}, \B^{-1} e_k)
	= \mg(\Ld{e_i, e_j}, e_k) - \mg(\Ld{\B e_i, \B e_j}, \B^{-1} e_k)  \\ 
	= \mg(\Ld{e_i, e_j}, e_k) - \mg ( \beta_i^a \beta_j^b [e_a, e_b], \bar{\beta}_k^c e_c)
		- \mg(\ (\partial_{\B e_i}(\beta_j^a )- \partial_{\B e_j}(\beta_i^a)) e_a, \B^{-1} e_k).
\end{multline*}
Then, on noting that 
$\mg(\Ld{e_i, e_j}, e_k) = \delta^a_i \delta^b_j \delta^c_k \mg(\Ld{e_a, e_b}, e_c)$,
we obtain the desired conclusion.
\end{proof}

With the aid of this, we re-organise the
expression \eqref{Eq:MainExp} in the following way: 

\begin{equation}
\label{Eq:MainExp2}
\begin{aligned}
(\conform^2_E(e_i) &- \U^{-1}\conform^2_F(\tilde{e}_i)) \\
	&= \frac{1}{4} \sum_{b < a} (\Xi^{qrs}_{abi} + \Xi_{iab}^{qrs} - \Xi^{qrs}_{bia}) 
		\mg(\Ld{e_q,e_r}, e_s)\ e_b \rep \e_a \\ 
	&\qquad\qquad+\frac{1}{4} \sum_{b < a} (\Upsilon_{abi} - \Upsilon_{bai} 
		+ \Upsilon_{iab} - \Upsilon_{aib} + \Upsilon_{iba}
		- \Upsilon_{bia}) \ e_b \rep e_a,
\end{aligned}
\end{equation}
where
$\Xi^{qrs}_{abc} = (\delta^q_a \delta^r_b \delta^s_c - \beta^q_a \beta^r_b \bar{\beta}^s_c)$,
$\Upsilon_{abc} = \partial_{\B e_a}(\beta^p_b)\bar{\beta}^q_c \delta_{pq}$.
We analyse terms of the form $\Upsilon_{rst}\ e_b \rep e_a$ where $(r,s,t)$ are permutations
of $\set{a,b,i}$.

\begin{lemma}
The following holds almost-everywhere in $\Omega$:
$$\Upsilon_{abc} = \tr \conn^\mg (\Lambda_{abc}) 
	- \upepsilon^p_b \partial_{\B e_l} (\bar{\beta}^q_c \theta_{ad}) \bar{\beta}^l_m \delta^{md}
	 + e_d(\Lambda_{abc}) w_{mk}^d \Rd \delta^{mk},$$ \Bk 
where $\tr$ denotes the trace with respect to the metric $\mg$
and where 
$\upepsilon^p_b = \beta^p_b - \delta^p_b$, $\Lambda_{abc} = \upepsilon^p_b \bar{\beta}_c^q \delta_{pq} \theta_{ad}\ e^d$
and  \Rd $\theta_{ad} = \beta^a_d = \delta_{ak} \beta^{k}_d$. \Bk 
\end{lemma}
\begin{proof}
We compute $\conn(\Lambda_{abc})$ on letting $v_a = \B e_a$
\begin{align*} 
\conn(\Lambda_{abc}) &= v^l \tensor \conn[v_l](\upepsilon^p_b \bar{\beta}^q_c \delta_{pq} \theta_{ad}\ e^d) \\
	&= \partial_{v_l}(\upepsilon^p_b) \bar{\beta}^q_c \delta_{pq} \theta_{ad} \bar{\beta}^l_m\ e^m \tensor e^d
		+ \upepsilon^p_b \partial_{v_l} (\bar{\beta}^q_c \theta_{ad}) \delta_{pq} \bar{\beta}^l_m\ e^m \tensor e^d \\
		&\qquad\qquad+ e_d(\Lambda_{abc}) v^l \tensor \conn[v_l] e^d.
\end{align*}
 Now, note that $v^l \tensor \conn[v_l] e^d = e^m \tensor \conn[e_m] e^d = -w_{mk}^d e^m \tensor e^k$ and hence,
$$ e_d(\Lambda_{abc}) v^l \tensor \conn[v_l] e^d
	= -e_d(\Lambda_{abc}) w_{mk}^d\ e^m \tensor e^k.$$
Take the trace with respect to $\mg$ to get
\begin{align*}
\tr \cbrac{\partial_{v_l}(\upepsilon^p_b) \bar{\beta}^q_c \delta_{pq} \theta_{ad} \bar{\beta}^l_m\ e^m \tensor e^d}
	&=\partial_{v_l}(\upepsilon^p_b) \bar{\beta}^q_c \delta_{pq} \theta_{ad} \bar{\beta}^l_m \delta^{md} \\
	&=\partial_{v_l}(\upepsilon^p_b) \bar{\beta}^q_c \delta_{pq} \delta^l_a 
	= \partial_{v_a}(\upepsilon^p_b) \bar{\beta}^q_c \delta_{pq} = \Upsilon_{abc}
\end{align*}
since 
$\theta_{ad}\bar{\beta}^l_m \delta^{md} = \sum_{m} \theta_{am} \bar{\beta}^l_m 
	= \sum_{m} \beta^a_m\bar{\beta}^l_m =  \delta^l_a$
by the symmetry of $\beta^p_q$. This yields the stated identity.
\end{proof} 

With this, we obtain the following local decomposition.
\begin{proposition}
\label{Prop:Coeff1}
There are pointwise multiplication operators  $X^{\Omega} \in \Lp{\infty}(\bddlf(\Spinors \Omega))$
and $Y^\Omega \in \Lp{\infty}(\bddlf(\cotanb \Omega \tensor \Spinors  \Omega, \Spinors  \Omega))$
and $\Lambda^\Omega \in \Lp{\infty} \intersect 
	\Lips(\bddlf(\Spinors\Omega, \cotanb \Omega \tensor \Spinors\Omega)))$
such that 
\begin{multline*}
\divv (\Lambda^\Omega \psi) + Y^\Omega \Sconn \psi + X^\Omega\psi \\
	= \frac{1}{4} \sum_{b < a} (\Upsilon_{abi} - \Upsilon_{bai} 
		+ \Upsilon_{iab} - \Upsilon_{aib} + \Upsilon_{iba}
		- \Upsilon_{bia}) \ e_b \cliff e_a \rep \psi 
\end{multline*}
holds distributionally for $\psi \in \Sob{1,2}(\Spinors\cM)$.
Moreover,
\begin{align*} 
&\norm{X^\Omega}_\infty \lesssim \norm{\iden - \B}_\infty,
\quad \norm{Y^\Omega}_\infty \lesssim \norm{\iden - \B}_\infty,\\
&\norm{\Lambda^\Omega}_\infty \lesssim \norm{\iden - \B}_\infty,
\ \text{and}\quad \norm{\Sconn \Lambda^\Omega}_\infty \lesssim 1,
\end{align*}
where the implicit constants in the gradient bound for $\Lambda^\Omega$
is independent of $\Omega$.
\end{proposition}
\begin{proof}
By the completeness and smoothness of $\mg$ along
with \ref{Hyp:Inj} and \ref{Hyp:MainAppLast} of Theorem \ref{Thm:MainApp}
we have uniform constants $C_1, C_2 > 0$
so that 
$\modulus{\conn e_a} \leq C_1$
and $\modulus{\partial_{e_c} \mgt_{ab}} \leq C_2$ inside $\Omega$.
Let $\Lambda^\Omega \psi = \Lambda_{rst} \tensor (e_b \rep e_a \rep \psi) \Rd 
= (\upepsilon^p_s \bar{\beta}_t^q \delta_{pq} \delta_{rk} \beta^{k}_d)\ 
	e^d \tensor (e_b \rep e_a \rep \psi)$ \Bk
and note that 
$$
\conn(\Lambda_{rst} \tensor (e_b \rep e_a \rep \psi))
	= \conn(\Lambda_{rst}) \tensor  (e_b \rep e_a \rep \psi)
	+ \Lambda_{rst}  \tensor \Sconn (e_b \rep e_a \rep \psi),$$
where
$$\Sconn(e_b \rep e_a \rep \psi)
	= e^m \tensor \conn[e_m](e_b \rep e_a) \rep \psi
		+ e^m \tensor (e_b \rep e_a)\rep \Sconn[e_m]\psi.$$
Taking traces with respect to $\mg$, we obtain that
\begin{multline*}
\tr \conn(\Lambda_{rst} (e_b \rep e_a \rep \psi))
	= (\tr \conn(\Lambda_{rst}))  (e_b \rep e_a \rep \psi)
	+\tr( \Lambda_{rst}  \tensor \Sconn (e_b \rep e_a \rep \psi)).
\end{multline*}
Moreover, \Rd note that we can write  \Bk 
$\Lambda_{rst} = e_d(\Lambda_{rst}) e^d$ and therefore,
we obtain that
\begin{multline*}
 \Lambda_{rst}  \tensor \Sconn (e_b \rep e_a \rep \psi)
	=  e_d(\Lambda_{rst}) e^d \tensor e^m \tensor \Sconn[e_m](e_b \rep e_a) \rep \psi
		+ e_d(\Lambda_{rst}) e^d \tensor e^m \tensor (e_b \rep e_a)\rep \Sconn[e_m]\psi
\end{multline*}
so that
\begin{multline*}
\tr ( \Lambda_{rst}  \tensor \Sconn (e_b \rep e_a \rep \psi))
	=  e_d(\Lambda_{rst}) \delta^{md} \Sconn[e_m](e_b \rep e_a) \rep \psi
		+ e_d(\Lambda_{rst}) \delta^{dm} (e_b \rep e_a)\rep \Sconn[e_m]\psi.
\end{multline*}
Define 
\begin{multline*}
X_{rst}^\Omega \psi=  e_d(\Lambda_{rst}) \delta^{md} \Sconn[e_m](e_b \rep e_a) \rep \psi \\
	+ \cbrac{e_d(\Lambda_{rst}) w_{mk}^d \delta^{mk} 
	- \upepsilon^p_s \partial_{\B e_l} (\bar{\beta}^q_t \theta_{rd}) \bar{\beta}^l_m \delta^{md}} e_b \rep e_a \rep \psi,
\end{multline*}
and for $\phi \in \Sect(\cotanb\cM \tensor \Spinors\cM)$, define 
$$ Y_{rst}^\Omega \phi 
	= Y^\Omega(\phi^\alpha_a e^a \tensor \spin{e}_\alpha)
	= e_d(\Lambda_{rst}) \delta^{da} \phi^\alpha_a (e_b \rep e_a)\rep \spin{e}_\alpha.$$

Estimating with Lemma \ref{Lem:DerBnd},
we get $\norm{X_{rst}^\Omega}_\infty \lesssim \norm{\iden - \B}_\infty$,
$\norm{Y_{rst}^\Omega}_\infty \lesssim \norm{ \iden - \B}_\infty$, 
$\norm{\Lambda_{rst}} \lesssim \norm{\iden - \B}_\infty$ and $\modulus{\Sconn \Lambda_{rst}^\Omega} \lesssim 1$.

Lastly, by taking a sum over permutations over $\set{abc}$ for the indices
$\set{r,s,t}$, the existence of coefficients $X^\Omega$, $Y^\Omega$ and $\Lambda^\Omega$
as stated in the conclusion is then immediate.
\end{proof} 

By collating our efforts throughout this section, we obtain the following
main result.

\begin{proposition}
\label{Prop:OpDiff}
\label{Prop:StrucL2}
We have
\begin{equation}
\label{Eq:Diff}
\SDirp \psi = \SDir \psi + A_1 \Sconn \psi + \divv A_2 \psi + A_3 \psi,
\end{equation}
distributionally for $\psi \in \Sob{1,2}(\Spinors\cM)$
where the coefficients $A_1, A_2, A_3$ satisfy
\begin{align*}
&A_1 \in \Lp{\infty}(\bddlf(\cotanb\cM \tensor \Spinors\cM, \Spinors \cM)),  \\ 
&A_2 \in \Lp{\infty}(\bddlf(\Sob{1,2}(\Spinors\cM), \dom(\divv))) \\
&A_3 \in \Lp{\infty}(\bddlf(\Spinors \cM))
\end{align*}
with $\norm{A_1}_\infty + \norm{A_2}_\infty + \norm{A_3}_\infty \lesssim \norm{\iden - \B}_\infty$
and $\norm{\Sconn A_2} \lesssim 1$.
\end{proposition}
\begin{proof}
First, we remark that by the assumptions in 
Theorem \ref{Thm:MainApp}, exist constants $C_1, C_2, C_3 > 0$, 
a covering $\set{B_j}$  which are of fixed radius $r > 0$
with orthonormal frames $e_{j,k}$
inside $B_j$, and a Lipschitz partition of unity $\set{\eta_p}$
subordinate to $\set{B_p}$ satisfying:
\begin{enumerate}[(a)]
\item $\modulus{\conn e_{j,i}} \leq C_1$
	for all $i$ almost-everywhere on $\close{B_p}$, 
\item $\modulus{\partial_{e_{j,k}}\ \mgt(e_{j,i}, e_{j, l})} \leq C_2$, 
	where $\mgt = \pullb{\zeta}\mh$, and 
\item $\modulus{\conn{\eta_j}} \leq C_3$ in $B_j$. 
\end{enumerate}

Let
$$W^{B_j} \psi = \frac{1}{4} \sum_{b < a} (\Xi^{qrs}_{abi} + \Xi_{iab}^{qrs} - \Xi^{qrs}_{bia}) 
		\mg(\Ld{e_q,e_r}, e_s)\ e_b \rep  \e_a 
	\Rd - ((\iden - \B)e^i) \rep \conform^2_E(e_i),\Bk$$
and recall the operator $Z$ from Proposition \ref{Prop:GlobDiff},
$\Lambda^U$, and $Y^U$ and $X^U$ from Proposition \ref{Prop:Coeff1}.
Inside $B_j$, we have the expression
$$ (\SDirp - \SDir)\psi = \sum_j \eta_j \divv (\Lambda^{B_j}\psi) + (Z + \sum_{j} \eta^j Y^{B_j}) \conn \psi 
	+ \sum_j \eta_j X^{B_j} \psi + \sum_j \eta_j W^{B_j} \psi$$

On noting that $\divv(\eta \phi) =  \eta \divv \phi + \tr (\conn \eta \tensor \phi)$
for $\eta \in \Ck{\infty}(\cM)$ and $\phi \in \Sect(\cotanb\cM \tensor\cV)$
differentiable almost-everywhere, we let  
\begin{align*}
A_1 &= Z + \sum_{j} Y^{B_j} \eta_j, \\ 
A_2 &= \sum_{j} \Lambda^{B_j} \eta_j, \\
A_3 &= X^{B_j} \eta_j + \sum_{j} W^{B_j} \eta_j - \sum_{j} \tr ((\conn \eta_j) \tensor \psi).
\end{align*}
It is easy to check that the decomposition of the operator
holds almost-everywhere.
The conditions (a) and (b) yield that 
$\norm{A_1} + \norm{A_2} +  \norm{A_3} \lesssim \norm{\iden - \B}_\infty$
by Propositions \ref{Prop:Coeff1}. Moreover, 
$$ \modulus{\conn A_2} 
	\leq \sum_j \modulus{\conn \eta_j} \modulus{\Lambda^{B_j}} 
	+  \sum_j \eta_j \modulus{\Lambda^{B_j}} \lesssim 1,$$
almost-everywhere uniformly with the constant depending on $C_1, C_2$ and $C_3$.
\end{proof}

\subsection{Riesz-Weitzenb\"ock formula for Dirac operator}
\label{Sec:Weitz}

The goal of this subsection is to demonstrate
\ref{Hyp:Weitz}. We begin by noting the following.

\begin{lemma}
\label{Lem:Density}
\Rd The Sobolev spaces satisfy \Bk
$\Sob[0]{2,2}(\Spinors\cM) = \Sob{2,2}(\Spinors\cM)$.
\end{lemma}
\begin{proof}
\Rd 
Due to the geometric assumptions \ref{Hyp:Inj} and \ref{Hyp:Curv}
in Theorem \ref{Thm:MainApp},  \Bk 
the argument to prove the assertion proceeds 
exactly as Proposition
3.2 in \cite{Hebey}, which is a version of this
result for functions. The crucial point in 
the proof is to note that  by the derivation
property for $\conn$, for $\eta \in \Ck{\infty}(\cM)$ and $u \in \Ck{\infty}(\cV)$
\[\modulus{\conn^2 (\eta u)}
	\leq \modulus{\eta} \modulus{\conn^2 u} + 2 \modulus{\conn \eta} \modulus{\conn u}
	+ \modulus{\conn^2 \eta} \modulus{u}.\qedhere \] 
\end{proof}

With this, we obtain the following
Riesz-Weizenb\"ock estimate.

\begin{proposition}
\label{Prop:RW}
There exists $C_W > 0$ such that 
$ \norm{\Sconn^2 \psi} \leq C_W ( \norm{\SDir^2_\mg \psi} + \norm{\psi})$
for all $\psi \in \dom(\SDir^2_\mg) = \Sob[0]{2,2}(\Spinors\cM) = \Sob{2,2}(\Spinors\cM)$.
\end{proposition} 
\begin{proof}
Since our metric $\mg$ is smooth, \Rd by Theorem 2.2 in \cite{Chernoff}, \Bk
it is well known 
that $\Ck[c]{\infty}(\Spinors\cM)$
is dense  (with norm $\norm{\mdot}_{\SDir^2}$) 
in the domain of $\SDir^2_\mg$ (and in fact 
for any positive power  $\SDir^k_\mg$).
By Lemma \ref{Lem:Density}, in order to obtain the conclusion, 
it suffices to establish 
\begin{equation}
\label{Eq:EE}
\norm{\Sconn^2 \psi} \lesssim \norm{\SDir^2_\mg \psi} + \norm{\psi}
\end{equation} 
for all $\psi \in \Ck[c]{\infty}(\Spinors\cM)$.

First we show that \eqref{Eq:EE} holds for $\psi \in \Ck[c]{\infty}(\Spinors\cM)$
with $\spt \psi \subset B(x,r_H)$. 
To consider just the second-order part of the operator $\SDir_\mg^2$,
we define
\begin{multline*}
L\psi = \SDir^2_\mg \psi - e^i \rep e^j \rep ( (e^j \psi_\alpha) \Sconn[e_i] \spin{e}_\alpha 
	+ (e_i \psi_\alpha) \Sconn[e_j] \spin{e}_\alpha 
	+ \psi^\alpha \Sconn[e_i]\Sconn[e_j] \spin{e}_\alpha)\\ 
		- e^i \rep \conn[e_i]e^j \rep \Sconn[e_j]\psi.
\end{multline*}
Estimating this operator by Plancherel's theorem, we get 
$\norm{D_2\psi}_{\Lp{2}(B(x,r_H))}^2 \lesssim \norm{L \psi}^2 + \norm{\psi}^2$,
where $D_2 = e^i \tensor e^j \tensor (e_i e_j \psi_\alpha) \spin{e}_\alpha$
is the second-order part of the Hessian. Also, 
\begin{multline*}
\norm{L \psi}^2 
	\lesssim  \norm{\SDir^2_{\mg} \psi}^2
		+ \max_\alpha \norm{\spin{e}_\alpha}_{\Ck{1}(B(x,r_H))}^2 \norm{\Sconn \psi}^2
		+ \norm{\spin{e}_\alpha}_{\Ck{2}(B(x,r_H))}^2 \norm{\psi}^2 \\
		+ \max_j \norm{e_j}^2_{\Ck{1}(B(x,r_H))} \norm{\Sconn \psi}^2.
\end{multline*}
As we have noted in 
\eqref{Eq:FB}, \Rd a consequence of the assumptions \ref{Hyp:MainAppFirst}-\ref{Hyp:MainAppLast}
in Theorem \ref{Thm:MainApp} \Bk is that
$\max_\alpha \modulus{\Sconn \spin{e}_\alpha} \lesssim 1$
and $\max_\alpha \modulus{\Sconn^2 \spin{e}_\alpha} \lesssim 1$
inside $B(x,r_H)$  \Rd with constants independent of $B(x,r_H)$. \Bk
Again, by Plancherel's theorem,
$$\norm{\Sconn \psi}^2 \lesssim \norm{\SDir_\mg \psi}^2 + \norm{\psi}^2 
		\lesssim \norm{\SDir^2_\mg \psi}^2 + \norm{\psi}^2.$$
Combining these estimates, we obtain that  
$\norm{\conn^2 \psi}^2 \lesssim \norm{\SDir_\mg^2 \psi}^2 + \norm{\psi}^2$.

Now, let $\psi \in \Ck[c]{\infty}(\Spinors\cM)$
and note by the assumptions we make, 
on invoking Lemma \ref{Lem:Cover}, we obtain
$C_H > 0$ such that 
$\set{B_i = B(x_i,r_H)}$ is a cover for $\cM$ with
$\norm{\mg_{ij}}_{\Ck{2}(B_i))} \leq C_H$ and a
smooth partition of unity $\set{\eta_i}$ such that
$\sum_i \modulus{\conn^j \eta_i}\leq  C_H$
for $j = 0, \dots, 3$. \Rd Moreover, this lemma guarantees
that there exists \Bk $M > 0$ such that $1 \leq M \sum_i \eta_i^2$.
From the derivation property for $\conn$, we obtain 
$$ \modulus{\eta_i \conn^2 \psi} 
	\lesssim \modulus{\conn^2 \eta_i}^2 \modulus{\psi}^2
		+ \modulus{\conn \eta_i}^2 \modulus{\conn \psi}^2 
		+ \modulus{\conn^2(\eta_i \psi)}^2,$$
and we have that
\begin{align*}
\norm{\conn^2 \psi}^2 &\leq  \int M \sum_i \eta_i^2 \modulus{\conn^2 \psi}^2\ d\mu \\
	&\leq M \int \sum_i \modulus{\conn^2 \eta_i}^2 \modulus{\psi}^2\ d\mu
	 	+ M \int \sum_i \modulus{\conn \eta_i}^2 \modulus{\conn \psi}^2\ d\mu \\
		&\qquad\qquad+ M \int \sum_i \modulus{\conn^2 (\eta_i \psi)}^2\ d\mu\\
	&\lesssim \norm{\psi}^2 + \norm{\conn \psi}^2 + \sum_i \norm{\conn^2 (\eta_i \psi)}^2.
\end{align*}
Now, 
$\spt (\eta_i \psi) \subset \Ball(x_i, r_H)$ and so 
$\norm{\conn^2 (\eta_i \psi)}^2 \lesssim \norm{\SDir_\mg^2 (\eta_i\psi)}^2 + \norm{\psi}^2$
by \Rd what we have just calculated, and so \Bk on noting
that $\SDir_\mg^2 (\eta_i \psi) = \eta _i\SDir_\mg^2 \psi - 2 \conn[(\grad \eta_i)]\psi - (\Lap \eta_i) \psi$
by \eqref{Eq:ProdRuleDirac2}, \Rd where $\grad \eta_i = (\conn \eta_i)^\sharp = \mg(\conn \eta_i, \cdot)$, \Bk
we estimate 
\begin{align*}
\sum_i \norm{\conn^2(\eta_i \psi)}^2 
	&\lesssim \sum_i \int \eta_i \modulus{\SDir_\mg^2 \psi}^2\ d\mu 
		+ \int \sum_i \modulus{\conn \eta_i}^2 \modulus{\psi}^2\ d\mu\\
		&\qquad\qquad+ \int \sum_i \modulus{\conn^2 \eta_i}^2 \modulus{\psi}^2\ d\mu \\ 
	&\lesssim \norm{\SDir_\mg^2 \psi}^2 + \norm{\psi}^2.
\end{align*}
In Proposition \ref{Prop:NormCmp}, we have already shown that
$\norm{\conn \psi}^2 \lesssim \norm{\SDir_\mg \psi}^2 + \norm{\psi}^2$ 
and hence it suffices to note that  
$$\norm{\SDir_\mg \psi}^2 = \inprod{\SDir_\mg^2 \psi, \psi} 
	\leq \norm{\SDir_\mg^2 \psi} \norm{\psi}
	\lesssim \norm{\SDir_\mg^2 \psi}^2 + \norm{\psi}^2,$$
to complete the proof.
\end{proof}
\section{Reduction to quadratic estimates}
\label{Sec:Red}

\Rd 
The estimates in this section are operator theoretical in their nature and only 
make use of the structure \eqref{Def:Struc}
of the perturbation, along with the assumption that $\Dirp$  and $\Dirb$ 
are self-adjoint operators with domains contained in $\SobH{1}(\cV)$.
We will show how to reduce the estimate \Bk
of $f(\Dirp) - f(\Dirb)$ in Theorem \ref{Thm:Main} to
quadratic estimates. We will see in \S\ref{Sec:SFE}
that the latter type of estimates allow
us to prove the main theorem via harmonic analysis techniques.
Throughout this section, we assume the hypothesis
of Theorem \ref{Thm:Main}. 

\subsection{Perturbations of resolvents}

Since the operators $\Dirb$ and $\Dirp$
are both self-adjoint, they admit
a Borel functional calculus via the
spectral theorem as well as a bounded
holomorphic functional calculus as outlined
in \S\ref{Sec:FunC}.

For $t > 0$, let us define operators
$$ 
\Ppb_t = \frac{1}{\iden + t^2\Dirb^2},\  
\Pp_t = \frac{1}{\iden + t^2\Dirp^2},\ 
\Qqb_t = t\Dirb \Ppb_t,
\quad\text{and}\quad
\Qq_t = t\Dirp \Pp_t.$$
The fact that $\Dirb$ and $\Dirp$ are  self-adjoint 
gives
$$
\int_{0}^\infty \norm{\Qq_t u}^2\ \dtt \leq \frac{1}{2} \norm{u}^2
\quad\text{and}\quad
\int_{0}^\infty \norm{\Qqb_t u}^2\ \dtt \leq \frac{1}{2} \norm{u}^2,$$
and also
$$ \sup_{t} \norm{\Ppb_t},\ 
\sup_{t} \norm{\Pp_t},\ 
\sup_{t} \norm{\Qqb_t},\ 
\sup_{t} \norm{\Qq_t} \leq \frac{1}{2}.$$
Furthermore, we note that \Rd the operators
$\Ppb_t,\ \Ppb_t,\ \Qqb_t,\ \Qq_t$ are \Bk self-adjoint.

Moreover, let
$$ \psi(\zeta) = \frac{\zeta}{1 + \zeta^2}
\quad\text{and}\quad
\psi_t(\zeta) = \psi(t\zeta)$$
and note that $\Qqb_t = \psi_t(\Dirb)$ and 
$\Qq_t = \psi_t(\Dirp)$.
We establish some operator theoretic facts about
$\Qq_t$ and $\Qqb_t$ that will be of
use to us later.

Let 
$$
\Rr_t = \frac{1}{\iden + \imath t \Dirp} = -(\imath t)^{-1} \rs{\Dirp}(-(\imath t)^{-1})
\quad\text{and}\quad
\Rrb_t = \frac{1}{\iden + \imath t \Dirb} = -(\imath t)^{-1} \rs{\Dirb}(-(\imath t)^{-1}),$$
and note that
\begin{equation}
\label{Eqn:Rrt}
\Rr_t 
	= \frac{1}{\iden + \imath t\Dirp}
	= \frac{1}{\iden + \imath t\Dirp} 
		\frac{\iden - \imath t\Dirp}{\iden - \imath t\Dirp}
	= \frac{1}{\iden + t^2\Dirp^2} - \imath \frac{t\Dirp}{\iden + t^2 \Dirp^2}
	= \Pp_t - \imath \Qq_t.
\end{equation}
Similarly, $\Rrb_t = \Ppb_t - \imath \Qqb_t$. 

\begin{proposition}
\label{Prop:Paraprod}
The difference of the resolvents satisfies the formula:
$$\Rr_t - \Rrb_t = \Rr_t[\imath t(\Dirb - \Dirp)]\Rrb_t.$$
Moreover, 
\begin{equation*}
\Qq_t - \Qqb_t 
	= - \Pp_t[t(\Dirp - \Dirb)]\Ppb_t 
		- \Qq_t[t(\Dirp - \Dirb)]\Qqb_t
\end{equation*}
\end{proposition}
\begin{proof}
First, note that: 
$$
\Rr_t - \Rrb_t = \Rr_t(1 + \imath t \Dirb)\Rrb_t 
		- \Rr_t(1 + \imath t\Dirp)\Rrb_t.$$
Since by assumption
$\dom(\Dirp) = \dom(\Dirb) = \SobH{1}(\cV)$,
we have that 
$\ran(\Rr_t) = \dom(\Dirp)$ and hence, 
$(\iden + \imath t \Dirp)\Rrb_t \in \bddlf(\Hil)$.
Thus,
$$\Rr_t - \Rrb_t =
	\Rr_t[ (1 + \imath t\Dirb) - (1 + \imath t \Dirp)]\Rrb_t 
	= \Rr_t[ \imath t (\Dirb - \Dirp)]\Rrb_t.$$

Expanding $\Rr_t = \Pp_t - \imath \Qq_t$
as we noted in \eqref{Eqn:Rrt}, 
a straightforward calculation yields that
\begin{multline*}
(\Pp_t - \Ppb_t) - \imath(\Qq_t - \Qqb_t) 
	= \Rr_t - \Rrb_t = \Pp_t[t(\Dirb - \Dirp)]\Qqb_t + \Qq_t[t(\Dirb -\Dirp)]\Ppb_t \\
		+ \imath \dbrac{\Pp_t[t(\Dirb - \Dirp)]\Ppb_t + \Qq_t[t(\Dirb - \Dirp)]\Qqb_t},
\end{multline*} 
which shows the expression for $\Qq_t - \Qq_t$.
\end{proof}

In particular, we see that
\begin{equation}
\label{Eqn:PsiToPtQt}
\begin{aligned} 
&\norm{(\Qq_t - \Qqb_t)f}\\ 
	&\qquad\leq \norm{\Pp_t(tA_1\conn)\Ppb_t f } + \norm{\Pp_t(t\divv A_2)\Ppb_t f} + \norm{\Pp_t(tA_3)\Ppb_t f} \\
	&\qquad\qquad+\norm{\Qq_t(tA_1\conn)\Qqb_t f} + \norm{\Qq_t(t\divv A_2)\Qqb_t f} + \norm{\Qq_t(tA_3)\Qqb_t f}, 
\end{aligned}
\end{equation}

\begin{proposition}
\label{Prop:BenignTerm}
We obtain the estimates
$$\sup_{t \in (0,1]} \norm{\Qq_t - \Qqb_t} \lesssim \norm{A}_\infty, 
\quad
\sup_{t \in (0,1]} \norm{\Rr_t - \Rrb_t} \lesssim \norm{A}_\infty,$$
where the implicit constants depend on $\const(\cM,\cV,\Dirb,\Dirp)$.
\end{proposition}
\begin{proof}
First, we bound the terms with $\Pp_t$ and $\Ppb_t$.
Note that, 
$$\norm{\Pp_t(t A_1\conn)\Ppb_t}
	\leq (\sup_{t \in (0,1]} \norm{\Pp_t}) 
		\norm{A_1}_\infty \norm{t\conn \Ppb_t}.$$
Moreover, by \eqref{Def:DomConst}, 
$$
\norm{t\conn \Ppb_t} \leq C_{\Dir} ( \norm{t\Dirb \Ppb_t}  +\norm{t\Ppb_t})
	\leq \const (1 + t).$$
On combining this with the assumption that 
$\norm{A_1}_\infty \leq \norm{A}_\infty $, we obtain that
$\norm{\Pp_t(tA_1\conn)\Ppb_t} \leq \const \norm{A}_\infty (1 + t)$.

Next, we estimate 
$\norm{\Pp_t(t\divv A_2)\Ppb_t}$.
First, we note that, for $v \in \dom(\divv)$,
$$\norm{\Pp_t (t\divv)v} \Rd
	= \sup_{\norm{g}=1} \modulus{\inprod{\Pp_t(t\divv)v, g}}
	= \sup_{\norm{g}=1} \modulus{\inprod{v, t\adj{\divv} \Pp_t g}}
	\leq \sup_{\norm{g}=1} \norm{v}
		\norm{t \adj{\divv} \Pp_tg}.$$
Now, note that $\adj{\divv} = -\conn$ and 
on invoking \eqref{Def:DomConst},
\begin{equation*}
\norm{t \adj{\divv} \Pp_tg} 
	\leq \const (\norm{t \Dirp \Pp_t g} + \norm{t \Pp_t g})
	\leq \const (1 + t) \norm{g}.
\end{equation*}
Thus, $\norm{\Pp_t(t\divv)v} \leq 2 \const \norm{v}$
and since $\dom(\divv)$ is dense in $\Lp{2}(\cotanb\cM \tensor \cV)$,
we obtain that
$\Pp_t(t\divv)$ extends to a bounded operator, 
uniformly bounded in $t \in (0,1]$.
Thus,
$$
\norm{\Pp_t(t\divv A_2)\Ppb_t} 
	\leq \norm{\Pp_t(t\divv)} \Rd \norm{A_2}_\infty \Bk \norm{\Ppb_t} 
	\leq \const \norm{A}_\infty. $$
It is immediate that 
$\norm{\Pp_tA_3 \Ppb_t} \leq \norm{\Pp_t} \Rd \norm{A_3}_\infty \Bk \norm{\Ppb_t} \leq \norm{A}_\infty$.

Similar bounds for $\Qq_t$ and $\Qqb_t$
in place of $\Pp_t$ and $\Ppb_t$
follow by exactly the same arguments
noting that $\norm{t\conn \Qqb_t} \simeq \norm{\iden - \Ppb_t}$. 
This shows that $\sup_{t \in (0,1]} \norm{\Qq_t - \Qqb_t} \lesssim \norm{A}_\infty$.
To show  $\sup_{t \in (0,1]} \norm{\Rr_t - \Rrb_t} \lesssim \norm{A}_\infty$,
we note that it suffices to simply verify that the previous
argument holds for $\Rr_t$ and $\Rrb_t$ in place
of $\Pp_t$ and $\Ppb_t$ due to the formula
\Rd established in \Bk Proposition \ref{Prop:Paraprod}. 
\end{proof}

We note that a similar estimate of $P_t$ also hold, but we shall
not need that.

\subsection{First reduction}

Now, let $f \in \Hol^\infty(\OSec{\omega,\sigma})$,
for $\omega \in (0, \pi/2)$ and $\sigma \in (0, \infty)$.
We reduce estimating 
$\norm{f(\Dirp) - f(\Dirb)}$ to obtaining
an appropriate estimate for $\norm{\Qq_t - \Qqb_t}$.
To that end, we begin with the following lemma. 

\begin{lemma}
The following identities hold:
$$ \iden = \Pp_1 + 2 \int_{0}^1 \Qq_s^2\ \dtt[s] = \Ppb_1 + 2 \int_{0}^1 \Qqb_s^2\ \dtt[s],$$
where $\Pp_1 = (\iden + \Dirp^2)^{-1}$ and $\Ppb_1 = (\iden + \Dirb^2)^{-1}$.
\end{lemma}
\begin{proof}
Note that,
 $$ \iden - \Ppb_1 = \iden - (\iden + \Dirb^2)^{-1} = \Dirb^2(\iden + \Dirb^2)^{-1}.$$  
Moreover,
$$\ddt{s}\cbrac{ \frac{s^2}{1 + s^2}} =  \frac{2s}{(1+s^2)^2}$$
and by setting $s = tz$, we have that
$$\int_{0}^1 \frac{(tz)^2}{(1 + (tz)^2)^{2}}\ \dtt 
= \int_{0}^z \frac{s^2}{(1+s^2)^2}\ \dtt[s] = \frac{1}{2}\ \frac{z^2}{1 + z^2}.$$
Thus, by the functional calculus we obtain that
$$\Dirb^2 (\iden + \Dirb^2)^{-1}u = 2 \int_{0}^1 \psi_t(\Dirb)^2u\ \dtt.$$
The calculation for $\Dirp^2(\iden + \Dirp^2)^{-1}$ is similar.
\end{proof}

With the aid of this lemma, we obtain
\begin{equation}
\label{Eq:FR1}
\begin{aligned}
f(\Dirp) - f(\Dirb) &= [\Pp_1 + (\iden - \Pp_1)]f(\Dirp)[\Pp_1 + (\iden - \Pp_1)] \\
	&\qquad\qquad- [\Ppb_1 + (\iden - \Ppb_1)]f(\Dirb)[\Ppb_1 + (\iden - \Ppb_1)] \\
	&= [(2\Pp_1 - \Pp_1^2)f(\Dirp) - (2\Ppb_1 - \Ppb_1^2)f(\Dirb)]  \\
	&\qquad\qquad+ 4\int_0^1 \int_0^1 [(\psi_s^2 f\psi_t^2)(\Dirp) - (\psi_s^2 f \psi_t^2)(\Dirb)]\ \dtt[s] \dtt.
\end{aligned}
\end{equation}

Consider the second term on the right. Using the
fact that the functional calculus is a homomorphism yields that
\begin{equation}
\label{Eq:FR2} 
\begin{aligned} 
(\psi_s^2 f \psi_t^2)(\Dirp) - (\psi_s^2 f \psi_t^2)(\Dirb)
	&= \psi_s(\Dirp)(\psi_s f \psi_t)(\Dirp)[\psi_t(\Dirp) - \psi_t(\Dirb)] \\ 
	&\quad\quad+ \psi_s(\Dirp)[(\psi_s f \psi_t)(\Dirp) - (\psi_s f\psi_t)(\Dirb)] \psi_t(\Dirb) \\
	&\quad\quad+ [\psi_s(\Dirp) - \psi_s(\Dirb)](\psi_s f\psi_t)(\Dirb) \psi_t(\Dirb).
\end{aligned} 
\end{equation}

Let $\eta(x) = \min\set{x, \frac{1}{x}} (1 + \modulus{\log\modulus{x}})$. Then, 
we have the following preliminary estimates for each of the three terms
appearing in \eqref{Eq:FR2}. 
\begin{lemma}
\label{Lem:Schur}
The following estimates hold:
\begin{multline*} 
\norm{(\psi_s f \psi_t)(\Dirp)} \lesssim \norm{f}_\infty \eta(s/t),
\quad
\norm{(\psi_s f \psi_t)(\Dirb)} \lesssim \norm{f}_\infty \eta(s/t),
\text{and}\quad \\
\norm{ (\psi_sf\psi_t)(\Dirp) - (\psi_s f\psi_t)(\Dirb)} \lesssim \norm{f}_\infty \norm{A}_\infty \eta(s/t),
\end{multline*}
where the implicit constants only depend on $\const(\cM,\cV,\Dirb, \Dirp)$.
\end{lemma}
\begin{proof}
The bound for the first two terms
follows directly from the norm estimate of the Riesz-Dunford
integral \eqref{Eq:RD}.
For the last estimate, we have 
that, after fixing an appropriate curve $\gamma$, 
\begin{align*}
&\norm{(\psi_sf\psi_t)(\Dirp) - (\psi_s f\psi_t)(\Dirb)} 
	\lesssim \oint_{\gamma} \norm{ (\psi_s f \psi_t)(\zeta) (\rs{\Dirp}(\zeta) - \rs{\Dirb}(\zeta))}\ \modulus{d\zeta} \\
	&\qquad\qquad\lesssim \norm{f}_\infty \eta(s/t) 
	\cbrac{\oint_{\gamma} \norm{\psi_s f \psi_t {\psi}(\zeta)}\ \frac{\modulus{d\zeta}}{\modulus{\zeta}} } 
		\sup_{\zeta \in \gamma} \cbrac{\norm{\rs{\Dirp}(\zeta) - \rs{\Dirb}(\zeta)}\ \modulus{\zeta}} \\
	&\qquad\qquad\lesssim \norm{f}_\infty \norm{A}_\infty \eta(s/t),
\end{align*}
where the penultimate inequality follows from the decay of $\psi_s f \psi_t$
and from Proposition \ref{Prop:BenignTerm}.
\end{proof}

\begin{proposition}
\label{Prop:FirstRed}
Suppose that
$$ \int_{0}^1 \norm{(\Qq_t - \Qqb_t)u}^2\ \dtt \leq C_0 \norm{A}_\infty^2 \norm{u}^2$$
for all $u \in \Lp{2}(\cV)$. 
Then,
$$ \norm{f(\Dirp) - f(\Dirb)} \lesssim \norm{A}_{\infty} \norm{f}_\infty,$$
where the implicit constant depends only on $\const(\cM,\cV,\Dirb,\Dirp)$
and $C_0$. 
\end{proposition}
\begin{proof}
We \Rd appeal to \eqref{Eq:FR1} and \Bk first prove that
$$\norm{(2\Pp_1 - \Pp_1^2)f(\Dirp) - (2\Ppb_1 - \Ppb_1^2)f(\Dirb)}
	\lesssim \norm{f}_{\infty} \norm{A}_\infty.$$
To that end, define 
$$\phi(\zeta) = \cbrac{\frac{2}{1 + \zeta^2} - \frac{1}{(1 + \zeta^2)^2}}f(\zeta)$$
and note that $\phi \in \Psi(\OSec{\omega,\sigma})$.
Moreover, by the functional calculus, we have
$[(2\Pp_1 - \Pp_1^2)f(\Dirp) - (2\Ppb_1 - \Ppb_1^2)f(\Dirp)] = \phi(\Dirp) - \phi(\Dirb)$.
Then, for an appropriate chosen curve $\gamma$, 
\begin{align*}
\norm{\phi(\Dirp)u - \phi(\Dirb)u}
	&\lesssim \norm{f}_{\infty} \oint_{\gamma} \modulus{\phi(\zeta)} 
		\norm{\rs{\Dirp}(\zeta)(\Dirb - \Dirp)\rs{\Dirb}(\zeta)u}\ \modulus{d\zeta}  \\
	&\lesssim \norm{f}_{\infty} \norm{A}_\infty  \norm{u}\ 
	\cbrac{\oint_{\gamma}\modulus{\phi(\zeta)}}\ \frac{\modulus{d\zeta}}{\modulus{\zeta}}
	\lesssim \norm{f}_\infty \norm{A}_\infty \norm{u}
\end{align*}
where the first inequality  follows from Proposition \ref{Prop:BenignTerm}.

Now, to bound \Rd the second term of \eqref{Eq:FR1},\Bk we appeal to
\eqref{Eq:FR2}. As we have previously
noted, $\psi_t(\Dirb) = \Qqb_t$
and $\psi_t(\Dirp) = \Qq_t$,
and so, 
\begin{align*} 
&\norm{\psi_s(\Dirp)(\psi_s f \psi_t)(\Dirp)[\psi_t(\Dirp) - \psi_t(\Dirb)]} \\
	&\qquad\qquad= \sup_{\norm{u} = \norm{v} = 1} \modulus{\inprod{\psi_s(\Dirp)(\psi_s f \psi_t)(\Dirp)[\psi_t(\Dirp) - \psi_t(\Dirb)]u,v}} \\
	&\qquad\qquad= \sup_{\norm{u} = \norm{v} = 1} \modulus{\inprod{(\psi_s f \psi_t)(\Dirp) (\Qq_t - \Qqb_t)u, \Qq_s v}}.
\end{align*}
Fix $\norm{u} = \norm{v} = 1$, and we compute 
\begin{multline*}
\modulus{\inprod{(\psi_s f \psi_t)(\Dirp) (\Qq_t - \Qqb_t)u, \Qq_s v}}
	\lesssim \norm{\psi_s f \psi_t(\Dirp) (\Qq_t - \Qqb_t)u} \norm{\Qq_s v} \\
	\lesssim \norm{f}_\infty \eta(s/t) \norm{(\Qq_t - \Qqb_t)u} \norm{\Qq_s v}. 
\end{multline*}  
Thus, 
\begin{align*}
\int_{0}^1 \int_0^1 &\modulus{\inprod{(\psi_s f \psi_t)(\Dirp) (\Qq_t - \Qqb_t)u, \Qq_s v}}\ \dtt[s] \dtt \\
	&\lesssim \norm{f}_\infty \cbrac{ \int_0^1 \cbrac{\int_0^1 \eta(s/t) \norm{(\Qq_t - \Qqb_t)u}^2 \dtt[s]} \dtt}^{\frac{1}{2}} \times \\ 
	&\qquad\qquad\qquad\cbrac{\int_0^1 \int_0^1 \eta(s/t) \norm{\Qq_s v}^2\ \dtt[s] \dtt}^{\frac{1}{2}} \\
	&\lesssim \norm{f}_\infty \cbrac{\int_{0}^1 \norm{(\Qq_t - \Qqb_t)u}^2 \dtt}^{\frac{1}{2}}
		 \cbrac{\int_{0}^1 \norm{\Qq_s v}^2 \dtt[s]}^{\frac{1}{2}} \\
	&\lesssim \norm{f}_\infty \norm{A}_\infty \norm{u} \norm{v},
\end{align*}
\Rd where the last inequality follows via our hypothesis and the self-adjointness
of $\Dirp$. This bounds the first term of \eqref{Eq:FR2}.
For the second term, \Bk we note that by using duality
to compute the norm, we arrive at:
$$\modulus{\inprod{[(\psi_s f \psi_t)(\Dirp) - (\psi_s f \psi_t)(\Dirb)] \Qqb_t u, \Qq_s v}}
	\lesssim  \norm{A}_\infty \norm{f}_\infty \eta(s/t) \norm{\Qqb_t u} \norm{\Qq_s v},$$
where we have used Lemma \ref{Lem:Schur}. By a similar computation to the 
previous integral, we obtain that 
$$\int_0^1 \int_0^1 \modulus{\inprod{[(\psi_s f \psi_t)(\Dirp) - (\psi_s f \psi_t)(\Dirb)] \Qqb_t u, \Qq_s v}}\ \dtt[s] \dtt 
	\lesssim \norm{A}_\infty \norm{f}_\infty \norm{u} \norm{v}.$$

The last term in \eqref{Eq:FR2} is argued similar to the first term.
Combining these estimates together, we obtain that
$ \norm{f(\Dirp) - f(\Dirb)} \lesssim \norm{A}_{\infty} \norm{f}_\infty$
as claimed. 
\end{proof}

\subsection{Second reduction}

In this section,
we show that the
quadratic estimate 
$$
\int_0^1 \norm{(\Qq_t - \Qqb_t)u}^2\ \dtt \lesssim \norm{A}_{\infty}^2 \norm{u}^2$$ 
can be reduced to quadratic estimates of 
the form 
$$
\int_{0}^1 \norm{\QQ_t S\PP_t u}^2\ \dtt \lesssim \norm{A}_{\infty}^2 \norm{u}^2,$$
where the operator $\QQ_t$ is
an operator satisfying quadratic estimates,
where $\PP_t$ is either $\Pp_t$ or $\Ppb_t$, 
and $S$ is an appropriate bounded operator
with norm controlled by $\const(\cM,\cV,\Dirb, \Dirp)$.
Due to Proposition \ref{Prop:Paraprod}, \Rd
via the decomposition of the difference 
$\Dirb - \Dirp = A_1 \conn + \divv A_2 + A_3$,
it is clear how the  term $\norm{A}_\infty$ arise
in the expression as we note in the following:  \Bk 
\begin{equation}
\begin{aligned} 
&\cbrac{\int_{0}^1 \norm{(\Qq_t - \Qqb_t)f}^2\ \dtt}^{\frac{1}{2}}\\ 
	&\qquad\leq \cbrac{\int_0^1 \norm{\Pp_t tA_1\conn \Ppb_t f }^2\ \dtt}^{\frac{1}{2}}
		+ \cbrac{\int_0^1 \norm{\Pp_t t\divv A_2 \Ppb_t f}^2\ \dtt}^{\frac{1}{2}} \\
		&\qquad\qquad\qquad\qquad\qquad+ \cbrac{\int_0^1 \norm{\Pp_t tA_3 \Ppb_t f}^2\ \dtt}^{\frac{1}{2}} \\
		&\qquad\qquad+
		\cbrac{\int_0^1 \norm{\Qq_t tA_1\conn \Qqb_t f}^2\ \dtt}^{\frac{1}{2}}
\label{Eqn:PsiToPtQt2}
		+ \cbrac{\int_0^1 \norm{\Qq_t t\divv A_2 \Qqb_t f}^2\ \dtt}^{\frac{1}{2}}\\ 
		&\qquad\qquad\qquad\qquad\qquad+ \cbrac{\norm{\Qq_t tA_3 \Qqb_t f}^2\ \dtt}^{\frac{1}{2}}. 
\end{aligned}
\end{equation}

\Rd With this, we obtain the following. \Bk

\begin{proposition}
\label{Prop:FinalRed}
Suppose that 
\begin{align*}
&\int_{0}^1 \norm{\Qq_t A_1 \conn (\imath\iden + \Dir)^{-1} \Ppb_t f}^2\ \dtt 
	\leq C_1  \norm{A}_{\infty}^2 \norm{f}^2,
\ \text{and} \ \\
&\int_{0}^1 \norm{t\Pp_t\divv A_2 \Ppb_tf}^2\ \dtt 
	\leq C_2 \norm{A}_{\infty}^2 \norm{f}^2
\end{align*}
for all $u \in \Lp{2}(\cV)$.
Then, for $\omega \in (0, \pi/2)$ and $\sigma \in (0, \infty)$,
whenever $f \in \Hol^\infty(\OSec{\omega,\sigma})$,
we obtain that 
$$\norm{f(\Dirp) - f(\Dirb)} \lesssim \norm{f}_\infty \norm{A}_\infty$$
where the implicit constant
depends  on $C_1,\ C_2$ and $\const(\cM,\cV,\Dirb,\Dirp)$.
\end{proposition}
\begin{proof}
We demonstrate that each term to the right of \eqref{Eqn:PsiToPtQt2}
is bounded by  
$$\max\set{C_1, C_2} \norm{A}_{\infty}^2 $$
and apply  Proposition \ref{Prop:FirstRed}.
First note that,
$$\int_0^1 \norm{\Pp_t(tA_3)\Ppb_tf}^2\ \dtt 
	\leq  \norm{A}_{\infty}^2 \int_0^1 t^2 \norm{f}^2\ \dtt 
	\leq  \norm{A}_{\infty}^2 \norm{f}^2,$$
and by the same calculation with $\Qq_t$ and
$\Qqb_t$ in place of $\Pp_t$ and $\Ppb_t$,
$\int_0^1 \norm{\Qq_t(tA_3)\Qqb_tf}^2\ \dtt \leq  \norm{A}_{\infty}^2 \norm{f}^2.$

By \eqref{Def:DomConst}  
and using the quadratic estimates for $\Qqb_t$, 
\begin{align*}
\int_0^1 \norm{\Pp_t(t A_1 \conn) \Ppb_tf}^2\ \dtt
	&\leq \norm{A_1}^2_\infty \int_0^1 \norm{t \conn\Ppb_t f}^2\ \dtt \\
	&\leq 2 \const^2 \norm{A}_{\infty}^2 \int_0^1(\norm{t \Dirb\Ppb_t f}^2 + \norm{t\Ppb_t f}^2)\ \dtt  \\
	&\leq 2 \const^2 \norm{A}_{\infty}^2 \int_0^1 (\norm{\Qqb_t f}^2 + t^2 \norm{f}^2)\ \dtt 
	\leq  \const^2 \norm{A}_{\infty}^2 \norm{f}^2.
\end{align*}

Next, note that for $u \in \dom(\divv)$, 
\begin{multline*}
\norm{\Qq_t t\divv u} 
	= \sup_{\norm{g}=1} \inprod{\Qq_t t \divv u, g}
	\leq \sup_{\norm{g}=1} \norm{u} \norm{t \adj{\divv} \Qq_tg} \\
	\leq \const \norm{u} \sup_{\norm{g}=1} (\norm{t\Dirp \Qq_t g} + \norm{t \Qq_t g})
	\lesssim \const \norm{u}.
\end{multline*}
Therefore
$$\int_0^1 \norm{\Qq_t(t \divv A_2 ) \Qqb_tf}^2\ \dtt
	\leq \const^2 \norm{A_2}^2 \int_0^1 \norm{\Qqb_t f}^2\ \dtt 
	\leq \const^2  \norm{A}_{\infty}^2 \norm{f}^2.$$

The two remaining terms are then handled
via the hypothesis. The first term is immediate. 
For the second term, first we have that
\begin{align*}
&\cbrac{\int_0^1 \norm{\Qq_t (t A_1 \conn)\Qqb_t f}^2\ \dtt}^{\frac{1}{2}}
	= \cbrac{\int_0^1 \norm{\Qq_t A_1 \conn (\imath\iden + \Dirb)^{-1} (t (\imath\iden + \Dirb) \Qqb_t)f}^2\ \dtt}^{\frac{1}{2}} \\ 
	&\qquad\leq \cbrac{\int_0^1 \norm{\Qq_t A_1 \conn (\imath\iden + \Dirb)^{-1} f}^2\ \dtt}^{\frac{1}{2}} 
		+ \cbrac{\int_0^1 \norm{\Qq_t A_1 \conn (\imath\iden + \Dirb)^{-1} \Ppb_t f}^2\ \dtt}^{\frac{1}{2}} \\
		&\qquad\qquad\qquad+ \cbrac{\int_0^1 \norm{\Qq_t A_1 \conn (\imath\iden + \Dirb)^{-1}t\Qqb_tf}^2\ \dtt}^{\frac{1}{2}},
\end{align*}
since $t\Dir \Qqb_t = \iden - \Ppb_t$.
By hypothesis, 
$$\int_0^1 \norm{\Qq_t A_1 \conn (\imath\iden + \Dirb)^{-1} \Ppb_t f}^2\ \dtt \leq C_2 \norm{A}_{\infty}^2 \norm{f}^2,$$
and by the quadratic estimates for $\Qq_t$, \eqref{Def:DomConst} and
noting that   $\norm{\conn(\imath \iden + \Dirb)^{-1}u}   \lesssim   \norm{u}$,   
$$
\int_0^1 \norm{\Qq_t A_1 \conn (\imath\iden + \Dirb)^{-1} f}^2\ \dtt 
	\leq \norm{A_1}^2_\infty \norm{\conn (\imath\iden + \Dirb)^{-1}}^2 \norm{f}^2
	\lesssim   \norm{A}_{\infty}^2 \norm{f}^2.$$
For the last term,
$$\int_0^1 \norm{\Qq_t A_1 \conn (\imath\iden + \Dirb)^{-1}(t\Qqb_t)f}^2\ \dtt 
	\lesssim \int_{0}^1 \norm{A_1}_\infty^2 t^2 \norm{f}^2\ \dtt \leq  \norm{A}_{\infty}^2 \norm{f}^2.$$
This finishes the proof.
\end{proof}

We conclude this section by remarking that 
in typical applications,
as we will see in \S\ref{Sec:SFE},
the constants $C_1$ and $C_2$ themselves will
depend on $\const(\cM,\cV,\Dirb, \Dirp)$.
\section{Quadratic estimates}
\label{Sec:SFE}

In this section, 
we prove the quadratic estimates
in the hypothesis of Proposition \ref{Prop:FinalRed}.
We consider both quadratic estimates
 appearing as the hypothesis of this proposition 
 combined into  the general form   
\begin{equation}
\label{Eq:SFE}
\int_0^1 \norm{\QQ_tS\Ppb_t f}^2\ \dtt \lesssim \norm{A}_{\infty}^2 \norm{f}^2,
\end{equation}
where $S: \Lp{2}(\cV) \to \Lp{2}(\cW)$
and $\QQ_t: \Lp{2}(\cW) \to \Lp{2}(\cV)$, 
where $\cW$ is an auxiliary vector bundle
 and $\QQ_t$ is a family of operators with sufficient
decay. 

\Rd It is well known in harmonic analysis, going back to the counter example in 
\cite{Mc72} by the second author to the abstract Kato square root conjecture, 
that estimates of the form \eqref{Eq:SFE}, even for multipliers $S$, 
cannot be proved only using operator theory methods such as those in \S\ref{Sec:Red}. Instead 
one needs to apply harmonic analysis to exploit the differential structure of the operators and the space. 
It is here that we require the full list \ref{Hyp:First}-\ref{Hyp:Last} of assumptions. \Bk

The purpose of considering an abstract estimate of this
form is \Rd due to the fact that to  \Bk satisfy the hypothesis
of Proposition \ref{Prop:FinalRed}, we are required to prove two
different quadratic estimates with the choice of operators
$S = \iden$ for $\QQ_t = \Pp_t \divv A_2$ and $S = \conn(\imath \iden + \Dir)^{-1}$
for $\QQ_t = \Qq_t A_1$.
\Rd Therefore, 
in order to make the presentation clearer for the reader, 
we combine these two estimates
into a single estimate. 
Note that while it may seem that the first choice for $\QQ_t$ and $S$
is an easy estimate, the fact that the operator $\Ppb_t$
appears in the required quadratic estimate to the right of $\QQ_t$ 
precisely means that this estimate that cannot be handled
by operator theory methods alone. 

In what will follow, the key is to reduce the  estimate \eqref{Eq:SFE} \Bk
to a \emph{Carleson measure estimate}.
We will impose further restrictions on $S$ as required  in the analysis that will follow.

\subsection{Dyadic grids and  GBG frames}
\label{Sec:DyaGrid}

A central consequence of the growth
assumption \eqref{Def:Eloc} 
is that it  affords us with a  dyadic decomposition. 
This is illustrated in the following theorem.

\begin{theorem}[Existence of a truncated dyadic structure]
\label{Thm:Dya:Christ}
Suppose that $(\cM,\mg)$ satisfies \eqref{Def:Eloc}.
Then, there  exist countably  many index sets $I_k$,
a countable collection of open subsets
$\set{\Q[\alpha]^k \subset \cM: \alpha \in I_k,\ k \in \Na}$,
points $z_\alpha^k \in \Q[\alpha]^k$ (called the \emph{centre} of $\Q[\alpha]^k$), 
and constants $\delta \in (0,1)$, 
$a_0 > 0$, $\eta > 0$ and $C_1, C_2 < \infty$ satisfying:
\begin{enumerate}[(i)]
\item for all $k \in \Na$,
	$\mu(\cM \setminus \union_{\alpha} \Q[\alpha]^k) = 0$,
\item if $l \geq k$, then either $\Q[\beta]^l \subset \Q[\alpha]^k$
	or $\Q[\beta]^l \intersect \Q[\alpha]^k = \emptyset$,
\item for each $(k,\alpha)$ and each $l < k$
	there exists a unique $\beta$ such that
	$\Q[\alpha]^k \subset \Q[\beta]^l$,
\item $\diam \Q[\alpha]^k < C_1 \delta^k$, 
\item $ \Ball(z_\alpha^k, a_0 \delta^k) \subset \Q[\alpha]^k$,  
\item for all $k, \alpha$ and for all $t > 0$, 
	$\mu\set{x \in \Q[\alpha]^k: d(x, \cM\setminus\Q[\alpha]^k) \leq t \delta^k} \leq C_2 t^\eta \mu(\Q[\alpha]^k).$
\end{enumerate} 
\end{theorem}

This theorem was first proved by Christ  in \cite{Christ} for $k\in\In$ 
(i.e. untruncated) for doubling measure metric spaces. It was
generalised by Morris in \cite{Morris3} to 
our particular setting. 

 In what is to follow
we  couple this 
dyadic grid with the notion of GBG for the
vector bundle $(\cV,\mh)$.
We encourage the reader
to assume familiarity with the constants 
$C_1,\ a_0$ and $\delta$ from Theorem \ref{Thm:Dya:Christ}.
 We remark that terminology we define below 
first arose in the harmonic analysis of the Kato square
root problem on vector bundles in \cite{BMc}.

\Rd We define and note the following:
\begin{equation} 
\label{Eq:HConsts}
\begin{aligned} 
&\bullet \text{fix $\jscale \in \Na$ such that $C_1 \delta^{\jscale} \leq \brad/5$ where $\brad$ is from Definition \eqref{Def:GBG}},\\ 
&\bullet \text{let $\scale = \delta^J$ which we call the \emph{scale}},\\ 
&\bullet \text{whenever $j \geq \jscale$, $\DyQ^j$ denotes the set of cubes $\Q[\alpha]^j$},\\ 
&\bullet \text{define $\DyQ = \union_{j \geq \jscale} \DyQ^j$},\\
&\bullet \text{whenever  $t\leq \scale$, we define $\DyQ_t = \DyQ^j$ if $\delta^{j+1} < t \leq \delta^j$,}\\
&\bullet \text{the length of a cube $\Q \in \DyQ^j$ is $\len(\Q) = \delta^j$,}\\  
&\bullet \text{for any $\Q \in \DyQ^j$,
there exists a unique ancestor
cube $\ancester{Q} \in \DyQ^\jscale$}\\
&\quad \text{such that $Q \subset \ancester{Q}$, and the cube $\ancester{Q}$ is
called  the \emph{GBG cube} of $\Q$}.
\end{aligned}
\end{equation}
\Bk

The following notion allows us to couple 
the dyadic structure with the GBG condition
yielding ``good'' coordinates for $\cV$ 
that enable us to import tools from 
Euclidean harmonic analysis to the vector bundle setting.
In the following definition, for a cube $\Q = \Q[\alpha]^j \in \DyQ^j$,
we define $x_{\Q} = z_{\alpha}^j$
and call this the \emph{centre} of the cube.

\begin{definition}
We call the following system of GBG trivialisations
$$
\sC = \set{\psi: \Ball(x_{\Q}, \brad) \times \C^N \to \pi^{-1}_{\cV}(\Ball(x_{\Q},\brad)),\ \Q \in \DyQ^J}$$
the \emph{GBG coordinates}.
Moreover, we let 
$$\sC_{\jscale} = \set{\psi\rest{\Q}: \Q \times \C^N \to \pi^{-1}_\cV(\Q),\ \psi \in \sC}$$
which we call the \emph{dyadic GBG coordinates}. 
For an arbitrary cube $\Q \in \DyQ$, 
the GBG coordinates of $\Q$ are
the GBG coordinates of the GBG cube $\ancester{Q}$. 
\end{definition}

An important tool in harmonic analysis
is to be able to perform averages, which requires
a notion of integration. In a general vector bundle, 
this is not a well-defined notion under transformations.
However, by using the GBG structure, we define 
the notion of \emph{cube integration}, as a map 
$\Ball(x_{\ancester{\Q}},\brad) \times \DyQ \ni (x, \Q) \mapsto (\int_{\Q} \mdot )(x)$.
For $u \in \Lp[loc]{1}(\cV)$, and $y \in \Ball(x_{\ancester{\Q}},\brad)$ we write
$$
\cbrac{\int_{Q} u\ d\mu}(y) = \cbrac{\int_{\Q} u_i\ d\mu}\  e^i(y)$$
where $u = u_i e^i$ in the GBG coordinates of $\Q$.
Note that this integral is only defined in 
$\Ball(x_{\ancester{\Q}}, \brad)$.
We then define the \emph{cube average} $u_{\Q} \in \Lp{\infty}(\cV)$
of some $u \in \Lp[loc]{1}(\cV)$ as
as 
$$u_{\Q}(y) = \begin{cases} 
	\fint_{\Q} u\ d\mu &y \in \Ball(x_{\ancester{\Q}},\brad)\\
	0 	&y \not\in \Ball(x_{\ancester{\Q}},\brad).
	\end{cases}$$

Lastly, for each $t > 0$, we define the \emph{dyadic averaging operator}
$\Av_t: \Lp[loc]{1}(\cV) \to \Lp[loc]{1}(\cV)$ by
\begin{equation} 
\label{Eq:DyAv} 
\Av_tu(x) = \cbrac{\fint_{\Q} u\ d\mu}(x)
\end{equation} 
where $\Q \in \DyQ_t$ and $x \in \Q$. 
This defines $\Av_tu(x)$ for $x$-a.e. on $\cM$. 
We remark that this operator is well defined, and that
$\Av_tu(x)$ on each $\Q \in \DyQ_t$. Moreover, 
$\Av_t: \Lp{2}(\cV) \to \Lp{2}(\cV)$ is bounded
 uniformly for $t \leq \scale$ with the bound depending
on the constant $C$ arising in the GBG criterion.

\subsection{Harmonic analysis}
\label{Sec:HarmAnal}

Let us assume that $\cV$ and $\cW$ are two
vector bundles both satisfying the GBG condition
and on taking a minimum of the GBG radius of the two bundles, 
assume that $\cV$ and $\cW$ share the same GBG
radius. 
Let $\QQ_t: \Lp{2}(\cW) \to \Lp{2}(\cV)$ be a
family of operators uniformly bounded in  $t \in (0,1]$.
The $\QQ_t$ we consider will naturally 
contain the coefficients $A_i$ as a factor. 

On defining $\maxx{a} = \max\set{1,a}$, 
we assume that $\QQ_t$ satisfies
\emph{off-diagonal estimates}:
there exists $C_{\QQ} > 0$
such that, for each $M > 0$, there exists a constant $C_{\Delta, M} > 0$
satisfying:
\begin{equation}
\label{Def:OD}
\begin{aligned}
\norm{\ch{E} \QQ_t(\ch{F}u)}_{\Lp{2}(\cV)}
	\leq  C_{\Delta,M} \norm{A}_{\infty}^2 &\maxx{\frac{\met(E,F)}{t}}^{-M} \\ 
		&\qquad\exp\cbrac{-C_{\QQ} \frac{\met(E,F)}{t}} \norm{\ch{F}u}_{\Lp{2}(\cW)}  
\end{aligned}
\end{equation}
for every Borel set $E,\ F \subset \cM$
and $u \in \Lp{2}(\cW)$.
Moreover, we assume that $\QQ_t$ satisfies
quadratic estimates, by which we mean 
there exists $C_{\QQ}' > 0$ so that 
\begin{equation}
\label{Def:SFE} 
\int_{0}^1 \norm{\QQ_t f}^2\ \dtt \leq C_{\QQ}' \norm{A}_\infty^2 \norm{f}^2
\end{equation}
for all $f \in \Lp{2}(\cV)$. 

Recalling the constants $c_E$ and $\kappa$
appearing in \eqref{Def:Eloc},
Lemma 4.4 in \cite{Morris3} states that, 
whenever $M > \kappa$ and $m > c_E/t$,
we have 
\begin{equation} 
\label{Eq:CubeSum}
\sup_{\Q' \in \DyQ_t} \sum_{\Q \in \DyQ_t} \frac{\mu(\Q)}{\mu(\Q')} 
	\maxx{\frac{\met(\Q,\Q')}{t}}^{-M} \exp\cbrac{-m \frac{\met(\Q,\Q')}{t}} \lesssim 1.
\end{equation}

As a consequence, arguing exactly as in Lemma 5.3 in \cite{Morris3}, 
we obtain that $\QQ_t$ extends to a bounded operator
$
\QQ_t: \Lp{\infty}(\cW) \to \Lp[loc]{2}(\cV)$
with $c > 0$ such that 
\begin{equation}
\label{Eq:QQbd}
\norm{\QQ_t u}_{\Lp{2}(\Q;\cV)}^2 \leq c \norm{A}_{\infty}^2 \mu(\Q)\norm{u}^2_{\Lp{\infty}(\cW)},
\end{equation}
whenever $t \in (0,\hscale(\QQ)]$, where 
$$\hscale(\QQ) = \min\set{\scale, \maxx{2 \maxx{\delta/C_1}^{-1} c_E/ C_{\QQ}}^{-1}}$$
we call the \emph{harmonic analysis scale of $\QQ_t$}. 

In the harmonic analysis, 
\emph{constant functions} are often required to extract
\emph{principal parts} of operators.
Under the guise of the 
GBG coordinate system, 
we are able to define a notion of a
constant section, locally,  of $\cV$. 
Let $x \in \Q \in \DyQ$ and  $w \in \cV_x \cong \C^N$, and
write $w = w_i\ e^i(x)$ in the GBG
frame $\set{e^i(x)}$ associated to $\Q$.
We then define the \emph{constant extension} 
of $w$ by
\begin{equation} 
\label{Eq:ConstExt} 
w^{c}(y) =   
	\begin{cases} w_i\ e^i(y) &y \in \Ball(x_{\ancester{Q}},\brad)\\
			0 &y \not\in \Ball(x_{\ancester{Q}},\brad),
	\end{cases}
\end{equation}
and we note that $w^{c} \in \Lp{\infty}(\cV)$.

For $x \in \Q \in \DyQ$, and $w \in \cV_x$, 
with GBG constant extension $w^{c} \in \Lp{\infty}(\cV)$, 
we define the \emph{principal part} of $\QQ_t$
by
\begin{equation}
\label{Eq:GammaQ}
\Pri^{\QQ}_t(x)w = (\QQ_tw^{c})(x).
\end{equation} 
It is easy to see that the principal part is 
a well defined operator 
$\Pri^{\QQ}_t(x): \cW_x \to \cV_x$ for
almost-every $x \in \cM$.
For convenience, we often write $\Pri_t$ instead of $\Pri^{\QQ}_t$.

We note that as a
consequence of \eqref{Eq:QQbd} that 
\begin{equation}
\label{Eq:EtAtbdd}
\fint_{Q} \modulus{\Pri_t(x)}^2\ d\mu(x) \leq \norm{A}_{\infty}^2
\quad\text{and}\quad
\sup_{t \in (0,\hscale(\QQ)]} \norm{\Pri_t \Av_t} \lesssim \norm{A}_{\infty}.
\end{equation}
for all $t \in (0, \hscale(\QQ)]$.
 This can be seen by a similar
argument to that found in \cite{Morris3} or \cite{BMc}. 

With this notation in hand, we split 
the quadratic from \eqref{Eq:SFE}
as follows:
\begin{align}
\label{Eq:SFEBreak}
\int_0^1 \norm{\QQ_tS\Ppb_t f}^2\ \dtt
	&\lesssim  \int_0^1 \norm{(\QQ_t - \Pri_t\Av_t) S\Ppb_t f}^2\ \dtt \\
		&\qquad+ \int_0^1 \norm{\Pri_t\Av_tS(\iden - \Ppb_t) f}^2\ \dtt 
		+ \int_0^1 \norm{\Pri_t \Av_t S f}^2\ \dtt.\nonumber
\end{align}
We call the first term on the left of \eqref{Eq:SFEBreak} 
the \emph{principal part}, the second term 
the \emph{cancellation part}
and the last term the \emph{Carleson part}.

From here on, we let the standing assumptions
throughout the remainder of this section be
\ref{Hyp:First}-\ref{Hyp:Last}.

\subsection{The principal part term}

In this subsection, under some additional conditions
on $S$, we bound the principal part. 
The first thing we observe and require
is a \Poincare inequality that is bootstrapped
from the \Poincare inequality for
functions. 
\begin{lemma}[Dyadic \Poincare Lemma]
\label{Lem:DyaPoin}
There exists $C_P > 0$ such that
$$ \int_{\Ball} \modulus{u - u_{\Q}}^2\ d\mu
			\leq C_P r^\kappa \e^{c_E rt}
			(rt)^2 
			\int_{\Ball} \cbrac{\modulus{\conn u}^2 + \modulus{u}^2}\ d\mu$$
for $u \in \SobH{1}(\cV)$, for all balls $\Ball = \Ball(x_{\Q},rt)$ with $r \geq C_1/\delta$
 (with the constant $C_1$ and $\delta$ from Theorem \ref{Thm:Dya:Christ}) 
where $\Q \in \DyQ_t$ with $t \leq \scale$ \Rd (with $\DyQ_t$ and $\scale$ from 
\eqref{Eq:HConsts}). \Bk The constant $C_P$ depends on $\const(\cM,\cV,\Dirb, \Dirp)$. 
\end{lemma}

The proof of this lemma proceeds similar to the  proof of Proposition 5.3 in \cite{BMc}.

\begin{proposition}[Principal part]
\label{Prop:PrinPart}
Let $(\cW,\mh_{\cW},\conn^\cW)$ be
another vector bundle satisfying
$\Ck{0,1}$-GBG and suppose there exists $C_{G,\cW}$ such that
in each GBG frame $\set{e^i}$
for $\cW$, $\modulus{\conn^{\cW}{e^i}(x)} \leq C_{G,\cW}$
for almost-every $x$.
Let $\QQ_t:\Lp{2}(\cW) \to \Lp{2}(\cV)$ be a
family of operators uniformly bounded in $t \in (0, 1]$
satisfying  \eqref{Def:OD} and \eqref{Def:SFE},
and suppose $S:\Lp{2}(\cV) \to \Lp{2}(\cW)$ is a bounded operator
for which 
$$\norm{\conn^\cW S v} \leq C_S \norm{v}_{\SobH{1}}$$
for some $C_S > 0$ and $v \in \SobH{1}(\cV)$.
Then, whenever $u \in \Lp{2}(\cV)$,  
$$
\int_{0}^{\Rd t_1(\QQ) \Bk} \norm{(\QQ_t - \Pri_t\Av_t)S\Ppb_t u}^2\ \dtt 
	\lesssim \norm{A}_{\infty}^2 \norm{u}^2,$$
\Rd where $t_1(\QQ) = \min\set{\hscale(\QQ), C_{\QQ}/(11c_E)}$. \Bk 
The implicit constant depends on $C_{G,W},\ C_S$, $C_{\Delta, \kappa + 3}$ from 
\eqref{Def:OD}, $C_\QQ'$ from \eqref{Def:SFE} and $\const(\cM, \cV, \Dirb, \Dirp)$.
\end{proposition}

\begin{remark}
We allow for an auxiliary vector bundle
$\cW$ in this proposition since, 
in the proof of Theorem \ref{Thm:Main}, we
are required to invoke this with different
choices for $\cW$. We will see later that
the constants $C_S$, $C_{G,\cW}$, $C_{\Delta,\kappa+3}$
and $C_{\QQ}'$ are themselves dependent on $\const(\cM,\cV,\Dirb,\Dirp)$.
\end{remark}

\begin{proof}
The proof proceeds similar to Proposition 8.4 in \cite{BMc},
by replacing their $\QQ_t^B$ with our $\QQ_t$.

Set $v = S\Ppb_t u$.
First, note from \eqref{Eq:DyAv} that $\Av_t v(x) = v_{\Q}(x)$ for $x \in \Q$, and so
$$ \norm{(\QQ_t - \Pri_t\Av_t)v}^2 = \sum_{\Q \in \DyQ_t} \norm{\QQ_t(v - v_{\Q})}_{\Lp{2}(\Q)}^2.$$

Letting $\Ball_{\Q} = \Ball(x_{\Q} ,C_1/\delta t)$, 
$C_j(\Q) = 2^{j+1}\Ball_{\Q}\setminus 2^j\Ball_{Q}$,
and on invoking \eqref{Def:OD} for $\QQ_t$ and for some 
$M > 0$ to  be chosen later, we obtain that
\begin{equation}
\begin{aligned}
\label{Eq:Pri0}
&\int_{\Q} \modulus{\QQ_t(v - v_{\Q})}^2\ d\mu \\
	&\qquad\lesssim \norm{A}_{\infty}^2 
	\cbrac{\sum_{j=0}^\infty \maxx{\frac{\met(\Q,C_j(\Q))}{t}}^{-M}
		\exp \cbrac{ - C_{\QQ} \frac{\met(\Q, C_j(\Q))}{t}}
		\norm{v - v_{\Q}}_{\Lp{2}(C_j(\Q))}}^{2}.
\end{aligned}
\end{equation}

\Rd By (4.1) in \cite{Morris3}, we have \Bk 
$$2^j\frac{C_1}{\delta}t \leq \met(x_{\Q}, C_j(\Q)) \leq \met(\Q, C_j(\Q)) + \diam\Q$$
\Rd and therefore \Bk 
\begin{equation}
\label{Eq:Pri2}
\begin{aligned}
\maxx{\frac{\met(\Q,C_j(\Q))}{t} }^{-M} &\lesssim 2^{-M(j+1)}\quad\text{and}, \\ 
\exp \cbrac{ - C_{\QQ} \frac{\met(\Q, C_j(\Q))}{t}} 
	&\lesssim \exp\cbrac{-\frac{C_{\QQ}C_1}{4\delta}\ 2^{j+1}}
\end{aligned}
\end{equation}
for all $j \geq 0$.
Thus, by Cauchy-Schwartz inequality applied to \eqref{Eq:Pri0}, we
obtain that 
\begin{equation}
\begin{aligned}
\label{Eq:Pri1}
&\int_{\Q} \modulus{\QQ_t(v - v_{\Q})}^2\ d\mu \\
	&\qquad\lesssim \norm{A}_{\infty}^2 
	\sum_{j=0}^\infty 2^{-M(j+1)}
		\exp \cbrac{ - C_{\QQ} \frac{C_1}{2\delta} 2^{j+1}}
		\int_{C_j(\Q)} \modulus{v - v_{\Q}}^2\ d\mu.
\end{aligned}
\end{equation}

On observing that $C_j(\Q) \subset 2^{j+1}\Ball_{\Q}$, 
$v \in \SobH{1}(\cW)$, $S:\SobH{1}(\cV) \to \SobH{1}(\cW)$,
and since $(\cW, \mh_{\cW},\conn^\cW)$ has $\Ck{0,1}$-GBG
with $\modulus{\conn^{\cW}e^i} \leq C_G$ almost-everywhere,
we apply Lemma \ref{Lem:DyaPoin} to obtain
\begin{equation}
\label{Eq:Pri3}
\begin{aligned}
\int_{C_j(\Q)} &\modulus{v - v_{\Q}}^2\ d\mu \\
	&\lesssim \cbrac{\frac{C_1}{\delta}}^{\kappa + 2}  \exp\cbrac{\frac{c_E C_1}{\delta} 2^{j+1} t}
		 2^{2(j+1)} t^2 \int_{2^{j+1}\Ball_{\Q}} (\modulus{\conn^\cW v}^2 + \modulus{v}^2)\ d\mu.
\end{aligned}
\end{equation}

To estimate the term 
$$\int_{2^{j+1}\Ball_{\Q}} (\modulus{\conn^\cW v}^2 + \modulus{v}^2)\ d\mu
	 =  \int \ch{ 2^{j+1}\Ball_{\Q}} (\modulus{\conn^\cW v}^2 + \modulus{v}^2)\ d\mu,$$
we use Lemma 8.3 in \cite{BMc}, which states 
that whenever $r > 0$ and $\set{\Ball_j = \Ball(x_j, r)}$
is a disjoint collection of balls, then for 
every $\eta \geq 1$,
$$\sum_{j} \ch{\eta\Ball_j} \lesssim \eta^\kappa \e^{4c_E\eta\kappa},$$
where the implicit constant depends on \eqref{Def:Eloc}. 
We apply this on setting $r = a_0 t$ and $\eta = 2^{j+1}C_1/(\delta a_0)$
so that $\set{\Ball(x_{\Q}, a_0 t)}$ is disjoint
to obtain the bound
\begin{equation}
\label{Eq:Pri4} 
\ch{ 2^{j+1}\Ball_{\Q}} \lesssim 2^{\kappa(j+1)}\exp\cbrac{\frac{4c_EC_1}{\delta} 2^{j+1}t}.
\end{equation}

On combining estimates \eqref{Eq:Pri2}, \eqref{Eq:Pri3}
and \eqref{Eq:Pri4} with \eqref{Eq:Pri1}, 
\begin{equation}
\begin{aligned}
\sum_{\Q \in \DyQ_t} \int_{\Q} &\modulus{\QQ_t(v - v_{\Q})}^2\ d\mu \\
	&\lesssim \norm{A}_{\infty}^2
	\sum_{j = 0}^\infty 2^{-(M - \kappa - 2)(j + 1)} 
	\exp\cbrac{-\frac{C_1}{2\delta}\cbrac{C_{\QQ} - 10c_E t }\ 2^{j+1}} \\
	&\qquad\qquad\qquad t^2 (\norm{\conn^\cW v}^2 + \norm{v}^2).
\end{aligned}
\end{equation}

This sum converges by choosing $M > \kappa + 2$
and for $t \leq  \frac{C_{\QQ}}{11c_E}.$
Then, on  setting $t_1(\QQ) = \min\set{\hscale(\QQ), C_{\QQ}/(11c_E)}$,
and  recalling that $v = S\Ppb_t u$,
\begin{align*}
\int_{0}^{t_1(\QQ)} &\norm{(\QQ_t - \Pri_t\Av_t)S\Ppb_t u}^2\ \dtt  \\
	&\lesssim \norm{A}_{\infty}^2 \int_{0}^{t_1(\QQ)} t^2 \norm{\conn^\cW S\Ppb_t u}^2\ \dtt  
		+ \norm{A}_{\infty}^2 \int_{0}^{t_1(\QQ)} t^2\norm{S\Ppb_tu}^2\ \dtt \\
	&\lesssim  \norm{A}_{\infty}^2 \int_{0}^{t_1(\QQ)} (t^2 \norm{\conn^\cV \Ppb_t u}^2 + \norm{\Ppb_t u}^2)\ \dtt 
		+ \norm{A}_{\infty}^2 \int_{0}^{t_1(\QQ)} t^2\norm{S\Ppb_t u}^2\ \dtt \\
	&\lesssim \norm{A}_{\infty}^2 \norm{u}^2 
		+ \norm{A}_{\infty}^2 \int_{0}^{t_1(\QQ)} t^2 \norm{\Dirb \Ppb_t u}^2\ \dtt  \\
	&\lesssim \norm{A}_{\infty}^2 \norm{u}^2,
\end{align*}
where the second inequality follows
from the assumption $\norm{\conn^\cW S w}^2 \lesssim \norm{\conn^\cV w}^2 + \norm{w}^2$, 
the third inequality from the boundedness of $S: \Lp{2}(\cV) \to \Lp{2}(\cW)$
and \eqref{Def:DomConst}, and the last inequality 
from the fact that $t\Dirb\Ppb_t = \Qqb_t$ satisfies
quadratic estimates.
\end{proof}

\subsection{The cancellation term}

In this subsection, we estimate the cancellation term.
First, we observe the following. 

\begin{lemma}
\label{Lem:DirCan} 
On each dyadic cube $\Q$, 
and for each $u \in \SobH{1}(\cV)$ with $\spt u \subset \Q$,
we have that
$$ \modulus{\int_{\Q} \Dirb u\ d\mu} \lesssim \mu(Q)^{\frac{1}{2}} \norm{u}.$$
The implicit constant depends on $\const(\cM,\cV,\Dirb,\Dirp)$.
\end{lemma}
\begin{proof}
Let $u = u_i e^i$ inside the GBG frame
associated to $\Q$, and let $\set{v_j}$ be
the GBG frame for $\tanb\cM$.
Then, from \eqref{Def:DirForm},
we write in this frame 
$$\Dirb u = (\alpha^{jk}_l \conn[v_j] u_k + u_i \omega_l^i)\ e^l,$$
and for a bounded Lipschitz $\eta:\cM \to \R$, 
\begin{multline*}
\comm{\eta, \Dirb}u
	= \eta \Dir u - \Dir(\eta u) \\
	= \eta(\alpha^{jk}_l \conn[v_j] u_k + u_i \omega_l^i)\ e^l
		- \alpha^{jk}_l  \conn[v_j](\eta u_k) + \eta u_i \omega_l^i)\ e^l 
	= \alpha^{jk}_l (\conn[v_j]\eta) u_k\ e^l,
\end{multline*}
almost-everywhere inside the GBG frame.
By choosing $\eta$ appropriately, 
i.e., $\conn \eta = v_j$, 
$$\sum_{j,k,l} \modulus{\alpha^{jk}_l}^2 \lesssim \dim(\cV).$$
Moreover, from \ref{Hyp:DirGBG}, we deduce the bound
$$\sum_k \modulus{\omega^i_k}^2 \simeq \modulus{\omega^i_k e^k}^2 = \modulus{\Dirb e^i}^2 \leq c_{\Dirb,\cV}.$$

Before we proceed, we note that
the assumption $\modulus{\conn e_i} \leq C_{G,\cV}$
implies that $\modulus{\conn[\nu_j] \mh_{ij}} \lesssim 1$
almost-everywhere since we assume that $\mh$ and $\conn$ are compatible
almost-everywhere.
The implicit constant here depends only on
of $C_{G,\cV}$ and $C_{\cV}$.

Now, let $\adj{\mh} = \mh_{ij} e^i \tensor e^j$ denote the induced metric for $\adj{\cV}$ from 
$\mh = \mh^{ij} e_i \tensor e_j$, where $e^i(e_j) = \delta_{ij}$.
Now, note that we can write a section $f \in \Lp[loc]{1}(\cV)$
in $\set{e^i}$ as 
$f = f_i e^i  = \mh(f, \mh_{ik}\ e^i)\ e^k$, and 
on choosing $\psi$ to be a Lipschitz function supported
inside the trivialisation for the frame $\set{e_i}$, 
with $\psi \equiv 1$ on $\Q$ we compute
using the fact that $u = 0$ on $\spt \conn \psi$
\begin{multline*}
\int_{\Q} \Dirb u
	= \int_{\Q} \mh(\Dirb u, \psi \mh_{ik}\ e^i)\ e^k 
 	= \int_{\cM} \mh(\Dirb u, \psi \mh_{ik}\ e^i)\ e^k 
	= \int_{\cM} \mh(u, \Dirb(\psi \mh_{ik}\ e^i))\ e^k \\
	= \int_{\Q} \mh(u, \Dirb(\mh_{ik}\ e^i))\ e^k 
	= \int_{\Q} \mh(u, (\alpha^{jm}_l \conn[v_j] \mh_{mk} + \mh_{ik} \omega^i_l)\ e^l)\ e^k.
\end{multline*}
Therefore,
\begin{multline*}
\modulus{\int_{\Q} \Dirb u}
	\lesssim \int_{\Q} \modulus{u}\ \sum_{k,m,l}\modulus{\alpha^{jm}_l \conn[v_j] \mh_{mk}}
		+\int_{\Q} \modulus{u}\ \sum_{k,m} \modulus{(\mh_{ik} \omega^i_m)\ e^m)} \\
	\lesssim \int_{\Q} \modulus{u}
	= \int_{\cM} \ch{\Q} \modulus{u}
	\leq \cbrac{\int_{\cM} \ch{\Q}^2}^{\frac{1}{2}} \cbrac{\int_{\cM} \modulus{u}^2}^{\frac{1}{2}}
	= \mu(\Q)^{\frac{1}{2}} \norm{u},
\end{multline*}
using the proved bounds on $\alpha^{jk}_l$ and $\omega^i_j$
and bounds on $\conn[v_j] \mh_{kl}$ and $\mh_{kl}$ from 
\ref{Hyp:BasGBGL}.
\end{proof}

\begin{lemma}
\label{Lem:D_iCan}
On each dyadic cube $\Q$, each $u \in \SobH{1}(\cV)$
and $v \in \dom(\divv)$ with $\spt v,\ \spt u \subset \Q$,
we have that
$$ \modulus{\int_{\Q} \conn u\ d\mu} \lesssim \mu(\Q)^{\frac{1}{2}} \norm{u}
\quad \text{and}\quad
\modulus{\int_{Q} \divv v\ d\mu} \lesssim \mu(\Q)^{\frac{1}{2}} \norm{v}.
$$
The implicit constants depend on $\const(\cM,\cV,\Dirb,\Dirp)$.
\end{lemma}

This lemma is proved very similar to 
Lemma \ref{Lem:DirCan}.
For a comprehensive outline of the
proof, we consult the reader to 
the proof of Theorem 6.2 in \cite{BMc}. 
Although the metrics in \cite{BMc} are assumed to be smooth, 
it is easy to verify that \Rd our assumption of \Bk $\Ck{0,1}$
regularity of the metric suffices in their proof. 
 
The following is a generalisation of a key estimate 
in \cite{AHLMcT}. 
 
\begin{lemma}[Cancellation lemma]
\label{Lem:Can}
Let $\Upsilon$
be either one of $\Dirb$,  $\Dirp$, $\conn$, or $\divv$.
Then,
$$
\modulus{\fint_{\Q} \Upsilon u\ d\mu}^2 \lesssim 
	\frac{1}{\len(\Q)^\eta} 
		\cbrac{\fint_{\Q} \modulus{u}^2\ d\mu}^{\frac{\eta}{2}}
		\cbrac{\fint_{\Q} \modulus{\Upsilon u}^2}^{1 - \frac{\eta}{2}}
		+ \fint_{\Q} \modulus{u}^2,$$
for all $u \in \dom(\Upsilon)$, $\Q \in \DyQ$, $t \in (0,\scale]$, 
\Rd where $\eta$ is the parameter from Theorem \ref{Thm:Dya:Christ}
and $\len(\Q)$ and $\scale$ are from \eqref{Eq:HConsts}. \Bk 
\end{lemma}

At this point, we note that 
the operator $\Dirb$ satisfies 
the following off-diagonal estimates. 
\begin{lemma}
\label{Lem:OD}
Let $U_t$ be one of 
$\Rrb_t = (\iden + \imath t \Dir)^{-1}$,
$\Ppb_t = (\iden + t^2 \Dir^2)^{-1}$, $\Qqb_t= t\Dir(\iden + t^2 \Dir^2)^{-1}$,
$t\conn \Ppb_t$,  $\Pp_t t\divv$, and $\Qq_t$. 
Then, there exists $C_{U} > 0$ 
such that, for each $M > 0$, there exists a constant $C_{\Delta} > 0$
so that
\begin{equation}
\label{Def:ODN}
\norm{\ch{E} U_t(\ch{F}u)} 
	\lesssim C_{\Delta} \maxx{\frac{\met(E,F)}{t} }^{-M}
		\exp\cbrac{-C_{U} \frac{\met(E,F)}{t}} \norm{\ch{F}u}
\end{equation}
for every Borel set $E,\ F \subset \cM$ and $u \in \Lp{2}(\cV)$.
\end{lemma}

This ``exponential'' version
of off-diagonal estimates 
first appeared as Lemma 5.3 in \cite{CMcM2} 
by Carbonaro, Morris and McIntosh.
The proof here is similar, and
relies on the commutator estimate 
\eqref{Def:CommConst}. 

With the aid of these tools, 
we estimate the cancellation term
in \eqref{Eq:SFEBreak}.
We note that the proof is similar 
to the corresponding result
found in \cite{AKMC}, with the exception 
being the complication arising from the operator
$S$ in the following statement. Thus, we
give sufficiently detailed recollection
of the proof.

\begin{proposition}
\label{Prop:CanPart}
Let $S = \iden$ or $S = \conn(\imath \iden + \Dirb)^{-1}$. Then, 
$$\int_0^{\Rd \hscale(\QQ) \Bk} \norm{\Pri_t\Av_tS(\iden - \Ppb_t) u}^2\ \dtt 
	\lesssim \norm{u}^2.$$
\end{proposition}
\begin{proof}
First we note that $\Av_t^2 = \Av_t$,
and therefore, 
$$\norm{\Pri_t \Av_t S(\iden  - \Ppb_t)u}
	= \norm{\Pri_t \Av_t \Av_t S(\iden - \Ppb_t)u}
	\leq \norm{A}_{\infty} \norm{\Av_tS(\iden  - \Ppb_t)u}.$$
By Schur estimate techniques (see Proposition 5.7 in \cite{AKMC}), 
it suffices to prove that
$$ \norm{\Av_tS(\iden - \Ppb_t)\Qqb_s} \lesssim \min\set{\cbrac{\frac{s}{t}}^\alpha, \cbrac{\frac{t}{s}}^\alpha}$$
for some $\alpha > 0$.

Note the identities
\begin{equation}
\label{Eq:Can:Ident}
(\iden - \Ppb_t)\Qqb_s = \frac{t}{s}\Qqb_t( \iden - \Ppb_s)
\ \text{and}\ 
\Ppb_t\Qqb_s = \frac{s}{t}\Qqb_t\Ppb_s.
\end{equation}

For $t \leq s$, it immediately follows 
from \eqref{Eq:Can:Ident} that
$$\norm{\Av_t S(\iden - \Ppb_t)\Qqb_s} 
	\lesssim \norm{(\iden - \Ppb_t)\Qqb_s}
	\lesssim \frac{t}{s}.$$

For $t > s$, we write
$$\norm{\Av_tS(\iden - \Ppb_t)\Qqb_s}
	\lesssim \norm{\Av_t S \Qqb_s} + \norm{\Ppb_t\Qqb_s} 
	\lesssim  \norm{\Av_t S \Qqb_s} +  \frac{s}{t},$$
where the last inequality follows from \eqref{Eq:Can:Ident}.
Thus, we only need to prove that there
is an $\alpha >0$ such that 
$$\norm{\Av_tS\Qqb_s} \lesssim \cbrac{\frac{s}{t}}^\alpha.$$
	
Fix $u \in \Lp{2}(\cV)$
and note that 
\begin{equation}
\label{Eq:Can2}
\norm{\Av_tS\Qqb_s u}^2 = \sum_{\Q \in \DyQ_t} 
	\norm{ \Av_tS \Qqb_su}^2_{\Lp{2}(\Q)}.
\end{equation}

If $S = \conn(\imath\iden + \Dir)^{-1}$,  we have that 
$$S\Qqb_s = S s \Dirb \Ppb_s = \conn(\imath\iden + \Dirb)^{-1} s \Dirb \Ppb_s 
	=  s\conn\Ppb_s - \imath s \conn(\imath \iden + \Dirb)^{-1} \Ppb_s.$$ 
Also, for $x \in \Q$,
$$\Av_tS\Qqb_s u(x) = \fint_{\Q} s\conn \Ppb_s u\ d\mu 
	- \fint_{\Q} \imath s\conn \Ppb_s (\imath \iden + \Dirb)^{-1} \Ppb_s u\ d\mu ,$$ 
and therefore, 
\begin{equation}
\label{Eq:Can3} 
\begin{aligned}
\norm{\Av_tS\Qqb_s u}_{\Lp{2}(\Q)}^2 
	&= \int_{\Q} \modulus{\fint_{\Q} s\conn \Ppb_s u\ d\mu 
		- \fint_{\Q} \imath s\conn \Ppb_s (\imath \iden + \Dirb)^{-1} u\ d\mu}^2\ d\mu \\
	&\lesssim \mu(\Q) \modulus{\fint_{\Q} s\conn \Ppb_s u\ d\mu}^2 
		+ \mu(\Q) \modulus{\fint_{\Q} s\conn \Ppb_s (\imath \iden + \Dirb)^{-1} u\ d\mu}^2.
\end{aligned}
\end{equation}
 In the case $S = \iden$, we obtain that
$\Av_tS\Qqb_s u = \fint_{\Q} s\Dirb \Ppb_s u\ d\mu$,
so that 
$$\norm{\Av_t S\Qqb_su}_{\Lp{2}(\Q)} \simeq
	\mu(\Q) \modulus{\fint_{\Q} s\Dir \Ppb_s u\ d\mu}^2.$$
 This latter estimate can be handled if we can handle the
former estimate and so it suffices to only consider this case. 
On noting that $t \simeq \len(\Q)$ \Rd from \eqref{Eq:HConsts},\Bk by Lemma \ref{Lem:Can}
\begin{multline*}
\modulus{\fint_{\Q} s\conn \Ppb_s u\ d\mu}^2
	\lesssim \cbrac{\frac{s}{t}}^\eta \frac{1}{\mu(\Q)}
		\norm{\Ppb_s u}^\eta_{\Lp{2}(\Q)}\norm{s\conn \Ppb_s u}^{2 - \eta}_{\Lp{2}(\Q)} \\
	+ t^2 \cbrac{\frac{s}{t}}^2 \frac{1}{\mu(\Q)} \norm{\Ppb_s u}^2_{\Lp{2}(\Q)}.
\end{multline*}

Then, by choosing $p = 2/\eta$ and $q = 2/(2 - \eta)$, and
by H\"older's inequality and 
 the uniform boundedness of $\Ppb_s$, $s \Ppb_s$, and $\Qqb_s = s\Dir \Ppb_s$
on $s \in (0, 1]$, 
\begin{align*} 
\sum_{\Q \in \DyQ_t}
&\norm{\Ppb_s u}^\eta_{\Lp{2}(\Q)} \norm{s\conn \Ppb_s u}^{2 - \eta}_{\Lp{2}(\Q)} \\
	&\lesssim \cbrac{\sum_{\Q \in \DyQ_t} \norm{\Ppb_s u}^2_{\Lp{2}(\Q)}}^{\frac{\eta}{2}}
		\cbrac{\sum_{\Q \in \DyQ_t} \norm{s\conn \Ppb_s u}^{2}_{\Lp{2}(\Q)}}^{\frac{2 - \eta}{2}} \\
	& \lesssim \norm{\Ppb_s u}^\eta (\norm{s \Dir \Ppb_s u}^2 + \norm{s \Ppb_s u}^2)^{\frac{2 - \eta}{2}} 
	\lesssim \norm{u}^2. 
\end{align*}

Thus, for $u$ replaced by $(\imath \iden + \Dirb)^{-1}u$,
we obtain, 
\begin{align*} 
\norm{\Av_tS\Qqb_s u}^2 &\lesssim \cbrac{\frac{s}{t}}^2 \norm{u}^2 
	+ \cbrac{\frac{s}{t}}^\eta \norm{u}^2
	+ \cbrac{\frac{s}{t}}^\eta \norm{(\imath \iden + \Dir)^{-1}u}^2 \\ 
	&\lesssim \cbrac{\frac{s}{t}}^2 \norm{u}^2 + \cbrac{\frac{s}{t}}^\eta \norm{u}^2.
\end{align*}
This finishes the proof.
\end{proof}

\subsection{The Carleson term}
\label{Sec:SFECarl}

We are now left with the task of estimating the last term,
the Carleson term in \eqref{Eq:SFEBreak}. 
Recall that $\nu$ is a \emph{local Carleson measure}
on $\cM \times (0,t']$ (for some $t' \in (0, \scale]$, 
 where $\scale$ is the scale we define in \S\ref{Sec:DyaGrid} ) if
$$
\norm{\nu}_{\Carl} = \sup_{t \in (0,t']} \sup_{\Q \in \DyQ_t} \
	\frac{\nu(\CBox(\Q))}{\mu(\Q)} < \infty,$$ 
where $\CBox(\Q) = \Q \times (0, \len(\Q))$, the 
\emph{Carleson box} over $\Q$.
The norm $\norm{\nu}_{\Carl}$ is the \emph{local Carleson norm}
of $\nu$.

If $\nu$ is a local Carleson measure,
then by Carleson's inequality,
$$\iint_{\cM \times (0,t']} \modulus{\Av_t(x)u(x)}^2\ d\nu(x,t) \lesssim \norm{\nu}_{\Carl} \norm{u}^2$$
for all $u \in \Lp{2}(\cV)$.
This is proved for functions in Theorem 4.2 in \cite{Morris3}
but we note that the proof carries over
\emph{mutatis mutandis} to our setting.

Since $S$ is a bounded operator, we can reduce Carleson's inequality
$$ \int_{0}^1 \norm{\Pri_t \Av_t S u}^2\ \dtt \lesssim \norm{A}_{\infty}^2 \norm{u}^2$$
to showing that
$$ d\nu(x,t) = \modulus{\Pri_t(x)}^2\ \frac{d\mu(x)dt}{t}$$ 
is a local Carleson measure 
with Carleson norm controlled by $\norm{A}_{\infty}^2$.

Fix a cube $\Q \in \DyQ_t$, let $\Ball_{\Q} = \Ball(x_{\Q}, C_1\len(\Q))$,
 Note that since we consider $t' \leq \scale$, 
we have that $3\Ball_{\Q} \subset \Ball(x_{\ancester{\Q}}, C_1\len(\ancester{\Q}))$,
where $\brad$ is the GBG radius. This is one reason why we fix
$\scale \leq \brad/5$ in our analysis. 

For $w \in \C^N$, let $w^c$ 
denote the local constant extension of $w$
\Rd as defined in \eqref{Eq:ConstExt},
and define \Bk $w^{\Q} = \ch{2\Ball_{\Q}}w^c$.
Then, we note that
$$\iint_{\CBox(\Q)} \modulus{\Pri_t(x)}^2\ \frac{d\mu(x)dt}{t}
	\lesssim  \sup_{\modulus{w}_{\C^N} = 1} \int_{0}^{\len(\Q)} \int_{\Q} \modulus{\Pri_t\Av_t w^{\Q}}^2 \frac{d\mu dt}{t},$$
and therefore,
it suffices to prove that 
\begin{equation}
\label{Eq:CarlRed}
\int_{0}^{\len(\Q)} \int_{\Q} \modulus{\Pri_t\Av_t w^{\Q}}^2 \frac{d\mu dt}{t}
	\lesssim \norm{A}_{\infty}^2 \mu(\Q)
\end{equation}
for each $\modulus{w}_{\C^N}  = 1$.

In order to do this, we split up this integral
in the following way: 
\begin{multline}
\int_{0}^{\len(\Q)} \int_{\Q} \modulus{\Pri_t\Av_t w^{\Q}}^2 \frac{d\mu dt}{t}
	\lesssim \\ 
	\int_{0}^{\len(\Q)} \int_{\Q} \modulus{(\Pri_t\Av_t - \QQ_t) w^{\Q}}^2 \frac{d\mu dt}{t}
		+ 
	\int_{0}^{\len(\Q)} \int_{\Q} \modulus{\QQ_t w^{\Q}}^2 \frac{d\mu dt}{t}
\end{multline}

\begin{proposition}
\label{Prop:Carl1}
Let $\QQ_t:\Lp{2}(\cW) \to \Lp{2}(\cV)$ be a
family of operators uniformly bounded in $t \in (0, 1]$
satisfying \eqref{Def:OD}.  
Then for each cube $\Q \in \DyQ_t$,
$$
\int_{0}^{\len(\Q)} \int_{\Q} \modulus{(\Pri_t\Av_t - \QQ_t) w^{\Q}}^2 \frac{d\mu dt}{t}
	\lesssim \norm{A}_{\infty}^2 \mu(\Q),$$
whenever $t \in (0, t_3(\QQ)]$, where $t_3(\QQ) = \min\set{\hscale(\QQ), \frac{C_{\QQ}}{3c_E}}$.
The implicit constant depends on $\const(\cM,\cV,\Dirb,\Dirp)$
and $C_{\Delta,\kappa + 1}$ from \eqref{Def:OD}.
\end{proposition}
\begin{proof}
First, we note that for $x \in \Q$, 
$\Av_t w^{\Q}(x) = w^c(x)$ and hence, $\Pri_t(x) \Av_t w^{\Q}(x) = (\QQ_tw^c)(x)$.
Setting $v = w^{\Q} - w^c$, we have
$\modulus{(\Pri_t \Av_t - \QQ_t)w^{\Q}} = \modulus{\QQ_tv}$
almost-everywhere in $\Q$. 

Letting $C_j(\Q) = 2^{j+1}\Ball_{\Q} \setminus 2^j \Ball_{\Q}$,
and fixing $M > 0$ to be chosen later,
we estimate via \eqref{Def:OD} and  by using Cauchy-Schwartz 
as in \eqref{Eq:Pri1}  
\begin{equation}
\begin{aligned}
\label{Eq:CarlPri1}
\int_{Q} \modulus{ \QQ_t v}^2\ d\mu  
	&= \int_{\Q} \modulus{ \QQ_t\cbrac{\sum_{j=0}^\infty \ch{C_j(\Q)}} v}^2\ d\mu \\
	&\lesssim \norm{A}_{\infty}^2 \sum_{j=0}^\infty \maxx{\frac{\met(\Q, C_j(\Q))}{t} }^{-M}  \\
		&\qquad\qquad\qquad\exp \cbrac{ - 2 C_{\QQ} \frac{\met(\Q, C_j(\Q))}{t}}  
		\int_{\cM} \modulus{\ch{C_j(\Q)}v}^2\ d\mu. 
\end{aligned}
\end{equation}

First, note that $v(x) = w^{\Q}(x) - w^c(x) = \ch{2\Ball_{\Q}}(x) w_i e^i(x) - w_i e^i(x)$
and hence, $\modulus{v(x)} \leq 1$ for almost-every $x$, and thus
$$\int_{\cM} \modulus{\ch{C_j(\Q)}v}^2\ d\mu 
	\leq \mu( C_j(\Q)) \leq \mu( 2^{j+1} \Ball_{\Q}).$$
Moreover, from \eqref{Def:Eloc}  and since $\delta^{j+1} < t \leq \len(\Q) = \delta^j$, 
$$
\mu(2^{j+1}\Ball_{\Q}) \leq  \mu(\Ball(x_{\Q}, 2^{j+1}t C_1/\delta))
	\lesssim 2^{\kappa(j+1)} \exp \cbrac{c_E\frac{C_1}{\delta} 2^{j+1} t} \mu(\Q).$$
Thus, on combining these two inequalities
with \eqref{Eq:Pri2} 
we obtain from \eqref{Eq:CarlPri1} that 
$$
\int_{Q} \modulus{ \QQ_t v}^2\ d\mu  
	\lesssim \norm{A}_{\infty}^2 \frac{t}{\len(\Q)} \mu(\Q)
		\sum_{j=0}^\infty 2^{(\kappa - M)(j+1)} 
		\exp\cbrac{\cbrac{\frac{c_E C_1}{\delta}t -\frac{C_\QQ C_1}{2\delta}} 2^{j+1}}.$$
Thus, by choosing $M > \kappa$, or explicitly, setting $M = \kappa + 1$
and choosing 
$t \leq \frac{C_{\QQ}}{3c_E}$, the right hand sum converges.
That is, 
$$\int_{\Q} \modulus{(\Pri_t\Av_t - \QQ_t) w^{\Q}}^2 \frac{d\mu dt}{t} 
	\lesssim \norm{A}_{\infty}^2 \mu(\Q),$$ 
which completes the proof. 
\end{proof}

From this, we obtain the following.

\begin{proposition}
\label{Prop:CarlPart}
Let $\QQ_t:\Lp{2}(\cW) \to \Lp{2}(\cV)$ be a
family of operators uniformly bounded in $t \in (0, 1]$
satisfying \eqref{Def:SFE} and \eqref{Def:OD}. 
Then, whenever $S \in \bddlf(\Lp{2}(\cV))$, for every $u \in \Lp{2}(\cV)$, we
obtain that 
$$ \int_{0}^{t_3(\QQ)} \norm{\Pri_t \Av_t S u}^2\ \dtt \lesssim \norm{A}_{\infty}^2 \norm{u}^2,$$
where $t_3(\QQ) = \min\set{\hscale(\QQ), \frac{C_{\QQ}}{3c_E}}$
and where the implicit constants depend on 
the bound on $\norm{S}_{\Lp{2}\to\Lp{2}}$,
$C(\QQ)'$ from \eqref{Def:SFE}, $C_{\Delta,\kappa + 1}$ from
\eqref{Def:OD}, and $\const(\cM,\cV,\Dirb,\Dirp)$.
\end{proposition}
\begin{proof}
This follows from Proposition \ref{Prop:Carl1} and the computation:
$$\int_{0}^{\len(\Q)} \int_{\Q} \modulus{\QQ_t  w^{\Q}}^2 \frac{d\mu dt}{t}
	\lesssim  \int_{0}^{1} \norm{\QQ_t w^{\Q}}^2\ \dtt 
	\lesssim \norm{A}_{\infty}^2 \norm{w^{\Q}}^2 \lesssim 
	\norm{A}_{\infty}^2 \mu(\Q)$$
where the second inequality comes from the
\eqref{Def:SFE} assumption on $\QQ_t$ and the third inequality
follows from the fact that $\spt w^{\Q} \subset 2 \Ball_{\Q}$
and $\mu(2\Ball_{\Q}) \lesssim \mu(\Q)$ by \eqref{Def:Eloc}.
\end{proof}

\subsection{Proof of the main theorem}
\label{Sec:SFEUnited}

Finally, we gather the estimates in \S\ref{Sec:Red}
and \S\ref{Sec:SFE} to obtain a proof of the
main theorem.

\begin{proof}[Proof of Theorem \ref{Thm:Main}]
First, we note that,  
by Proposition \ref{Prop:FinalRed}, it suffices to show that
\begin{align*}
&\int_{0}^1 \norm{t\Pp_t\divv A_2 \Ppb_tf}^2\ \dtt
	\lesssim \norm{A}_{\infty}^2 \norm{f}^2,
	\quad\text{and}\quad \\
	&\int_{0}^1 \norm{\Qq_t A_1 \conn (\imath\iden + \Dir)^{-1} \Ppb_t f}^2\ \dtt 
	\lesssim \norm{A}_{\infty}^2 \norm{f}^2.
\end{align*} 

For the first inequality,
we set $\QQ_t = t\Pp_t\divv A_2$, \Rd 
and noting the identity 
$ t \Pp_t \divv = (\Qq_t + \imath t\Pp_t)\adj{(\conn (\imath \iden - \Dirp)^{-1})},$
the  quadratic estimates for $\Qq_t$, 
the boundedness of $\Pp_t$ uniformly in $t$ and the  
the boundedness of $\conn (\imath \iden -\Dirp)^{-1}$, 
we obtain 
$$ \int_{0}^1 \norm{\QQ_t f}^2\ \dtt = \int_{0}^1 \norm{(t \Pp_t \divv) A_2f}^2\ \dtt
	\lesssim \norm{A_2 f}^2 \leq \norm{A}_\infty^2 \norm{f}^2.$$
\Bk 
Moreover, from Lemma \ref{Lem:OD}
with $\Dir' = \divv$ and $u = A_2 f$,
we obtain that $\QQ_t$ satisfies 
\eqref{Def:OD}.
Letting $S = \iden$ Propositions \ref{Prop:PrinPart},
\ref{Prop:CanPart} and \ref{Prop:CarlPart}
yields
$$\int_{0}^{t_1(\QQ)} \norm{t\Pp_t\divv A_2 \Ppb_tf}^2\ \dtt \lesssim \norm{A}_\infty^2 \norm{f}^2$$
for all $f \in \Lp{2}(\cV)$, \Rd where $t_1(\QQ) =\min\set{\hscale(\QQ), C_{\QQ}/(11c_E)}$ (from 
Proposition \ref{Prop:PrinPart}), and since 
$t_1(\QQ) \leq t_3(\QQ)$ where $t_3(\QQ)$ 
is defined in Proposition \ref{Prop:CarlPart}. 
We obtain 
$$\int_{t_1(\QQ)}^1 \norm{t\Pp_t\divv A_2 \Ppb_tf}^2\ \dtt \lesssim \norm{A}_\infty^2 \norm{f}^2$$
from recalling that $\norm{t\Pp_t\divv A_2 \Ppb_tf}  \lesssim \norm{A_2}_\infty \norm{f}$
uniformly in $t$.
\Bk 

Now, set $\QQ_t = \Qq_tA_1$ and $S = \conn(\imath \iden + \Dirb)^{-1}$.
This $\QQ_t$ clearly satisfies \eqref{Def:SFE} and
by Lemma \ref{Lem:OD} it satisfies
\eqref{Def:OD}. Thus, we are able to apply
 Propositions \ref{Prop:CanPart} and \ref{Prop:CarlPart},
but in order to apply Proposition \ref{Prop:PrinPart}, 
it remains to verify that the operator $S$
satisfies $\norm{\conn S u} \lesssim \norm{\conn u} + \norm{u}$
whenever $u \in \SobH{1}(\cV)$, 
To this end, we use the assumptions \ref{Hyp:Dom} and \ref{Hyp:Weitz}
to estimate
\begin{align*}
\norm{\conn S u} 
	&= \norm{\conn \conn(\imath \iden + \Dirb)^{-1}u}
	= \norm{\conn^2 (\imath \iden + \Dirb)^{-1} u} \\
	&\lesssim \norm{\Dir^2 (\imath \iden + \Dirb)^{-1}u} +  \norm{(\imath \iden + \Dirb)^{-1}u} \\
	&\lesssim \norm{\Dir(\imath \iden + \Dirb^{-1})\Dir u} + \norm{u} 
	\lesssim \norm{\Dir u} + \norm{u}
	\lesssim \norm{\conn u} + \norm{u}.
\end{align*}
We obtain 
$$\int_{0}^{ t_1(\QQ) } \norm{\Qq_t A_1 \conn (\imath\iden + \Dir)^{-1} \Ppb_t f}\ \dtt 
	\lesssim \norm{A}_{\infty}^2 \norm{f}^2$$
for $f \in \Lp{2}(\cV)$. 
\Rd  Similar to our previous calculation,  
$$ \int_{ t_1(\QQ) }^1 \norm{\Qq_t A_1 \conn (\imath\iden + \Dir)^{-1} \Ppb_t f}\ \dtt 
	\lesssim \norm{A}_{\infty}^2 \norm{f}^2$$
follows from $\norm{\Qq_t A_1 \conn (\imath\iden + \Dir)^{-1} \Ppb_t f} \lesssim \norm{A_1}_\infty \norm{f}$
uniformly in $t$. 
\Bk 

For the two choices of $\QQ_t$ which we made,
namely $\QQ_t = t\Pp_t\divv A_2$ and $\QQ_t = \Qq_tA_1$, 
the constants
$C_{\Delta,M}$ from \eqref{Def:OD} 
and $C_{\QQ}'$ from \eqref{Def:SFE}
only depend on $\const(\cM,\cV,\Dirb,\Dirp)$
and the constants $C_S$ and $C_{G,W}$ from Proposition \ref{Prop:PrinPart}.
This completes the proof. 
\end{proof}

\newpage
\bibliographystyle{amsplain}
\def\cprime{$'$}
\providecommand{\bysame}{\leavevmode\hbox to3em{\hrulefill}\thinspace}
\providecommand{\MR}{\relax\ifhmode\unskip\space\fi MR }
\providecommand{\MRhref}[2]{%
  \href{http://www.ams.org/mathscinet-getitem?mr=#1}{#2}
}
\providecommand{\href}[2]{#2}

\setlength{\parskip}{0mm}

\end{document}